\documentclass[12pt]{article}

\usepackage{amsmath}
\usepackage{amssymb}
\usepackage{amsthm}
\usepackage{amsfonts}
\usepackage{a4}
\usepackage{graphicx}
\usepackage{ esint }
\usepackage{color}
\usepackage{overpic}
\usepackage{caption}
\usepackage{wrapfig}
\usepackage{float}
\usepackage{graphicx}
\usepackage{here} 

\newtheorem{theorem}{Theorem}[section]
\newtheorem{lemma}[theorem]{Lemma}

\newtheorem{corollary}[theorem]{Corollary}
\newtheorem{definition}[theorem]{Definition}

\newcommand{\Proof}{\par\noindent{\em Proof. }}
\newcommand{\eop}{\nopagebreak\hspace*{\fill}$\Box$\smallskip}

\newtheorem{rem}[theorem]{Remark}

\newcommand{\N}{\mathbb N}

\newcommand{\R}{\mathbb R}
\newcommand{\C}{\mathbb C}
\newcommand{\hn}{{\mathcal H}^{N-1}}
\newcommand{\wto}{\rightharpoonup}

\def\calL{\mathcal{L}}

\def\eps{\varepsilon}
\def\diam{{\rm diam}}

\def\dist{\operatorname{dist}}
\def\Xint#1{\mathchoice
{\XXint\displaystyle\textstyle{#1}}%
{\XXint\textstyle\scriptstyle{#1}}%
{\XXint\scriptstyle\scriptscriptstyle{#1}}%
{\XXint\scriptscriptstyle\scriptscriptstyle{#1}}%
\!\int}
\def\XXint#1#2#3{{\setbox0=\hbox{$#1{#2#3}{\int}$ }
\vcenter{\hbox{$#2#3$ }}\kern-.6\wd0}}

\def\dashint{\Xint-}

\usepackage{color}

\begin{document}

\begin{center}
\begin{Large}
{\bf Quasistatic crack growth in 2d-linearized elasticity}
\end{Large}
\end{center}

\begin{center}
\begin{large}
Manuel Friedrich\footnote{Universit{\"a}t Wien, Fakult\"at f\"ur Mathematik, 
Oskar-Morgenstern-Platz 1, 1090 Wien, Austria. {\tt manuel.friedrich@univie.ac.at}}
and 
{Francesco Solombrino\footnote{Zentrum Mathematik, Technische Universit{\"a}t M{\"u}nchen, 
Boltzmannstr.\ 3, 85747 Garching, Germany. {\tt francesco.solombrino@ma.tum.de}} \footnote{New: Dipartimento di Matematica e Applicazioni, Universita di Napoli “Federico II”, `
Via Cintia, 80126 Napoli, Italy. {\tt francesco.solombrino@unina.it} }}
\\
\end{large}
\end{center}

\bigskip

\begin{abstract}
In this paper we prove a two-dimensional existence result for a variational model of crack growth for brittle materials in the realm of linearized elasticity. Starting with a time-discretized version of the evolution driven by a prescribed boundary load, we derive a time-continuous quasistatic crack growth in the framework of generalized special functions of bounded deformation ($GSBD$). As the time-discretization step tends to zero, the major difficulty lies in showing the stability of the static equilibrium condition, which is achieved by means of a Jump Transfer Lemma generalizing the result of \cite{Francfort-Larsen:2003} to the $GSBD$ setting. Moreover, we present a general compactness theorem for this framework and prove existence of the evolution without imposing a-priori bounds on the displacements or applied body forces. 
\end{abstract}
\bigskip

\begin{small}

\end{small}

\tableofcontents

\section{Introduction}\label{sec:intro}
The mathematical foundations of the theory of brittle fracture were laid by the work of A. Griffith \cite{Griffith:20} in the 1920s. The fundamental idea is that the formation of cracks may be seen as the result of the competition between the elastic bulk energy of the body and the work needed to produce a new crack. This latter is modeled as a surface energy, which, in its  simplest form, is proportional to the surface measure of the crack via a material constant, called the {\it toughness} of the material. The rigorous mathematical formulation of the problem, introduced in \cite{frma98}, requires that the function $t\to (u(t), \Gamma(t))$, associating to each time $t$ a deformation $u(t)$ of the reference configuration and a crack set $\Gamma(t)$, is a quasistatic evolution satisfying the following conditions:
\begin{itemize}
\item (a) static equilibrium: for every $t$ the pair $(u(t), \Gamma(t))$  minimizes the energy at time $t$ among all admissible competitors;
\item (b) irreversibility: $\Gamma(s)$ is contained in $\Gamma(t)$ for $0 \le s < t$;
\item (c) nondissipativity: the derivative of the internal energy equals the power of the applied
forces.
\end{itemize}
Remarkable features of this approach are that it is able to show crack initation, as well as a discontinuous evolution of the crack path, which {\it needs not to be a priori prescribed}. On the other hand, establishing a rigorous mathematical framework for the existence of a continuous-time evolution has proved to be quite a hard task.

\subsection*{Existence results for continuous-time evolution} 
The first breakthrough results in this direction are the ones in \cite{DM-Toa} and \cite{Chambolle:2003} tackling in a {\it planar setting} the case of anti-plane shear and linearized elasticity, respectively.   The evolution is driven by a given prescribed load $g(t)$ on a Dirichlet part $\partial_D \Omega$ of the boundary of the reference configuration $\Omega$. Namely, in the case considered in \cite{Chambolle:2003}, the energy associated to a displacement $u$ and a crack $\Gamma$ is given by
\begin{equation}\label{eq: model}
\mathcal E(u,\Gamma):=\int_{\Omega \setminus \Gamma} Q(e(u))\,\mathrm{d}x + \mathcal{H}^1(\Gamma)\,,
\end{equation}
where $Q$ is a quadratic form acting on the symmetrized gradient $e(u)$. At each time $t$, the deformation $u(t)$, which fulfills the boundary condition $u(t)=g(t)$ on $\partial_D \Omega\setminus \Gamma(t)$ has to satisfy the minimality property
\begin{equation}\label{eq: stability}
\mathcal E(u(t),\Gamma(t))\le \mathcal E(v,\Gamma)
\end{equation}
for all $\Gamma \supset \bigcup_{s<t}\Gamma(s)$ and all $v \in LD(\Omega \setminus \Gamma)$ with $v=g(t)$ on $\partial_D \Omega\setminus \Gamma$. Here $LD$ is the space of displacements with square-integrable strains. The existence of an evolution is proved by following the natural idea, in the context of quasistatic brittle fracture, of starting with a time-discretized evolution, and then letting the time-step go to $0$. Namely, for a given time step $\delta>0$ and $n\in \N$, the pair $(u(n\delta), \Gamma(n\delta))$ is inductively defined as a solution for the problem
\begin{equation}\label{eq: strongformulation}
\arg\min \left\{\int_{\Omega \setminus \Gamma} Q(e(u))\,\mathrm{d}x + \mathcal{H}^1(\Gamma)\right\}
\end{equation}
among all cracks $\Gamma\supset\Gamma((n-1)\delta)$ and displacements $u \in LD(\Omega \setminus \Gamma)$ with  $u=g(n\delta)$ on $\partial_D \Omega\setminus \Gamma$. Notice that the existence for the above minimum problems can be proved under the {\it additional restriction} that the admissible cracks have {\it at most a fixed number}  of connected components. Indeed, in this case the direct method proves succesful: crack sets are compact and lower semicontinuous with respect to the Hausdorff topology of sets via Go{\l}ab's Theorem (see \cite{Golab:28}), while compactness of the displacements is recovered via the Poincar\'e-Korn inequality, upon noticing that the energy stays invariant under subtraction of rigid movements in the connected components of $\Omega\setminus \Gamma$ whose boundary has no intersection with $\partial_D \Omega\setminus \Gamma$.

The aforementioned important restriction plays furthermore a fundamental role in overcoming a stability issue, which arises when taking the limit for a time step $\delta$ going to $0$. Indeed, if this hypothesis is dropped, the convergence in the Hausdorff metric of the approximating cracks $\Gamma_\delta(t)$ (obtained as piecewise constant interpolations of $\Gamma(n\delta)$, $n \in \N$) to a set $\Gamma(t)$ does not imply that piecewise constant interpolations of the time-discretized displacements $u_\delta(t)$ converge to a solution of the minimum problem \eqref{eq: strongformulation}. This issue, which is due to a {\it Neumann-sieve-type} phenomenon (see \cite{Murat}), can be overcome in a planar setting imposing an a-priori bound on the connected components of the cracks and using some results from the analysis of Neumann problems in varying domains, contained in \cite{Buc-Var:2000, Chamb-Dov:1997}.

To avoid this restriction, a different and more powerful approach has been proposed in \cite{Francfort-Larsen:2003}, and successfully applied to the case of {\it anti-plane shear} in arbitrary dimension $N$. In this case, the reference configuration is an infinite cylinder $\Omega \times \R$ with $\Omega \subset \R^N$ open and bounded,  and admissible displacements are of the form $(0,\dots,0, u(x))$ where $x$ varies in $\Omega$ and the only nonzero component $u(x)$ is scalar-valued. In this case, the linear elastic energy reduces to the Dirichlet energy
$
\int_{\Omega \setminus \Gamma} |\nabla u|^2\,\mathrm{d}x
$
and the incremental minimum problems become very similar to the strong formulation of the Mumford-Shah functional in image segmentation proposed in \cite{Mum-Sha:1989}. Inspired by De Giorgi's weak formulation in the space of special functions of bounded variation $SBV(\Omega)$ (see \cite{dg91, DeGiorgi-Carriero-Leaci:1989}), the authors model crack sets as (union of) jump sets of admissible displacements. The minimum problems to be solved at every time step essentially reduce (up to some modifications in order to allow for cracks running alongside the boundary) to
\begin{equation}\label{eq: mum-shah}
\arg\min \left\{\int_{\Omega} |\nabla u|^2\,\mathrm{d}x + \mathcal{H}^{N-1}(J_u\setminus \Gamma((n-1)\delta))\right\}\,,
\end{equation}
with $J_u$ denoting the jump set of $u$, among all displacement satisfying $u=g(n\delta)$ on $\partial_D \Omega$. Provided one assumes an $L^\infty$ bound on the boundary datum, the maximum principle and Ambrosio's compactness theorem in $SBV$ (see \cite{Ambrosio:89}) ensure well-posedness for the above problem. If $u$ is a solution thereof, the crack set is then updated by setting $\Gamma(n\delta):=J_u\cup \Gamma((n-1)\delta)$.  

A key tool introduced in the paper \cite{Francfort-Larsen:2003} in order to deal with the above mentioned stability issues, when the time step $\delta$ tends to $0$, is the so-called {\it Jump Transfer Lemma}. It allows to transer most of the jump set of any function in $SBV$ that lies inside of the jump set of a function $u$ onto that of $u_n$, if $u_n$ is a  sequence in $SBV$ weakly converging to $u$. As a consequence of this lemma, the authors are able to recover a weak form of \eqref{eq: stability} in the limit. 
The existence result has been later generalized to finite hyperelastic energies and vector-valued deformations in \cite{DFT}, whereas the existence of a weak quasistatic evolution for the fully linear elastic model \eqref{eq: model} has remained an open issue, due to at least two major difficulties.

\bigskip

\subsection*{Challenges for linear elastic models}

As a first point, even in the static setting the existence of minimizers for the weak formulation is not clear. A natural attempt of generalising \eqref{eq: mum-shah} consists indeed in considering problems of the type
\begin{equation}\label{eq: linearly-elastic}
\arg\min \left\{\int_{\Omega} Q(e(u))\,\mathrm{d}x + \mathcal{H}^{N-1}(J_u\setminus \Gamma)\right\}\,,
\end{equation}
under some prescribed boundary condition $g$, in the space $SBD$ of special functions of bounded deformation (see \cite{ACD, BCD}), for which a symmetrized gradient and an $\mathcal{H}^{N-1}$-rectifiable jump set are well defined. However, weak sequential compactness in $SBD$ requires (see \cite[Theorem 1.1]{BCD}) a uniform bound on the $L^\infty$ norm of the sequence, similarly to the $SBV$-case, which in this setting is not guaranteed along a minimizing sequence, due to the lack of a maximum principle. The addition of lower order terms, related for instance to the action of bulk forces, can at least provide some uniform bound on the $L^p$ norm of the minimizing sequences, so that, mimicking a succesful approach to similar problems in spaces of functions of bounded variation, one can recover an existence result in the space $GSBD$ of {\it generalized special functions of bounded deformation}. A correct definition of this space and the investigation of the related compactness and lower 
semicontinuity properties have proved to be a quite delicate issue, which has been  overcome only recently in the paper \cite{DM}. On the other hand, it would be highly desirable to have an existence result also for the model \eqref{eq: linearly-elastic} without the addition of lower order terms. This requires a suitable  Korn-type inequality in $GSBD$ to be  available, allowing in some sense to reproduce the steps of the existence proof for \eqref{eq: strongformulation} in a weak setting.

The other major issue to be faced in order to give an existence proof of a quasistatic evolution with values in $GSBD$ is the generalisation of the Jump Transfer Lemma to this setting. Actually, the proof strategy devised in \cite{Francfort-Larsen:2003} cannot be straightforwardly reproduced in this context. Indeed, there the jump set $J_u$ is written as a countable union of pairwise intersections of level sets of $u$. The parts of the corresponding level sets for $u_n$ lying outside $J_{u_n}$ are then shown to have small length. With this, one can transer onto pieces of these sets the jump $J_{\phi}\cap J_{u}$ for a given competitor $\phi$. In this procedure, the \emph{coarea formula} and the equiintegrability of $\nabla u_n$ play a crucial role.  In the framework of linearized elasticity, however, only an a-priori control on the symmetrized gradient is available. Again, being able to estimate gradients in terms of their symmetrized part via a Korn-type inequality would remove 
 parts of these obstacles and be a good starting point for proving an analog of the lemma in the $GSBD$ setting.

\bigskip
\subsection*{The present paper}

This preliminary discussion leads us to the purpose of the present paper. Our goal is to provide an existence result, in dimension $N=2$, for quasistatic crack growth in the sense of Griffith in a linearly elastic material.  In Theorem \ref{th: main} we show the existence of a pair $(u(t), \Gamma(t))$, with $u(t) \in  GSBD^2(\Omega)$, $J_{u(t)}\subset \Gamma(t)$, and $\Gamma(t)$ nondecreasing in time, such that $u(t)$ minimizes
\begin{equation*}
 \int_{\Omega} Q(e(v))\,\mathrm{d}x +\mathcal{H}^{1}(J_v\setminus \Gamma(t))
\end{equation*}
among all $v\in GSBD^{2}(\Omega)$ satisfying the prescribed time-dependent Dirichlet condition  $g(t)$, 
and the total energy satisfies the energy-dissipation balance
\begin{equation*}
\mathcal{E}(u(t), \Gamma(t))=\mathcal{E}(u(0), \Gamma(0))+\int_0^t\int_{\Omega}\C e(u(s,x))\cdot e(\dot g(s,x))\,\mathrm{d}s\,\mathrm{d}x\,.
\end{equation*}
In the above equality $\C$ is the elastic tensor generating the quadratic form $Q$, so that the integral term can be interpreted as the virtual work of the applied boundary load. We also mention that, as it is typical of variational problems in spaces of functions of bounded deformation, the boundary condition has to be understood in a relaxed sense (see Section \ref{sec: general} for details).

A starting point for our proof strategy is the use of a piecewise Korn inequality for $GSBD$ functions, proved in the planar setting in \cite{Friedrich:15-4}, extending other recent results in the literature (\cite[Theorem 1.1]{Friedrich:15-2} and \cite[Theorem 1.2]{Conti-Focardi-Iurlano:15}). For every $1\le p<2$ it allows to control the $L^p$-norm of a displacement and  its gradient  in terms of the square norm of the symmetrized gradient, provided a suitable piecewise infinitesimal rigid motion is subtracted. With this construction the jump set is enlarged, but still controlled by the length of the original jump set.

A major ingredient is then a sharp version of the piecewise Korn inequality proved in Theorem \ref{th: kornpoin-sharp}. We show that the jump set  can even only be enlarged by a small length at the prize of having only an $L^1$-control on the gradient. This control, however, involves constants which behave well with respect to scaling and particularly are small on small squares (see Remark \ref{rem:square}).

Equipped with this result, we can prove  Theorem \ref{th: approx}, where, up to an arbitrarily small error $\theta$, the jump set of a weakly compact sequence $(u_n)_n$ in  $GSBD$   is shown to  coincide with the one of a sequence $(v_n)_n$ of $SBV$ functions, still $L^1$-converging to $u$ up to some small exceptional sets. Furthermore, the $L^1$-norm of $\nabla v_n$  is uniformly small in a tubular neighborhood of the jump set $J_u$. Notice that the construction of $v_n$ is quite involved and depends on the given covering of $J_u$ (see Section \ref{sec:jump} for details).

 This allows  to prove a Jump Transfer Lemma also in this setting (Theorem \ref{th: JTransf}), adapting the arguments of \cite[Theorem 2.1]{Francfort-Larsen:2003}. The reflection procedure that the authors use there in order to define the sequence  $(\phi_n)_n$ corresponding to the competitor $\phi$, which is not compatible with a control only on the symmetrized gradient $e(\phi)$, is here replaced by a suitable generalization introduced in \cite{Nit} and adjusted to our purposes in Lemma \ref{lemma: nitsche}.

The existence proof for the minimum problem \eqref{eq: linearly-elastic} requires an additional step, namely a version of the sharp piecewise Korn inequality proved in Theorem \ref{th: kornpoin-sharp} which also takes into account the relaxed boundary conditions. This is proved in Theorem \ref{th: korn-boundary}. With this, we can derive a  general compactness result for minimizing sequences of the energy \eqref{eq: linearly-elastic} drawing some ideas from \cite{Friedrich:15-2}: while typically sequences are not compact, it is always possible to pass to modifications  by subtracting suitable piecewise infinitesimal rigid motions (which do not change the elastic part of the energy) at the expense of arbitrarily small additional fracture  energy.  This allows us to construct a minimizing sequence $(y_n)_n$ which satisfies the uniform bound
$$\int_{\Omega} \psi(|y_n|)\mathrm{d}x + \int_{\Omega} |e(y_n)|^2\mathrm{d}x + {\cal H}^{1}(J_{y_n}) \le M $$
for an increasing, continuous functions $\psi:[0,\infty) \to [0,\infty)$ with $\psi(s) \to \infty$ as $s \to \infty$.  This bound, in general weaker than any $L^p$-bound, is enough to apply the compactness result in \cite[Theorem 11.3]{DM} deducing the existence of a minimizer (see Theorem \ref{th: comp} and Theorem \ref{th: existence} below). An additional delicate point of the proof is showing that the function $\psi$ is only depending on  the reference configuration $\Omega$ and the $H^1$ norm of the boundary displacement $g(t)$, so that, under the usual regularity assumptions on the boundary load, it is independent from the time $t$ along an evolution. This is crucial in the proof of Theorem \ref{th: extension} where the global stability property is derived.

Once this two major hurdles have been fixed, the by now well-known machinery succesfully exploited in \cite{Francfort-Larsen:2003} and in \cite{DFT}, in the linear antiplane and in the finite elastic context, respectively, can be adapted to our setting with minor modifications, which we however detail to some extent in Section \ref{se: exist result}. This leads to the proof of our main result stated in  Theorem \ref{th: main}.

As already mentioned, we establish the result only in two dimensions as we make a heavy use of the piecewise Korn inequality of \cite{Friedrich:15-4} which has been only derived in a planar setting due to technical difficulties, concerning the topological structure of crack geometries in higher dimensions. Additionally, also its generalisation to the sharp version (Theorem \ref{th: kornpoin-sharp}) and the case of prescribed boundary conditions (Theorem \ref{th: korn-boundary}) makes use of  estimates holding in a planar setting (see  Lemma \ref{lemma: rigid motion}, Lemma \ref{lemma: interface}, and Lemma \ref{lemma: maggi}). Without these restrictions, the methods we use actually hold in any dimension. We therefore believe that our results can be extended to the $N$-dimensional case and  that the proof provides the principal techniques being necessary to establish the result in arbitrary space dimension.

\section{Preliminaries}\label{sec:pre}

In this section we introduce basic definitions and the function spaces which we will use in the paper. Moreover, we recall a piecewise Korn inequality for $GSBD$ functions proved in \cite{Friedrich:15-4}. 

\subsection{Basic definitions}\label{sec:pre1}
For a bounded, measurable set $E \subset\R^N$ we define 
$$\diam(E) = {\rm ess }\sup\lbrace \vert x-y \vert: x,y \in E \rbrace.$$
The above definition is independent of the particular Lebesgue representative. If $U$ is an open set in $\R^N$, and $u: U \to \R^m$ is a $\mathcal L^N$-measurable function, $u$ is said to have an approximate limit $a \in \R^m$ at a point $x \in U$ if and only if
$$
\lim_{\varrho \to 0^+}\frac{\mathcal L^N\left(\{|u-a|\ge \varepsilon\}\cap B_\varrho(x)\right)}{\varrho^N}=0\hbox{ for every }\varepsilon >0\,,
$$
where $B_\varrho(x)$ is the ball of radius $\varrho$ centered at $x$. In this case, one writes $\mathrm{ap }\lim_{y\to x}u(y)=a$. The approximate jump set $J_u$ is defined as the set of points $x \in U$ such that there exist $a\neq b \in \R^m$ and $\nu \in S^{N-1}:=\{\xi \in \R^{N}: |\xi|=1\}$ with
$$
\mathrm{ap}\lim_{\substack{y\to x\\(y-x)\cdot \nu >0}}u(y)=a\,,\quad \mathrm{ap}\lim_{\substack{y\to x\\(y-x)\cdot \nu <0}}u(y)=b\,.
$$
The triplet $(a, b, \nu)$ is uniquely determined up to a permutation of $(a, b)$ and a change of sign of $\nu$, and is denoted by $(u^+ (x), u ^-(x), \nu_u (x))$. The jump of  $u$ is the function $[u]:J_u \to \R^m$ defined by $[u](x) := u^+ (x)-u^- (x)$ for every $x \in J_u$. It follows from Lusin's Theorem that $u$ has $u(x)$ as approximate limit at $\mathcal L^N$-a.e.\ $x \in U$, in which case one says that $u$ is approximately continuous at $x$, and therefore $J_u$ is a $\mathcal L^N$-null set.
Given $x\in U$ such that $u$ is approximately continuous at $x$, an $m\times N$ matrix $\nabla u(x)$ is said to be an approximate gradient of $u$ at $x$ if and only if
\begin{equation*}
\mathrm{ap }\lim_{y\to x}\frac{u(y)- u(x)-\nabla u(x)(y-x)}{|y-x|}=0\,.
\end{equation*}
We say that $u$ has an approximate symmetric differential $e(u)(x)\in \mathbb R_{\mathrm{sym}}^{N\times N} $ at $x$ if
\begin{equation*}
\mathrm{ap }\lim_{y\to x}\frac{(u(y)-u(x)-e(u)(x)(y-x))\cdot(y-x)}{|y-x|^2}=0\,.
\end{equation*}

We will make use of the following measure-theoretical result from \cite{Friedrich:15-2}. A short proof is reported for the reader's convenience.
\begin{lemma}\label{rig-lemma: concave function2}
Let $F\subset \R^N$ with $\mathcal L^N(F)<+\infty$ and let $(s_n)_n$, $(t_n)_n$ be nonnegative, monotone sequences with $s_n \to \infty$ and $t_n \to 0$ as $n \to \infty$. Then there is a nonnegative, increasing, concave function $\psi$ with 
\begin{equation}\label{coerc}
\lim_{s\to+\infty}\psi(s)=+\infty
\end{equation} 
only depending on $F$, $(s_n)_n$, $(t_n)_n$ such that for every sequence $(u_n)_n \subset L^1(F;\R^m)$ with
$$\Vert u_n\Vert_{L^1(F)} \le s_n, \ \ \ \ \mathcal L^N\left(\bigcup\nolimits_{m \ge n} \lbrace |u_m - u_n| > 1 \rbrace\right) \le t_n$$  
for all $n \in \N$ there is a not relabeled subsequence such that $$\sup_{n \ge 1} \int_{F}\psi(|u_n|)\mathrm{d}x\le 1\,.$$ 
\end{lemma}

\Proof  Let $A_n = \bigcap_{m \ge n} \lbrace |u_n - u_m|\le 1 \rbrace$ and set $B_1 = A_1$ as well as $B_n = A_n \setminus \bigcup^{n-1}_{m=1} B_{m}$ for all $n \in \N$. The sets $(B_n)_n$ are pairwise disjoint with $\sum_n \mathcal L^N(B_n) = \mathcal L^N(F)$. We choose $0=n_1 < n_2 < \ldots$ such that $\sum_{1\le n \le n_{i}} \frac{\mathcal L^N(B_n)}{\mathcal L^N(F)} \ge 1 - 4^{-i}$. We let $B^i = \bigcup^{n_{i+1}}_{n=n_i+1} B_n$ and observe $\mathcal L^N(B^i)\le  4^{-i}\mathcal L^N(F) $. 

From now on we consider the subsequence $(n_i)_{i \in \N}$ and observe that the choice of $(n_i)_{i \in \N}$ only depends on the sequence $(t_n)_n$.  Choose $ E^i \supset B^i$ such that $\mathcal L^N( E^i) = 4^{-i} \mathcal L^N(F)$. Let $b_i = \frac{s_{n_{i+1}}}{\mathcal L^N(E^i)} + 2=  4^{i}\frac{s_{n_{i+1}}}{\mathcal L^N(F)}+ 2$ for $i\in \N$  and note that $(b_i)_i$ is increasing with $b_i \to \infty$.  By an elementary construction (see \cite[Lemma 4.1]{Friedrich:15-2}) we find an increasing concave function $\psi:[0,\infty) \to [0,\infty)$ with $\lim_{s \to \infty} \psi(s) = \infty$ and $\psi(b_i) \le \frac{ 2^{i}}{\mathcal L^N(F)}$ for all $i \in \N$.    

For $\hat{B}^{ i} := \Omega \setminus \bigcup^{n_i}_{n=1} B_n$ we have $\mathcal L^N(\hat{B}^{ i}) \le 4^{-i}\mathcal L^N(F)$ and choose $\hat{E}^{ i} \supset \hat{B}^i$ with $\mathcal L^N(\hat{E}^{ i}) = 4^{-i} \mathcal L^N(F)$. We then obtain  $\frac{s_{n_i}}{\mathcal L^N(\hat{E}^{ i})}=  4^{i}\frac{s_{n_{i}}}{\mathcal L^N(F)} \le b_{i}$.  Now let $l = n_i$. Using Jensen's inequality, the definition of the sets $B^i$, $\Vert u_l \Vert_{1} \le s_l$  and the monotonicity of $\psi$ we compute
\begin{align*}
\int_{F} \psi(|u_l|) &=\sum\nolimits_{1 \le j \le i-1} \int_{B^j} \psi(|u_l|)\mathrm{d}x + \int_{\hat{B}^{ i}} \psi(|u_l|)\mathrm{d}x   \\
& \le \sum\nolimits_{1 \le j \le i-1}  \int_{B^j} \psi(|u_{n_{j+1}}|+ 2)\mathrm{d}x + \int_{\hat{B}^{i}} \psi(|u_l|)\mathrm{d}x \\
& \le \sum\nolimits_{1 \le j \le i-1}  \mathcal L^N(E^j) \psi\Big(\dashint_{E^j} |u_{n_{j+1}}|+2\Big) + \mathcal L^N(\hat{E}^{ i}) \psi\Big(\dashint_{\hat{E}^{ i}} |u_l|\Big) \\
& \le \sum\nolimits_{1 \le j \le i-1}  \mathcal L^N(F)4^{-j}  \tfrac{2^{j}}{\mathcal L^N(F)} + \mathcal L^N(F)4^{-i}  \tfrac{2^{i}}{\mathcal L^N(F)} \le \sum\nolimits_{j \in \N} 2^{-j} = 1.
\end{align*}   
As the estimate is independent of $l \in (n_i)_i$, this yields $\int_F \psi(|u_l|)\mathrm{d}x \le 1$ uniformly in $l$, as desired. \eop

\begin{rem}\label{meas-conv}
Let $u$ be a measurable function and $(u_n)_n \subset L^1(F;\R^m)$ a sequence such that $u_n \to u$ in measure. Then it follows from the previous lemma that there exist a subsequence $(u_{n_k})_k$ of $(u_n)_n$ and a nonnegative, increasing, concave function $\psi$ satisfying \eqref{coerc}, such that
$$
\sup_{k \ge 1} \int_{F}\psi(|u_{n_k}|)\mathrm{d}x\le 1\,.
$$
Indeed, by definition of convergence in measure we can always find a subsequence $(u_{n_k})_k$ with the property that, setting
$
E_k:=\{|u_{n_k}-u|\ge\frac{1}{2^k}\}\,
$,
one has $\mathcal L^N(E_k)\le \frac{1}{2^k}$. Now, for all $k\in \N$ we have by the triangle inequality that
$$
\bigcup_{m\ge k}\lbrace |u_{n_m} - u_{n_k}| \ge 1 \rbrace \subseteq \bigcup_{m\ge k} E_m
$$
and therefore
$$\mathcal L^N\left(\bigcup\nolimits_{m\ge k}\lbrace |u_{n_m} - u_{n_k}| \ge 1 \rbrace\right)\le\sum\nolimits_{m=k}^{+\infty}\frac{1}{2^m}=\frac{1}{2^{k-1}}\,.
$$
Now it suffices to apply the previous lemma with $s_k:=\max\{\max_{1\le i \le k}\Vert u_{n_i}\Vert_{L^1(F)}, k\}$ and $t_k:=\frac{1}{2^{k-1}}$.
\end{rem}

In a two-dimensional setting, we will often make use of the following simple lemma.

\begin{lemma}\label{lemma: rigid motion}
Let  $A \in \R^{2 \times 2}_{\rm skew}$, $b \in \R^2$.

(a) There is a universal constant $c>0$ independent of $A$ and $b$ such that for all measurable  $E \subset \R^2$ we have $(\mathcal{L}^2(E))^{\frac{1}{2}} |A| \le  c\Vert A\cdot +b \Vert_{L^\infty(E;\R^2)}$.

(b) Let $F$ be a bounded measurable subset of $\R^2$, $\delta>0$ and let a continuous nondecreasing function $\psi \colon \R^+\to \R^+$ satisfying \eqref{coerc} be given. Consider a measurable subset $E\subset F$ with $\calL^2(E)\ge \delta$. Then, if
$$
M\ge \int_E \psi(|Ax+b|)\,\mathrm{d}x\,,
$$
there exists a constant $C$ only depending on $M$, $\delta$, $\psi$, and $F$ such that
\begin{equation}\label{stimasemplice}
|A|+|b|\le C\,.
\end{equation}
If $\psi(s) =s^p$ for $p \in [1,\infty)$ we get $|A|+|b|\le \tilde{C} M^{\frac{1}{p}}$ for a constant $\tilde{C}$ only depending on  $\delta$, $p$ and $F$.

\end{lemma}

\begin{proof}
(a) It suffices to consider the case $A \neq 0$. If $A\neq 0$, the assumption $A \in \R^{2 \times 2}_{\rm skew}$ implies that $A$ is invertible and that $|Ay|=\frac{\sqrt2}{2}|A| |y|$ for all $y \in \R^{2}$. We notice that for all $z \in \R^2$ there exists $x \in E$ with $|x-z| \ge \frac{1}{4}\diam(E)$. For the special choice  $z= - A^{-1}b$ we obtain $|A\,x+b| = |A\,(x - z)| = \frac{\sqrt2}{2}|A| |x-z| \ge \frac{\sqrt2}{8}|A| \diam(E)$ which implies the result due to the isodiametric inequality. 

(b) If $A=0$, we have
$$
\frac{M}{\delta}\ge \psi(|b|)
$$
and the result follows from \eqref{coerc}. If $A\neq 0$, we set $z:= - A^{-1}b$ and $\lambda:=\sqrt{\frac{\delta}{2\pi}}$. Then we have that $\calL^2(E \setminus B_\lambda(z))\ge \frac{\delta}2$. Since $\psi$ is nonnegative and increasing, we get
\begin{align*}
M&\ge \int_E \psi\left(\frac{\sqrt2}{2}|A| |x-z|\right)\,\mathrm{d}x\\
& \ge\int_{E\setminus B(z,\lambda)} \psi\left(\frac{\sqrt2}{2}|A| |x-z|\right)\,\mathrm{d}x
\ge \frac{\delta}2 \psi\left(\frac{\sqrt2}{2}|A| \lambda\right)\,.
\end{align*}
 By this and \eqref{coerc} it exists a constant $\hat C$ only depending on $M$, $\delta$, and $\psi$ such that
\begin{align}\label{eq: A}
|A|\le \hat{C}.
\end{align}
It also follows that $|Ax|\le C'$ for all $x\in F$, where $C'$ is allowed to depend on $F$, too. If now $|b|\le C'$ we are done, otherwise it holds $|Ax+b|\ge |b|-C'>0$ for all $x\in F$. The monotonicty of $\psi$ yields then 
$$
\frac{M}{\delta}\ge \psi(|b|-C')
$$
and again \eqref{coerc} implies the conclusion. The case $\psi(s) = s^p$ may be proved along similar lines taking into account that \eqref{eq: A} can be replaced by $|A|\le \tilde{C}M^{\frac{1}{p}}$ for $\tilde{C}$ independent of $M$. 
\end{proof}

 
\subsection{Function spaces}\label{sec: gsbd}

In the whole paper we use standard notations for the spaces $SBV$ and $SBD$. We refer the reader to \cite{AFP} and \cite{ACD, BCD, Tem}, respectively, for definitions and basic properties. In this section we only give the definition and some properties of \emph{generalized functions of bounded deformation} introduced in \cite{DM}, being the setting of our existence result.  For fixed $\xi \in S^{N-1}$, we set
\begin{align*}
\Pi^\xi:=\{y \in \R^{N}: y\cdot \xi=0\}\,,&\quad U^\xi_y:=\{t \in \R: y+t\xi\in U\}\hbox{ for }y\in \Pi^\xi.
\end{align*}

\begin{definition}
An $\mathcal L^{N}$-measurable function $u:U\to \R^{N}$ belongs to $GBD(U)$ if there exists a positive bounded Radon measure $\lambda_u$ such that, for all $\tau \in C^{1}(\R^{N})$ with $-\frac12 \le \tau \le \frac12$ and $0\le \tau'\le 1$, and all $\xi \in S^{N-1}$, the distributional derivative $D_\xi (\tau(u\cdot \xi))$ is a bounded Radon measure on $U$ whose total variation satisfies
$$
\left|D_\xi (\tau(u\cdot \xi))\right|(B)\le \lambda_u(B)
$$
for every Borel subset $B$ of $U$. A function $u \in GBD(U)$ belongs to the subset $GSBD(U)$ of special functions of bounded deformation if in addition for every $\xi \in S^{N-1}$ and $\mathcal H^{N-1}$-a.e.\ $y \in \Pi^\xi$, the function $u^\xi_y(t):=u(y+t\xi)$ belongs to $SBV_{\mathrm{loc}}(U^\xi_y)$.
\end{definition}

By \cite[Remark 4.5]{DM} one has the inclusions $BD(U)\subset GBD(U)$ and $SBD(U)\subset GSBD(U)$, which are in general strict. Some relevant properties of functions with bounded deformation can be generalized to this weak setting: in particular, in \cite[Theorem 6.2 and Theorem 9.1]{DM} it is shown that the jump set $J_u$ of a $GBD$-function is $\mathcal H^{N-1}$-rectifiable and that $GBD$-functions have an approximate symmetric differential $e(u)(x)$ at $\mathcal L^{N}$-a.e.\ $x\in U$, respectively. The space $GSBD^2(U)$ is defined through:
$$
GSBD^2 (U):= \{u \in GSBD(U): e(u) \in L^2 (U; \mathbb R_{\mathrm{sym}}^{N\times N})\,,\,\mathcal H^{N-1}(J_u) < +\infty\}\,.
$$

Furthermore, the following compactness theorem has been proved in \cite{DM}, which we slightly adapt for our purposes.
\begin{theorem}\label{th: GSBD comp}
Let $\Gamma$ be a measurable set with ${\cal H}^{N-1}(\Gamma) < + \infty$.  Let $(y_k)_k$ be a sequence in $GSBD^2(U)$. Suppose that there exist a constant $M>0$ and an increasing continuous functions $\psi:[0,\infty) \to [0,\infty)$ with $\lim_{s \to \infty} \psi(s) = + \infty$ such that 
$$\int_{U} \psi(|y_k|)\mathrm{d}x + \int_{U} |e(y_k)|^2\mathrm{d}x + {\cal H}^{N-1}(J_{y_k}) \le M $$
for every $k \in \N$. Then there exist a subsequence, still denoted by $(y_k)_k$, and a function $y \in GSBD^2(U)$ such that
\begin{align}\label{eq: convergence sense}
\begin{split}
& y_k \to y \ \ \ \text{in measure in} \ \ \ U,\\ 
&e (y_k) \rightharpoonup  e (y)  \ \ \text{ weakly in} \ L^2(U;\R^{N\times N}_{\rm sym}),\\
& {\cal H}^{N-1}(J_y \setminus \Gamma) \le \liminf_{k \to \infty} {\cal H}^{N-1}(J_{y_k} \setminus  \Gamma).
\end{split}
\end{align}
\end{theorem}

\Proof In \cite{DM} the assertion has been proved in the case $\Gamma = \emptyset$. We briefly indicate the necessary adaption for the derivation of \eqref{eq: convergence sense}(iii) following the argumentation in \cite[Theorem 2.8]{DFT}. If $\Gamma$ is compact, it suffices to replace $\Omega$ by $\Omega \setminus \Gamma$. In the general case let $K \subset \Gamma$ compact with ${\cal H}^1(\Gamma \setminus K) \le \eps$. Since $J_y \setminus \Gamma \subset J_y \setminus K$ and $J_{y_k} \setminus K \subset (J_{y_k} \setminus \Gamma) \cup (\Gamma \setminus K)$ we have
\begin{align*}
{\cal H}^1(J_y \setminus \Gamma) &\le {\cal H}^1(J_y \setminus K) \le \liminf\nolimits_{k\to\infty}{\cal H}^1(J_{y_k} \setminus K) \\
& \le \liminf\nolimits_{k\to\infty}{\cal H}^1(J_{y_k} \setminus \Gamma) + {\cal H}^1(\Gamma \setminus K) \le \liminf\nolimits_{k\to\infty}{\cal H}^1(J_{y_k} \setminus \Gamma) + \eps.
\end{align*}
We conclude by letting $\eps \to 0$. \eop

We now define a class of displacements with regular jump set. We say  that $u \in L^1(U; \R^N)$ is a displacements with regular jump set if the following properties are satisfied
\begin{align}\label{def: regular}
(i) & \nonumber \ \ u \in SBV^2(U; \R^N),\\
(ii) & \ \ J_{u}  = \bigcup_{k=1}^m\Sigma_k,\quad \Sigma_k \hbox{ closed connected pieces of }C^1\hbox{-hypersurfaces}, \\
(iii) &  \nonumber \ \ u \in H^{1}(U \setminus J_{u};\R^N).
\end{align}
Displacements with regular jump set are dense in $GSBD^2(U)\cap L^2(U; \R^N)$ in the sense given by the following statement, proved in \cite{Iur} (cf. also \cite[Theorem 3, Remark 5.3]{Chambolle:2004})).

\begin{theorem}\label{th: density-l2}
Let $U \subset \R^N$ open, bounded with Lipschitz boundary. Let $u \in GSBD^2(U)\cap L^2(U; \R^N)$. Then there exists a sequence $(u_k)_k$ of displacements with regular jump set so that
\begin{align*}
(i) & \ \ \Vert u_k-u  \Vert_{L^2(\Omega;\R^N)} \to 0\\
(ii) & \ \ \Vert e(u_k) - e(u) \Vert_{L^2(\Omega;\R^{N\times N}_{\rm sym})} \to 0,\\
(iii) &  \ \ {\cal H}^{N-1}(J_{u_k} \triangle J_u) \to  0.
\end{align*}
\end{theorem}

\subsection{Caccioppoli partitions}\label{sec: caccio}

We say that a partition ${\cal P} = (P_j)_j$ of an open set $U\subset \R^N$ is a \textit{Caccioppoli partition} of $U$ if 
$$
\sum\nolimits_j {\cal H}^1(\partial^* P_j) < + \infty,
$$
where $\partial^* P_j$ denotes  the \emph{essential boundary} of $P_j$ (see \cite[Definition 3.60]{AFP}). We say a  partition is \textit{ordered} if $\mathcal L^N\left(P_i\right) \ge \mathcal L^N\left(P_j\right)$ for $i \le j$. In the whole article, when dealing with infinite partitions, we will always tacitly assume that they are ordered. 
Moreover, we say that a set of finite perimeter $P_j$ is \emph{indecomposable} if it cannot be written as $P^1 \cup P^2$ with $P^1 \cap P^2 = \emptyset$, ${\cal L}^N(P^1), {\cal L}^N(P^2) >0$ and ${\cal H}^{N-1}(\partial^* P_j) = {\cal H}^{N-1}(\partial^* P^1) + {\cal H}^{N-1}(\partial^* P^2)$. The  local structure of Caccioppoli partitions can be characterized as follows (see \cite[Theorem 4.17]{AFP}).
\begin{theorem}\label{th: local structure}
Let $(P_j)_j$ be a Caccioppoli partition of $U$. Then 
$$\bigcup\nolimits_j (P_j)^1 \cup \bigcup\nolimits_{i \neq j} (\partial^* P_i \cap \partial^* P_j)$$
contains ${\cal H}^{N-1}$-almost all of $U$.
\end{theorem}
Here $(P)^1$ denote the points where $P$ has density one (see again \cite[Definition 3.60]{AFP}). 
 Essentially, the  theorem states that ${\cal H}^{N-1}$-a.e. point of $U$ either belongs to exactly one element of the partition or to the intersection of exactly two sets $\partial^* P_i$, $\partial^* P_j$. 
We now state a compactness result for ordered Caccioppoli partitions (see \cite[Theorem 4.19, Remark 4.20]{AFP}) slightly adapted for our purposes. 

\begin{theorem}\label{th: comp cacciop}
Let $U \subset \R^N$ open, bounded with Lipschitz boundary.  Let ${\cal P}_i = (P_{j,i})_j$, $i \in \N$, be a sequence of ordered Caccioppoli partitions of $U$ with $$\sup\nolimits_{i \ge 1} \sum\nolimits_{j\ge 1}{\cal H}^{N-1}(\partial^* P_{j,i}) < + \infty.$$
Then there exists a Caccioppoli partition ${\cal P} = (P_j)_j$ and a not relabeled subsequence such that  
$\sum_{j \ge 1} \mathcal L^N\left(P_{j,i} \triangle  P_j\right) \to 0$  as $i \to \infty$.
\end{theorem}

\Proof In \cite{AFP} it was proved that $P_{j,i} \to P_j$ in measure for all $j \in \N$ as $i \to \infty$. We briefly show that this already implies $\sum_j \mathcal L^N\left(P_{j,i} \triangle  P_j\right) \to 0$  as $i \to \infty$. Let $\eps >0$ and choose $j_0 \in \N$ sufficiently large such that $\sum_{j < j_0}\mathcal L^N\left(P_j\right) \ge \mathcal L^N\left(U\right) - \eps$. Then the convergence in measure implies that for $i_0$ large enough depending on $j_0$ we have $\sum_{j < j_0} \mathcal L^N\left(P_{j,i} \triangle  P_j\right)\le \eps$ for all $i \ge i_0$. Moreover, this overlapping property and the choice $j_0$ imply $\sum_{j \ge j_0} \mathcal L^N\left(P_{j,i}\right) \le 2\eps$ for $i \ge i_0$. Consequently, we find $\sum_j \mathcal L^N\left(P_{j,i} \triangle  P_j\right) \le 4\eps$ for $i \ge i_0$. As $\eps >0$ was arbitrary, the assertion follows.   \eop

\subsection{Piecewise Korn inequality in GSBD}\label{sec:korn}
In this section we recall  a piecewise Korn inequality for $GSBD$ functions, proved in the planar setting in \cite{Friedrich:15-4} and being one of the major ingredients of our proofs. It implies in particular a density result Theorem \ref{th: density} which in the planar case improves upon Theorem \ref{th: density-l2}. Here and henceforth we will call an affine mapping of the form $a_{A,b}(x):=Ax+b$  with $A\in \R_{\mathrm{skew}}^{2\times 2}$ and $b \in \R^2$ an  {\it infinitesimal rigid motion}.

\begin{theorem}\label{th: korn}
Let $\Omega \subset \R^2$ open, bounded with Lipschitz boundary.  Let $p \in [1,2)$. Then there is a constant $c=c(p)>0$ and   $C_{\rm korn}=C_{\rm korn}(p,\Omega)>0$   such that for  each $u \in GSBD^2(\Omega)$  there is a Caccioppoli partition $\Omega = \bigcup^\infty_{j=1} P_j$ and corresponding infinitesimal rigid motions $(a_j)_j = (a_{A_j,b_j})_j$ such that  
$$v := u - \sum\nolimits^\infty_{j=1} a_j \chi_{P_j} \in SBV^p(\Omega; \R^2) \cap { L^\infty}(\Omega;\R^2)$$ 
and 
\begin{align}\label{eq: small set main}
\begin{split}
(i) & \ \   \sum\nolimits_{j=1}^\infty {\cal H}^1( \partial^* P_j ) \le  c({\cal H}^1(J_u) + {\cal H}^1(\partial \Omega)),\\
(ii) &\ \ \Vert v \Vert_{L^{\infty}(\Omega; \R^2)} \le C_{\rm korn} \Vert  e(u) \Vert_{L^2(\Omega; \R_{\mathrm{sym}}^{2\times 2})},\\ 
(iii) &\ \  \Vert \nabla v\Vert_{L^p(\Omega; \R^{2\times 2})} \le C_{\rm korn} \Vert  e(u) \Vert_{L^2(\Omega; \R_{\mathrm{sym}}^{2\times 2})}.
\end{split}
\end{align}
\end{theorem}

Below in Section \ref{sec:korn-poincare} we prove a refined version of Theorem \ref{th: korn} which (a) provides a sharp estimate for the boundary of the partition in \eqref{eq: small set main}(i) and (b) takes into account boundary data. This refined result will then be fundamental in proving the jump transfer lemma and  the existence theorem for the time-incremental minimum problems.

Applying the above result, approximating $u$ by the sequence   $v_n := u - \sum\nolimits^\infty_{j=n+1} a_j \chi_{P_j} \in GSBD^2(\Omega) \cap { L^\infty}(\Omega;\R^2)$, and using Theorem \ref{th: density-l2}, we  obtain the following density result for $GSBD$ functions (see again \cite{Friedrich:15-4}).

\begin{theorem}\label{th: density}
Let $\Omega \subset \R^2$ open, bounded with Lipschitz boundary. Let $u \in GSBD^2(\Omega)$. Then there exists a sequence $(u_k)_k$ of displacements with regular jump set such that
\begin{align*}
(i) & \ \ u_k \to u \ \text{ in measure},\\
(ii) & \ \ \Vert e(u_k) - e(u) \Vert_{L^2(\Omega;\R^{2\times 2}_{\rm sym})} \to 0,\\
(iii) &  \ \ {\cal H}^{1}(J_{u_k} \triangle J_u) \to  0.
\end{align*}
\end{theorem}

Note that in contrast to the original density result reported in Theorem \ref{th: density-l2} the assumption that $u \in L^2(\Omega)$ is not needed in the planar setting. 

\section{The model and statement of the main result}\label{sec: general}
In this section we introduce the model we study and we fix the related notations. This preliminary discussion is still conducted in a general $N$-dimensional setting, while our main result, given at the end of the section, is stated and proved only in the planar case $N=2$.

We analyze the evolution of a brittle material in the sense of Griffith \cite{Griffith:20} whose total energy consists of a linear elastic bulk term and a surface term proportional to the  $(N-1)$-dimensional measure of the crack. The body is under the action of a time-dependent {\it prescribed boundary displacement} $g(t)$ on a relatively open part $\partial_D \Omega$ of the boundary ({\it Dirichlet part}) of the reference configuration $\Omega \subset \R^N$, which is supposed to be open, bounded with Lipschitz boundary. The rest of the boundary will be instead assumed to be force-free for simplicity. The variables of the model are a $GSBD$-valued displacement $u$ and a (not a priori prescribed) crack $\Gamma$ with finite $\mathcal H^{N-1}$ measure. The uncracked part of the body has a linear elastic stored energy of the form
$$
\int_{\Omega \setminus \Gamma} Q(e(u))\,\mathrm{d}x.
$$
In the above expression $e(u)$ is the approximate symmetrized gradient of $u$ and $Q: \mathbb R_{\mathrm{sym}}^{N\times N} \to \R$ is the quadratic form associated to a symmetric bounded and positive definite stiffness tensor $\C:\mathbb R_{\mathrm{sym}}^{N\times N}\to \mathbb R_{\mathrm{sym}}^{N\times N}$, that is
\begin{equation}\label{eq: elasticity-tensor}
Q(e):=\frac12\C e \colon e, 
\end{equation}
with the colon denoting the Euclidean product between matrices.

The prescribed boundary displacement $g$ is a time dependent function \\$g\in~W^{1,1}_{\mathrm{loc}}([0,+\infty);H^{1}(\R^N;\R^N))$. As it is typical for the weak formulation of evolutionary problems in spaces of functions of bounded deformation, the boundary condition will be relaxed as follows. We will  assume that it exists an open, bounded Lipschitz set $\Omega' \supset \Omega$ such that
\begin{equation}\label{eq: omega'}
\Omega' \cap \overline\Omega=\partial_D \Omega\,\qquad \Omega'\setminus \overline\Omega \hbox{ has Lipschitz boundary}
\end{equation}
and impose, for every time $t$, that an admissible displacement $u(t)$ satisfies $u(t)=g(t)$ a.e.\ in $\Omega'\setminus \overline\Omega$. A competing crack may choose indeed to run alongside $\partial_D \Omega$, in which case the boundary condition is not attained in the sense of traces, at the expense of a crack energy. 

The energy of a crack $\Gamma \subset \overline\Omega$ will be proportional to its ($N-1)$-dimensional Hausdorff measure, namely of the form
$$
\kappa \hn(\Gamma\cap \Omega'),
$$
where the material parameter $\kappa$ represents the toughness of the material. Within this choice, and because of \eqref{eq: omega'}, formation of cracks along $\partial_D \Omega$ is penalized, while no energy is spent for a crack sitting on the load-free part of the boundary $\partial \Omega\setminus \partial_D \Omega$. In the following we will set $\kappa=1$ without loss of generality. 

The quasistatic evolution problem associated to the model with the prescribed boundary displacement $g(t)$ consists in finding a displacement and crack path $(u(t), \Gamma(t))$ with $J_{u(t)}\subset \Gamma(t)\subset \overline{\Omega}$ and $u(t)=g(t)$ a.e.\ in $\Omega'\setminus \overline\Omega$ such that $\Gamma(t)$ is irreversible, namely $\Gamma(t)\supset \Gamma(s)$ whenever $t>s$, and the following two conditions hold:
\begin{itemize}
 \item {\it global stability}. For each $t$, $u(t)$ minimizes
\begin{equation}\label{eq: global stability}
 \int_{\Omega} Q(e(v))\,\mathrm{d}x +\hn(J_v\setminus \Gamma(t))
\end{equation}
 among all $v\in GSBD^{2}(\Omega')$ such that $v=g(t)$ on $\Omega'\setminus \overline{\Omega}$;
\item{\it energy-dissipation balance}. The total energy 
\begin{equation}\label{eq: total energy}
\mathcal{E}(t):=\int_{\Omega} Q(e(u(t)))\,\mathrm{d}x +\hn(\Gamma(t)\cap \Omega')
\end{equation}
is absolutely continuous and satisfies for all $t>0$
\begin{equation}\label{eq: energy balance}
\mathcal{E}(t)=\mathcal{E}(0)+\int_0^t\langle\sigma(s), e(\dot g(s))\rangle\,\mathrm{d}s,
\end{equation}
where $\sigma(s)=\C e(u(s))$, $\langle \cdot,\cdot \rangle$ is the duality pairing in $L^2(\Omega; \R^{N\times N}_{\mathrm{sym}})$, and  $\dot g(s)$ denotes the   Fr\`{e}chet derivative of $g$ with respect to $s$.
\end{itemize}

Notice that even for a given $\Gamma$, the existence of a minimizer for the problem considered in \eqref{eq: global stability} is a nontrivial issue, which we are able to overcome for the moment only in the planar case $N=2$ (Theorem \ref{th: existence}).  
Indeed, in the planar case we are able to show the existence of a quasistatic evolution according to the following statement, which constitutes the main result of the paper.

\begin{theorem}\label{th: main}
Let $N=2$. Let $\Omega\subset\Omega'$ be bounded domains in $\R^2$ with Lipschitz boundary satisfying \eqref{eq: omega'}, $g\in W^{1,1}_{\rm loc}([0,+\infty); H^{1}(\R^2;\R^2))$, and consider $Q$ as in \eqref{eq: elasticity-tensor}. Then, for all $t\ge 0$ it exists an $\mathcal H^1$-rectifiable crack $\Gamma(t) \subset \overline{\Omega}$ and a field $u(t)\in GSBD^{2}(\Omega')$ such that
\begin{itemize}
\item $\Gamma(t)$ is nondecreasing in $t$;
\item $u(0)$ minimizes
 $$
 \int_{\Omega} Q(e(v))\,\mathrm{d}x +\mathcal H^1(J_v)
 $$
 among all $v\in GSBD^{2}(\Omega')$ such that $v=g(0)$ on $\Omega'\setminus \overline{\Omega}$;
 \item for all $t>0$, $u(t)$ satisfies the global stability \eqref{eq: global stability} for $N=2$;
 \item $J_{u(0)}=\Gamma(0)$ and $J_{u(t)}\subset \Gamma(t)$ up to a set of $\mathcal H^1$-measure $0$.
\end{itemize}

Furthermore, the total energy $\mathcal{E}(t)$ defined by \eqref{eq: total energy} satisfies the energy dissipation balance \eqref{eq: energy balance}.
Finally, for any countable, dense subset $I\subset [0,+\infty)$ containing zero, we have
$$
\Gamma(t)=\bigcup_{\tau \in I\,,\,\tau \le t}J_{u(\tau)}
$$
for all $t>0$.
\end{theorem}

\section{A sharp piecewise Korn  inequality in GSBD}\label{sec:korn-poincare}

In this section we derive a  piecewise Korn  inequality with a sharp estimate for the  surface energy and also prove a version taking Dirichlet boundary conditions into account. 

\subsection{A refined piecewise Korn inequality}

The goal of this section is to prove the following result.

\begin{theorem}\label{th: kornpoin-sharp}
Let $\Omega \subset \R^2$ open, bounded with Lipschitz boundary and $0 < \theta<1$. Then there is a universal constant $c>0$, some $C_\Omega=C_\Omega(\Omega)>0$ and some $C_{\theta,\Omega}=C_{\theta,\Omega}(\theta,\Omega)>0$   such that the following holds: For  each $u \in GSBD^2(\Omega)$ we find $u^\theta \in SBV(\Omega;\R^2) \cap L^\infty(\Omega;\R^2)$ such that $\lbrace u \neq u^\theta \rbrace$ is a set of finite perimeter  with
\begin{align}\label{eq: kornpoinsharp1}
\begin{split}
(i) & \ \ 
{\cal L}^2(\lbrace u \neq u^\theta \rbrace)  \le c\theta ({\cal H}^1(J_u)+ {\cal H}^1(\partial \Omega))^2, \\
(ii) & \ \    {\cal H}^1((\partial^* \lbrace u \neq u^\theta \rbrace  \cap \Omega) \setminus J_u)  \le c\theta ({\cal H}^1(J_u)+ {\cal H}^1(\partial \Omega)),
\end{split}
\end{align}
a (finite) Caccioppoli partition $\Omega = \bigcup^{I}_{i=0} P_i$, and corresponding infinitesimal rigid motions $(a_i)_{i=0}^I$ such that  
 $v := u^\theta - \sum\nolimits_{i=0}^{I} a_i \chi_{P_i} \in SBV(\Omega;\R^2)\cap L^\infty(\Omega;\R^2)$ and
\begin{align}\label{eq: kornpoinsharp2}
\begin{split}
(i) & \ \  \sum\nolimits_{i=0}^{I}{\cal H}^1( (\partial^* P_i \cap \Omega) \setminus J_u ) \le c\theta ({\cal H}^1(J_u) + {\cal H}^1(\partial \Omega)),\\
(ii) & \ \ \mathcal{L}^2(P_i) \ge C_\Omega\theta^2 \ \ \text{ for all } \ \ 1 \le i \le I, \ \ \ \mathcal{L}^2(\lbrace u \neq u^\theta \rbrace \triangle P_0) = 0,\\
(iii)& \ \ \Vert v \Vert_{L^\infty(\Omega; \R^2)} + \Vert \nabla v \Vert_{L^1(\Omega;\R^{2\times 2})} \le C_{\theta,\Omega} \Vert  e(u) \Vert_{L^2(\Omega; \R_{\mathrm{sym}}^{2\times 2})}.
\end{split}
\end{align}
\end{theorem}

Note that the refined estimate \eqref{eq: kornpoinsharp2}(i) comes at the expense of the fact that we have to pass to a slightly modified function (see \eqref{eq: kornpoinsharp1}) and that in \eqref{eq: kornpoinsharp1}(iii) only the $L^1$-norm of the derivative is controlled.

\begin{rem}\label{rem:square}
Let $C_{Q_1}$ and $C_{\theta,Q_1}$ be the constants in Theorem \ref{th: kornpoin-sharp} for the unit square $\Omega = Q_1 = (0,1)^2$. Using a rescaling argument, \eqref{eq: kornpoinsharp2}(ii),(iii) in Theorem \ref{th: kornpoin-sharp} applied for any square $\Omega =Q \subset \R^2$ read as 
\begin{align}\label{eq:local1-new}
\begin{split}
(i) & \ \ \mathcal{L}^2(P_i) \ge C_{\rm Q_1}\mathcal{L}^2(Q)\theta^2 \ \ \text{ for } \ \ 1 \le i \le I, \\
(ii) &  \ \   \Vert v \Vert_{L^\infty(Q; \R^2)} + (\diam(Q))^{-1} \Vert \nabla v \Vert_{L^1(Q;\R^{2\times 2})} \le C_{\theta,Q_1} \Vert  e(u) \Vert_{L^2(Q; \R_{\mathrm{sym}}^{2\times 2})}.
\end{split}
\end{align}
\end{rem}

Below after the proof of Theorem \ref{th: kornpoin-sharp} we briefly indicate how Remark \ref{rem:square} can be derived from \eqref{eq: kornpoinsharp2} for convenience of the reader. As a preparation we formulate two lemmas. Recall the notion of decomposable sets in Section \ref{sec: caccio}  and the definition of  ${\rm diam}$ in Section \ref{sec:pre1}. 

\begin{lemma}\label{lemma: diam}
Let $B \subset \R^2$ be an indecomposable, bounded set with finite perimeter. Then  ${\rm diam}(B) \le {\cal H}^1(\partial^* B)$.
\end{lemma}

The proof can be found in \cite[Propostion 12.19, Remark 12.28]{maggi}. The following lemma investigates some properties of the jump set of a piecewise-defined function on the interface of two sets of finite perimeter.

  \begin{lemma}\label{lemma: interface}
Let $\Omega \subset \R^2$ open, bounded and $y \in SBV(\Omega;\R^2) \cap L^\infty(\Omega;\R^2)$. Let  $P_1,P_2 \subset \Omega$ be sets of finite perimeter and $a_i = a_{A_i,b_i}$, $i=1,2$, infinitesimal rigid motions. Then  there is a ball $B \subset \R^2$ with 
\begin{align*}
(i) & \ \ \diam(B) \le 4 \diam(P_2) \,  \Vert a_1 - a_2 \Vert^{-1}_{L^\infty(P_2;\R^2)} \, \sum\nolimits_{i=1,2}\Vert y - a_i \Vert_{L^\infty(P_i;\R^2)}, \\ 
(ii) & \ \ \mathcal{H}^1\big( (\partial^* P_1 \cap \partial^* P_2) \setminus (B \cup J_y) \big) = 0.
\end{align*}   
\end{lemma}

\begin{proof}
We define $\gamma =  \Vert a_1 - a_2 \Vert_{L^\infty(P_2;\R^2)}$ and  $\delta=  \sum\nolimits_{i=1,2}\Vert y - a_i \Vert_{L^\infty(P_i;\R^2)} $ for shorthand. First, if $\delta \ge \frac{1}{2}\gamma$, we can choose $B$ as a ball containing $P_2$ with $\diam(B) \le 2\diam(P_2)$. Consequently, it suffices to assume $\delta < \frac{1}{2}\gamma$.

For $i=1,2$ we denote by $T_i y$ the trace of $y$ on $\partial^* P_i$, which exists by \cite[Theorem 3.77]{AFP} and satisfies  
$$|T_i y(x) - a_i(x)| \le \Vert y - a_i \Vert_{L^\infty(P_i;\R^2)} \ \ \ \text{for $\mathcal{H}^1$-a.e. } x \in \partial^* P_i.$$
Assume the statement was wrong. Then we would find two points $x_1,x_2$ with $|x_1-x_2|> 4\gamma^{-1}\delta\diam(P_2)$ such that $x_1,x_2 \in (\partial^*P_1 \cap \partial^* P_2) \setminus J_y$ and  for $i,j=1,2$
$$|T_i y(x_j) - a_i(x_j)| \le \Vert y - a_i \Vert_{L^\infty(P_i;\R^2)}.$$ 
Since $x_1,x_2 \notin J_y$ and thus $T_1y(x_1) = T_2y(x_1)$,  $T_1y(x_2) = T_2y(x_2)$ we compute
\begin{align*}
|a_1(x_j) - a_2(x_j)| \le |T_1 y(x_j) - a_1(x_j)| + |T_2 y(x_j) - a_2(x_j)| \le \delta
\end{align*}
 for $j=1,2$. Combining the estimates for $j=1,2$ we get  
\begin{align*}
|x_1-x_2||A_1 - A_2| &\le 2|(A_1\,x_1 + b_1) - (A_2\,x_1 +b_2) - (A_1\,x_2 + b_1) + (A_2\,x_2 + b_2)|
\\& \le 2(|a_1(x_1) - a_2(x_1)| + |a_1(x_2) - a_2(x_2)|) \le 2\delta 
\end{align*}
and therefore $|A_1-A_2| \le \frac{1}{2}(\diam(P_2))^{-1} \gamma$ as well as
\begin{align*}
\gamma = \Vert a_1 - a_2 \Vert_{L^\infty(P_2;\R^2)} \le |a_1(x_1) - a_2(x_1)| + \diam(P_2)|A_1-A_2|  \le    \delta + \tfrac{1}{2}\gamma,
\end{align*}
which contradicts $\gamma >  2 \delta$. 
\end{proof}

\begin{proof}[Proof of Theorem \ref{th: kornpoin-sharp}] 
Let $u \in GSBD^2(\Omega)$ be given and set for shorthand $\mathcal{E} = \Vert  e(u) \Vert_{L^2(\Omega; \R_{\mathrm{sym}}^{2\times 2})}$ and  $J_u' = J_u \cup \partial \Omega$. Without restriction we can assume $\theta^{-1}\in \N$ and that $\Omega$ is connected as otherwise the following arguments are applied for each connected component of $\Omega$. Moreover, we may suppose that $\mathcal{H}^1(J_u) \le  (\theta^{-1}\mathcal{L}^2(\Omega))^{\frac{1}{2}} $ as otherwise the assertion trivially holds with $u^\theta=0$.

In the following $c>0$ stands for a universal constant and $C_\Omega=C_\Omega(\Omega)>0$, $C_{\theta, \Omega}=C_{\theta, \Omega}(\theta,\Omega)>0$ represent generic constants which may vary from line to line. We may further assume that $\theta$ is chosen   (depending on $\Omega$) such that $\theta \le \theta_0 :=\frac{1}{16}C_{\rm korn}^{-1}$,  where $C_{\rm korn}$ is the constant from \eqref{eq: small set main}.\footnote{If $\theta> \theta_0$, the result holds for $u^\theta = u^{\theta_0}$, upon replacing $C_\Omega$ by $C_\Omega \theta_0^2$ in \eqref{eq: kornpoinsharp2}(ii).}

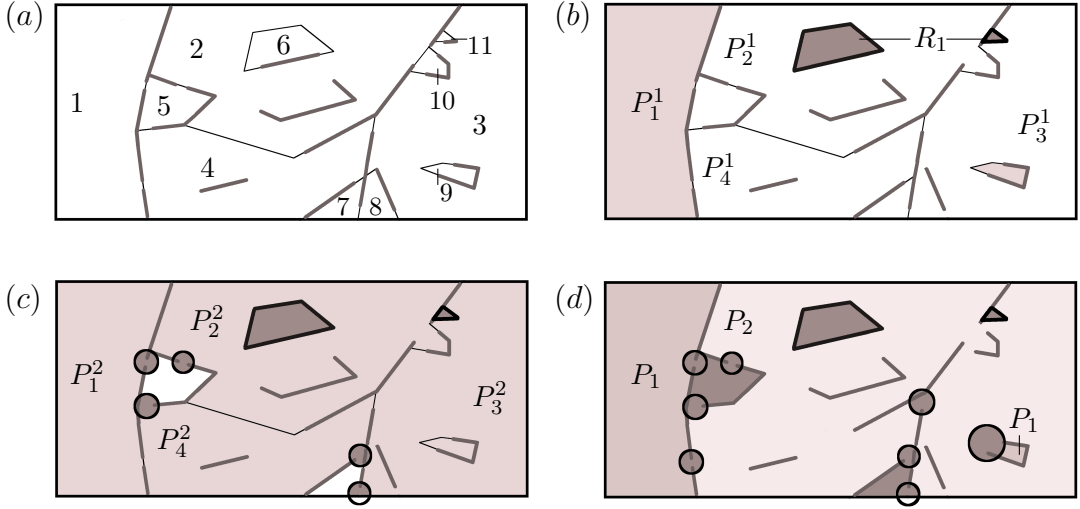
\begin{figure}[!h]
\centering
\begin{overpic}[width=1.0\linewidth,clip]{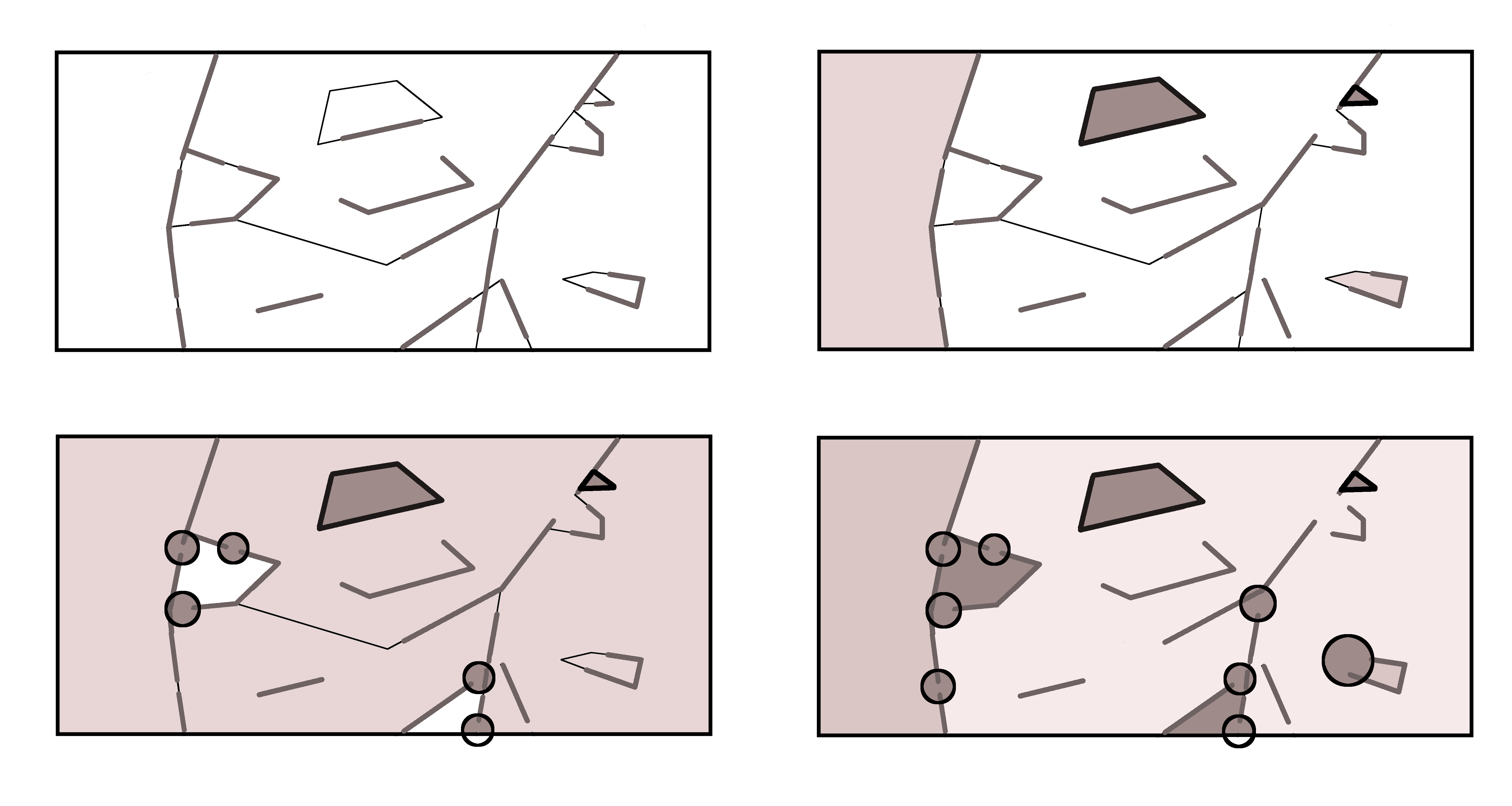}
\put(-1,48){{$(a)$}}
\put(49.5,48){{$(b)$}}
\put(-1,22){{$(c)$}}
\put(49.5,22){{$(d)$}}

\put(5,40){{\small $1$}}
\put(16,45){{\small $2$}} 
\put(42,38){{\small $3$}}
\put(17,34){{\small $4$}} 
\put(13,39.7){{\small $5$}} 
\put(24,45.3){{\small $6$}}
\put(29.5,30.5){{\footnotesize $7$}} 
\put(32.5,30.5){{\footnotesize $8$}}
\put(39,31.8){{\footnotesize $9$}}
\put(38.8,33.3) {\line(0,1){1.5}}
\put(38,40.3){{\footnotesize $10$}}
\put(38.8,42.5) {\line(0,1){1.5}}
\put(41.5,45.3){{\footnotesize $11$}}
\put(40.0,46.7) {\line(1,0){1.5}}

\put(56.5,40){{\small $P_1^1$}}
\put(65,45){{\small $P_2^1$}} 
\put(92,38){{\small $P_3^1$}}
\put(63,34){{\small $P_4^1$}}

\put(77.5,46.7) {\line(1,0){4.7}}
\put(82.5,46){{\small $R_1$}}
\put(85.5,46.7) {\line(1,0){3.8}}

\put(5,15){{\small $P_1^2$}}
\put(16,20){{\small $P_2^2$}} 
\put(42,13){{\small $P_3^2$}}
\put(13,9){{\small $P_4^2$}}

\put(56.5,15){{\small $P_1$}}
\put(65,20){{\small $P_2$}} 
\put(91.2,11){{\small $P_1$}}
\put(92.2,8.5) {\line(0,1){1.8}}

\end{overpic}
\caption{\small Illustration of the constructions in the proof of Theorem  \ref{th: kornpoin-sharp}. (a) The partition $(P'_j)_{j=1}^{11}$ is sketched (for convenience only the indices are given). Note that in general the jump set (depicted in light gray) is not a subset of $\bigcup^{11}_{j=1}\partial^* P_j'$. (b) The \emph{large components} of $(P^1_j)^{6}_{j=1}$ are given by $P_1^1 = P_1' \cup P_9'$, $P_2^1 = P_2'\cup P_{10}'$, $P_3^1 = P_3' \cup P_8'$, $P_4^1 = P_4'$ (i.e. $I'=4$), the exceptional set is $R_1 = P_6' \cup P_{11}'$ and the \emph{small components} are $P_5'$,  $P_7'$. Observe that $P_1^1$,  depicted in light gray, is not connected. (c) The union of balls $R_2$ is illustrated and the set $\Omega_{\rm good} = \bigcup^4_{j=1}P^1_j\setminus R_2 = \bigcup^4_{j=1}P^2_j$ is given in light gray. (d) In this example we have $R_4 = \emptyset$. The set $\Omega_{\rm bad}$ is depicted in dark gray and $\Omega \setminus \Omega_{\rm bad} = P^3_1 \cup P^3_2 = P_1 \cup P_2$ consists of 
two components, i.e. $I''=2$. We further have $\mathcal{Z}_1 = \emptyset$, $\mathcal{Z}_2 = \lbrace (1,2),(1,3),(1,4),(3,4)\rbrace$ and $\mathcal{Z}_3 = \lbrace (2,3),(2,4) \rbrace$.} \label{figure}
\end{figure}

\noindent \textit{Step 0 (Overview of the proof).} The general idea behind the proof is to modify suitably the infinitesimal rigid motions provided by Theorem \ref{th: korn} so that all the sets $P_j$ of the Caccioppoli partition are \emph{almost completely disconnected} by $J_u$: by this we mean that the interface between different components will be contained in the jump set of $u$ up to a small (in area and perimeter) exceptional set. In doing this, we must anyway be able not to lose the estimate in \eqref{eq: small set main}(iii). These are the main observations that allow us to pursue this strategy:
\begin{itemize}
 \item[(O1)] If the $L^\infty$ distance between two infinitesimal rigid motions $a_{j_1}$ and $a_{j_2}$, that are subtracted from $u$ on two sets $P_{j_1}$ and $P_{j_2}$, respectively, lies below a fixed threshold depending on the error parameter $\theta$ (see \eqref{eq:I}(iii)), we can replace $a_{j_2}$ with $a_{j_1}$ on $P_{j_2}$. Indeed, by construction and using Lemma \ref{lemma: rigid motion}(a), \eqref{eq: small set main}(iii) will still hold up to enlarging $C_{\theta, \Omega}$ suitably.
 
 \item[(O2)] If the $L^\infty$ distance between two infinitesimal rigid motions $a_{j_1}$ and $a_{j_2}$, that are subtracted from $u$ on two sets $P_{j_1}$ and $P_{j_2}$, respectively, lies above an (even larger) fixed threshold depending on $\theta$ (see \eqref{eq:I}(iv)), using Lemma \ref{lemma: interface} the interface between $P_{j_1}$ and $P_{j_2}$ not contained in $J_u$ can be covered by a small ball. This will lead to neglecting a small exceptional set with small perimeter, provided this is not done `too often'. Some combinatorial arguments will indeed be needed (cf., for instance, the derivation of \eqref{eq:balls} later in the proof).
 
 \item[(O3)] On neighboring components $P_{j_1}$ and $P_{j_2}$, whose size lies above a fixed threshold depending on $\theta$, and that are {\it not} almost completely disconnected by $J_u$, the $L^\infty$ estimate in \eqref{eq: small set main}(iii), the continuity of $u$ on part of the interface, together with Lemma \ref{lemma: rigid motion}(a), allow us to estimate the $L^\infty$ distance between the corresponding infinitesimal rigid motions $a_{j_1}$ and $a_{j_2}$ basically only in terms of $\theta$, and therefore we may apply (O1) to remove the artificially introduced boundaries.
\end{itemize}

Guided by these observations, the proof is organized as follows. In Step I we reorganize the partition given by Theorem \ref{th: korn} into \emph{large sets}, of size at least $\theta^2\mathcal{L}^2(\Omega)$, \emph{small sets}, covering only a small part of $\Omega$ and a \emph{rest set}, denoted by $R_1$, which has small perimeter (see \eqref{eq:R}). Using (O1) the partition has now the property that the infinitesimal rigid motions given on large and small components, respectively, differ very much (see \eqref{eq:I}(iv)). This is the starting point for Step II, where, in the spirit of (O2),  we show that the part of the interfaces between large and small components not contained in $J_u$ can be covered by an exceptional set which is small in area and perimeter. In Step III we then investigate the difference of the infinitesimal rigid motions given on large components, again employing Lemma \ref{lemma: interface} to completely disconnect various components, and using (O3) on the others. In Step IV we collect 
all estimates and conclude the proof.\\
\smallskip

\noindent \textit{Step I (Identification of large components).}  The goal of this step is to define a set $R_1 \subset \Omega$ with 
\begin{align}\label{eq:R}
\mathcal{H}^1(\partial^* R_1) \le \theta {\cal H}^1(J_u'), \ \ \ \ \mathcal{L}^2(R_1) \le c\theta^2({\cal H}^1(J_u'))^2,
\end{align}
an (ordered) Caccioppoli partition $\Omega \setminus R_1 = \bigcup^\infty_{j=1} P^1_j$  and corresponding infinitesimal rigid motions $(a^1_j)_j$ such that $v_1 := u - \sum_{j \ge 1} a^1_j \chi_{P^1_j}$ satisfies for an index $I' \in \N$ with $I' \le \theta^{-2}$ and some $K_\theta \in \N$, $K_\theta \le \theta^{-1}$,
\begin{align}\label{eq:I}
(i) & \ \ \mathcal{L}^2(P^1_j) \ge \theta^2\mathcal{L}^2(\Omega) \ \   \text{for all} \ \ 1 \le j \le I', \ \  \mathcal{L}^2\big(\Omega \setminus \bigcup\nolimits_{j=1}^{I'} P_j^1 \big) \le c\theta({\cal H}^1(J_u'))^2, \notag\\
(ii) & \ \ \sum\nolimits_{j\ge 1} \mathcal{H}^1(\partial^* P^1_j) \le c\mathcal{H}^1(J_u'), \notag\\
(iii) &  \ \ \Vert v_1 \Vert_{L^\infty(\Omega \setminus R_1; \R^{2})} \le 2C_{\rm korn}\theta^{-4K_\theta}\mathcal{E} , \ \ \ \Vert \nabla v_1 \Vert_{L^1(\Omega \setminus R_1; \R^{2 \times 2})} \le C_{\theta, \Omega}\mathcal{E}, \notag\\
(iv) & \ \ \min\nolimits_{1 \le i \le I'} \Vert a_i^1 - a_j^1 \Vert_{L^\infty(P^1_j;\R^2)} \ge \theta^{-4(K_\theta+1)}\mathcal{E}  \ \ \text{ for all } j>I'.
\end{align}
Moreover, the sets $(P^1_j)_{j>I'}$ are indecomposable,  while the sets $(P^1_j)_{j=1}^{I'}$ are possibly not indecomposable.

We first apply Theorem \ref{th: korn}  to find an ordered Caccioppoli partition $(P'_j)_{j \ge 1}$ of $\Omega$ and corresponding infinitesimal rigid motions $(a'_j)_j = (a_{A'_j,b'_j})_j$ such that $v' := u - \sum_{j \ge 1} a'_j \chi_{P'_j} \in SBV(\Omega;\R^2) \cap L^\infty(\Omega;\R^2)$ satisfies \eqref{eq: small set main}, in particular
\begin{align}\label{eq: kornpoinsharp4.0}
\Vert v' \Vert_{L^{\infty}(\Omega; \R^2)} + \Vert \nabla v'\Vert_{L^1(\Omega; \R^{2\times 2})} \le C_{\rm korn} \mathcal{E}. 
\end{align}
Without restriction we  assume that the sets $(P'_j)_{j \ge 1}$ are indecomposable. Let $I' \in \N$ be the largest index such that ${\cal L}^2(P'_{I'}) \ge \theta^2 {\cal L}^2(\Omega)$. (Recall that the partition is assumed to be ordered.) Then $I' \le \theta^{-2}$  and    by the isoperimetric inequality and  \eqref{eq: small set main}(i) 
\begin{align}\label{eq: kornpoinsharp4}
\begin{split}
(i) & \ \  \sum\nolimits_{j \ge 1} ({\cal L}^2(P'_j))^{\frac{1}{2}}  \le c\sum\nolimits_{j \ge 1} {\cal H}^1(\partial^* P'_j) \le c {\cal H}^1(J_u') \le C_{\theta, \Omega}, \\ 
(ii) & \ \   \sum\nolimits_{j >I'} {\cal L}^2(P'_j) \le \theta({\cal L}^2(\Omega))^{\frac{1}{2}} \sum\nolimits_{j >I'} ({\cal L}^2(P'_j))^{\frac{1}{2}} \\
& \ \ \ \ \ \ \ \ \ \ \  \ \ \ \ \ \ \  \ \ \  \le c\theta {\cal H}^1(\partial \Omega) \sum\nolimits_{j >I'} {\cal H}^1(\partial^* P'_j) \le c\theta ({\cal H}^1(J_u'))^2,
\end{split}
\end{align}
where in the last step of (i) we used the assumption $\mathcal{H}^1(J_u) \le (\theta^{-1}\mathcal{L}^2(\Omega))^{\frac{1}{2}} $. We introduce a decomposition for the small components according to the difference of infinitesimal rigid motions as follows. For $k \in \N$ we introduce the set of indices
\begin{align}\label{eq: kornpoinsharp5}
\begin{split}
&\mathcal{J}^0 = \lbrace j > I':  \min\nolimits_{1 \le i \le I'}\Vert a'_j - a'_i \Vert_{L^\infty(P'_j;\R^2)} \le \mathcal{E} \theta^{-4} \rbrace, \\
& \mathcal{J}^k = \lbrace j > I':  \mathcal{E} \theta^{-4k}< \min\nolimits_{1 \le i \le I'}\Vert a'_j - a'_i \Vert_{L^\infty(P'_j;\R^2)} \le \mathcal{E} \theta^{-4(k+1)} \rbrace
\end{split}
\end{align}
and  define $s_k = \sum_{j \in \mathcal{J}^k} {\cal H}^1(\partial^* P'_j)$ for $k \in \N_0$. In view of  \eqref{eq: kornpoinsharp4}(i) we find some $K_\theta \in \N$, $K_\theta \le \theta^{-1}$, such that $s_{K_\theta} \le c\theta {\cal H}^1(J_u')$.

We let $R_1 := \bigcup_{j \in \mathcal{J}^{K_\theta}} P'_j$ and the choice of $K_\theta$ together with the isoperimetric inequality shows \eqref{eq:R}. We introduce the Caccioppoli partition $(P^1_j)_{j \ge 1}$ of $\Omega \setminus R_1$ by combining different components of $(P'_j)_{j \ge 1}$. We decompose the indices in $\bigcup_{k=0}^{K_\theta-1} \mathcal{J}^k$ into sets $\mathcal{J}'_i$ with $\bigcup_{i=1}^{I'} \mathcal{J}'_i = \bigcup_{k=0}^{K_\theta-1} \mathcal{J}^k$ according to the following rule: an index $j\in \mathcal{J}^k$ is assigned to $\mathcal{J}'_i$ whenever $i$ is the smallest index such that the minimum in \eqref{eq: kornpoinsharp5} is attained.  

Define the \emph{large components} $P^1_i = P'_i \cup \bigcup_{j \in \mathcal{J}'_i} P'_j$ for $1 \le i \le I'$ and by $(P^1_i)_{i > I'}$ we denote the \emph{small components}
\begin{align}\label{eq: small components}
\big\{ P'_j: j > I',  \ j \in \bigcup\nolimits_{k=K_\theta+1}^\infty \mathcal{J}^k \big\}.
\end{align}
Then \eqref{eq:I}(i) holds by \eqref{eq:R}, \eqref{eq: kornpoinsharp4}(ii) and we see that the sets $(P^1_j)_{j>I'}$ are indecomposable. Likewise, \eqref{eq:I}(ii) follows from \eqref{eq: kornpoinsharp4}(i). Moreover, we define $a^1_j = a'_j$ for $1 \le j \le I'$ and let  $a^1_j = a'_{k_j}$ for $j >I'$, where $k_j \in \N$ such that $P^1_j  = P'_{k_j}$. We introduce $v_1 = u - \sum_{j \ge 1} a_j^1 \chi_{P_j^1}$ and observe that by \eqref{eq: kornpoinsharp4.0},   \eqref{eq: kornpoinsharp5} and the definition of $P^1_j$  for $1 \le j \le I'$ we have
\begin{align*}
\Vert v_1 \Vert_{L^\infty(P_j^1;\R^{2})} &\le \Vert  v' \Vert_{L^\infty(P_j^1;\R^{2})} + \Vert v_1 - v'\Vert_{L^\infty(P_j^1;\R^{2})}   \le C_{\rm korn}\mathcal{E} + \theta^{-4K_\theta} \mathcal{E} \\&\le 2C_{\rm korn}\theta^{-4K_\theta} \mathcal{E}.
\end{align*} 
Moreover, by Lemma \ref{lemma: rigid motion}, \eqref{eq: kornpoinsharp4.0}, \eqref{eq: kornpoinsharp4}(i), \eqref{eq: kornpoinsharp5} and the definition of $\mathcal{J}_i'$ we find  
\begin{align*}
\sum\nolimits_{i=1}^{I'}\Vert \nabla v_1 \Vert_{L^1(P_i^1;\R^{2 \times 2})} &\le \sum\nolimits_{i=1}^{I'} \Big( \Vert \nabla v' \Vert_{L^1(P_i^1;\R^{2 \times 2})} + \sum\nolimits_{j \in \mathcal{J}'_i} \mathcal{L}^2(P'_j)|A'_j - A^1_i|\Big) \\&\le \Vert \nabla v' \Vert_{L^1(\Omega;\R^{2 \times 2})} + c\sum_{i=1}^{I'}\sum_{j \in \mathcal{J}'_i} (\mathcal{L}^2(P'_j))^{\frac{1}{2}} \Vert a'_j - a'_i \Vert_{L^\infty(P_j';\R^2)}\\
&\le C_{\rm korn}\mathcal{E} + c\theta^{-4K_\theta} \mathcal{E}\sum\nolimits_{j \ge 1} (\mathcal{L}^2(P'_j))^{\frac{1}{2}}  \le C_{\theta, \Omega}\mathcal{E}.
\end{align*}
Note that the last constant $C_{\theta, \Omega}$ indeed only depends on $\theta$ and $\Omega$ since $K_\theta \le \theta^{-1}$ and $C_{\rm korn}$ only depends on $\Omega$. The last two estimates together with \eqref{eq: kornpoinsharp4.0} show \eqref{eq:I}(iii). Finally, the definition of the small components in \eqref{eq: small components}  together with \eqref{eq: kornpoinsharp5} implies \eqref{eq:I}(iv).\\

\noindent \textit{Step II (Interface between large and small components).} We now show that there is a union of balls $R_2 \subset \Omega $  and a Caccioppoli partition $\bigcup^{I'}_{j=1} P^2_j$ of $\Omega_{\rm good} := \bigcup_{j=1}^{I'} (P^1_j \setminus R_2)$  and corresponding infinitesimal rigid motions $(a^2_j)_{j=1}^{I'}$ such that with $v_2 := u - \sum_{j =1}^{I'} a^2_j \chi_{P^2_j}$ we have
\begin{align}\label{eq:II}
\begin{split}
(i) & \ \ \mathcal{L}^2(\Omega \setminus \Omega_{\rm good}) \le c\theta({\cal H}^1(J_u'))^2,\\
(ii) & \ \ \mathcal{H}^1(\partial^* \Omega_{\rm good} \setminus J_u') \le c\theta\mathcal{H}^1(J_u'), \\
(iii) & \ \ \sum\nolimits_{j= 1}^{I'} \mathcal{H}^1(\partial^* P^2_j) \le c\mathcal{H}^1(J_u'),\\
(iv) &  \ \ \Vert v_2 \Vert_{L^\infty(\Omega_{\rm good}; \R^{2})} +  \Vert \nabla v_2 \Vert_{L^1(\Omega_{\rm good}; \R^{2 \times 2})} \le C_{\theta,\Omega}\mathcal{E}.
\end{split}
\end{align}
First, for each $1 \le i \le I'$ and $j >I'$ we apply Lemma \ref{lemma: interface} for $P_1 = P^1_i$ and $P_2 = P^1_j$ and obtain a ball $B_{i,j}$ with $\mathcal{H}^1( (\partial^* P_i^1 \cap \partial^* P_j^1) \setminus (B_{i,j} \cup J_u))=0$ such that    by \eqref{eq:I}(iii),(iv) 
\begin{align*}
\diam(B_{i,j}) &\le 16 C_{\rm korn} \,\diam(P_j^1) \cdot \theta^{-4K_\theta} \cdot (\theta^{-4(K_\theta+1)})^{-1} \le  \theta^3 \,\diam(P_j^1),
\end{align*}
where the last step follows from the fact that  $\theta \le \frac{1}{16}C_{\rm korn}^{-1}$. Then by Lemma \ref{lemma: diam} and the fact that $P_j^1$ is indecomposable (see below \eqref{eq:I}) we get $\diam(B_{i,j}) \le \theta^3 \mathcal{H}^1(\partial^* P_j^1)$.    

 Define $R_2 = \bigcup_{i \le I' < j} B_{i,j}$ and compute by   \eqref{eq:I}(ii) and $I' \le \theta^{-2}$ (cf. \eqref{eq:I}(i))
\begin{align}\label{eq:balls}
\sum\nolimits_{i \le I' < j} \mathcal{H}^1(\partial B_{i,j}) \le \theta^3 I'\sum\nolimits_{j>I'} \mathcal{H}^1(\partial^* P_j^1) \le c\theta  \mathcal{H}^1(J_u').
\end{align}
Then the isoperimetric inequality yields $\mathcal{L}^2(R_2) \le c\theta^2  (\mathcal{H}^1(J_u'))^2$ and this together with   \eqref{eq:I}(i) shows \eqref{eq:II}(i). Let $P^2_j = P^1_j \setminus R_2$ and $a^2_j = a^1_j$ for $1 \le j \le I'$. Then    \eqref{eq:II}(iii) follows from  \eqref{eq:I}(ii) and \eqref{eq:balls}. To see   \eqref{eq:II}(ii), we calculate by Theorem \ref{th: local structure}, \eqref{eq:R}  and \eqref{eq:balls} recalling that $\Omega_{\rm good} \cup \bigcup_{j > I'} (P^1_j \setminus R_2) \cup  (R_1 \setminus R_2)  \cup R_2$ is a partition of $\Omega$
\begin{align*}
\mathcal{H}^1(\partial^* \Omega_{\rm good} \setminus (J_u \cup \partial \Omega)) &\le \sum_{i \le I' < j}\Big( \mathcal{H}^1\big(  (\partial^* P^1_i \cap \partial^* P^1_j)  \setminus (J_u \cup B_{i,j})\big) + \mathcal{H}^1(\partial B_{i,j}) \Big)  \\
& \ \ \ + \mathcal{H}^1(\partial^* R_1) \le  0 + c\theta  \mathcal{H}^1(J_u') = c\theta  \mathcal{H}^1(J_u'). 
\end{align*}
Finally,  \eqref{eq:II}(iv) follows from \eqref{eq:I}(iii), the definition of $v_2$ and the fact that $K_\theta \le \theta^{-1}$. \\

 \noindent \textit{Step III (Interface between large components).} We now investigate the difference of the infinitesimal rigid motions $(a^2_j)_{j=1}^{I'}$. We show that there is a  union of balls  $R_3 \subset \Omega$ and a Caccioppoli partition $\Omega_{\rm good} \setminus R_3= \bigcup^{I''}_{i=1} P^3_i$ with $I'' \le I'$  and corresponding infinitesimal rigid motions $(a^3_i)_{i=1}^{I''}$ such that with $v_3 := u - \sum_{i =1}^{I''} a^3_i \chi_{P^3_i}$ we have
\begin{align}\label{eq:III}
\begin{split}
(i) & \ \ \mathcal{H}^1(\partial^* R_3) \le c\theta {\cal H}^1(J_u'), \ \ \ \ \mathcal{L}^2(R_3) \le c\theta^2({\cal H}^1(J_u'))^2,\\
(ii) & \ \ \sum\nolimits_{i= 1}^{I''} \mathcal{H}^1(\partial^* P^3_i \setminus J_u') \le c\theta\mathcal{H}^1(J_u'),\\
(iii) &  \ \ \Vert v_3 \Vert_{L^\infty(\Omega_{\rm good} \setminus R_3; \R^{2})} +  \Vert \nabla v_3 \Vert_{L^1(\Omega_{\rm good} \setminus R_3; \R^{2 \times 2})} \le C_{\theta,\Omega}\mathcal{E}.
\end{split}
\end{align}

In the following we  denote  the constant given in \eqref{eq:II}(iv) by $\bar{C} = \bar{C}(\theta,\Omega)$ to distinguish it from other generic constants $C_{\theta,\Omega}$. We introduce the set of indices $\mathcal{Z}_1 = \lbrace 1 \le j \le I': \diam(P^2_j) \le \theta^3 \mathcal{H}^1(\partial \Omega) \rbrace$ \footnote{The introduction of $\mathcal{Z}_1$ is only a technical point due to the fact that by the previous step some large components may have become small after cutting of $R_2$.} and let $\mathcal{Z}_2 = \lbrace (i,j): 1 \le i<j \le I', \ i,j \notin \mathcal{Z}_1 \rbrace$ be the collection of pairs with
\begin{align}\label{eq:Bdef}
\begin{split}
\max\nolimits_{k=i,j} \Vert a^2_i - a^2_j \Vert_{L^\infty(P^2_k;\R^2)} > \bar{C}\theta^{-5} \mathcal{E}.
\end{split}
\end{align}
Finally, let $\mathcal{Z}_3 = \lbrace (i,j): 1 \le i<j \le I', \ i,j \notin \mathcal{Z}_1, (i,j) \notin \mathcal{Z}_2 \rbrace$.

For each $j \in \mathcal{Z}_1$ we find a ball $B^1_j$ with $\mathcal{H}^1(\partial B_j^1) \le c\theta^3\mathcal{H}^1(\partial \Omega)$ and $P_j^2 \subset B^1_j$. Moreover,   by Lemma \ref{lemma: interface}  we find for each $(i,j) \in \mathcal{Z}_2$   a ball $B^2_{i,j}$ satisfying $\mathcal{H}^1\big( (\partial^* P_i^2 \cap \partial^* P_j^2) \setminus (B^2_{i,j} \cup J_u) \big) = 0$ and by \eqref{eq:II}(iv)
\begin{align*}
{\rm diam}(B^2_{i,j})& \le c \max_{k=i,j}\diam(P^2_k) \  \big( \max_{k=i,j} \Vert a^2_i - a^2_j \Vert_{L^\infty(P^2_k;\R^2)}\big)^{-1} \ \Vert v_2 \Vert_{L^\infty(\Omega_{\rm good}; \R^{2})}\\
& \le c\diam(\Omega)  (\bar{C}\theta^{-5} \mathcal{E})^{-1} \bar{C}\mathcal{E}  \le c\theta^5 \mathcal{H}^1(\partial \Omega),
\end{align*}
where in the last step ${\rm diam}(\Omega) \le \mathcal{H}^1(\partial \Omega)$ follows from the fact that  $\Omega$ is assumed to be  connected.

We define $R_3 = \bigcup_{j \in \mathcal{Z}_1} B^1_j \cup \bigcup_{(i,j) \in \mathcal{Z}_2} B^2_{i,j}$ and the fact that $ \# \mathcal{Z}_1 \le \theta^{-2}$, $\# \mathcal{Z}_2 \le I'(I'-1) \le \theta^{-4}$ yields 
\begin{align}\label{eq:balls2}
\sum\nolimits_{j \in \mathcal{Z}_1}\mathcal{H}^1(\partial B^1_{j}) +  \sum\nolimits_{(i,j) \in \mathcal{Z}_2} \mathcal{H}^1(\partial B^2_{i,j}) \le c\theta \mathcal{H}^1(\partial \Omega),
 \end{align}
 which together with the isoperimetric inequality gives \eqref{eq:III}(i). We now combine different components $(P^2_j)_{j=1}^{I'}$: we can find a decomposition  $\mathcal{I}_1 \dot{\cup} \ldots \dot{\cup} \mathcal{I}_{I''}$ of the indices  $\lbrace 1, \ldots, I' \rbrace \setminus \mathcal{Z}_1$  with the property that for each pair $i_1,i_2 \in \mathcal{I}_j$,  $i_1 < i_2$, we find a chain $i_1 =l_1 < l_2 < \ldots <l_n=i_2$ such that $(l_k,l_{k+1}) \in \mathcal{Z}_3$ for all $k=1,\ldots,n-1$.

Then we introduce a partition of $\Omega_{\rm good} \setminus  R_3$ consisting of the sets $P^3_i = \bigcup_{j \in \mathcal{I}_i} (P^2_j \setminus R_3)$, $1 \le i \le I''$. (Note that this is indeed a partition of $\Omega_{\rm good} \setminus  R_3$ since, by construction, $P_j^2 \subset \Omega_{\rm good}$ for $j \in \mathcal{I}_i$ and $P_j^2 \subset R_3$ for $j \in \mathcal{Z}_1$.) To see \eqref{eq:III}(ii), we now compute using the property of the balls $B^1_{j}$,  $B^2_{i,j}$, as well  as \eqref{eq:II}(ii), \eqref{eq:balls2} and Theorem \ref{th: local structure} 
\begin{align*}
\sum\nolimits_{i= 1}^{I''} \mathcal{H}^1(\partial^* P^3_i \setminus J'_u)& \le  \mathcal{H}^1(\partial^* \Omega_{\rm good} \setminus J_u') +  \sum\nolimits_{j \in \mathcal{Z}_1} \mathcal{H}^1(\partial B^1_j)\\
&  \  + \sum\nolimits_{(i,j) \in \mathcal{Z}_2} \Big( \mathcal{H}^1\big(  (\partial^* P^2_i \cap \partial^* P^2_j)  \setminus (  {B}^2_{i,j} \cup J_u \big) + \mathcal{H}^1(\partial {B}^2_{i,j}) \Big)  \\ & \le c\theta\mathcal{H}^1(J_u')  +  c\theta\mathcal{H}^1(\partial \Omega)   + 0\le c\theta\mathcal{H}^1(J_u').
\end{align*}
It remains to define $v_3$ and to show \eqref{eq:III}(iii).  Fix $(i,j) \in \mathcal{Z}_3$. Then by the fact that \eqref{eq:Bdef} does not hold and $\min_{k=i,j}\diam(P^2_k) \ge \theta^3\mathcal{H}^1(\partial \Omega) \ge \theta^3\diam(\Omega)$ a short calculation implies $\Vert a^2_i - a^2_j \Vert_{L^\infty(\Omega;\R^2)} \le C \mathcal{E}$ for some $C=C(\Omega,\theta,\bar{C})$. Then the triangle inequality together with $\#\mathcal{I}_j \le  I' \le  \theta^{-2}$ yields
$$\max\nolimits_{i_1,i_2 \in \mathcal{I}_j} \Vert a^2_{i_1} - a^2_{i_2} \Vert_{L^\infty(\Omega;\R^2)} \le C_{\theta,  \Omega}\mathcal{E}$$
for all $1 \le j \le I''$, which by Lemma \ref{lemma: rigid motion} implies $\max\nolimits_{i_1,i_2 \in \mathcal{I}_j} |A^2_{i_1} - A^2_{i_2}| \le C_{\theta,  \Omega}\mathcal{E}$. For each $P^3_j$,   $1 \le j \le I''$, we choose an infinitesimal rigid motion $a^3_j$, which coincides with an arbitrary $a^2_{i}$, $i \in \mathcal{I}_j$. Then \eqref{eq:III}(iii) follows from  \eqref{eq:II}(iv).\\

 \noindent \textit{Step IV (Conclusion).} We are now in a position to prove the assertion of the theorem. Suppose that the partition $(P^3_j)_{j=1}^{I''}$ is ordered and choose the smallest index $I$ such that $\mathcal{L}^2(P^3_{I+1}) \le (\theta \mathcal{H}^1(J_u'))^2$. Define $R_4 = \bigcup_{j=I+1}^{I''} P^3_j$ and compute by the isoperimetric inequality and \eqref{eq:III}(ii)
  $$\mathcal{L}^2(R_4) \le \theta\mathcal{H}^1(J_u') \sum_{j=I+1}^{I''} (\mathcal{L}^2(P^3_j))^{\frac{1}{2}} \le c\theta \mathcal{H}^1(J_u') \sum_{j=I+1}^{I''} \mathcal{H}^1(\partial^* P^3_j) \le c\theta (\mathcal{H}^1(J_u'))^2.$$  
Then we define $\Omega_{\rm bad} := (\Omega \setminus \Omega_{\rm good}) \cup (R_3 \cup R_4)$ and by   \eqref{eq:II}(i),(ii), \eqref{eq:III}(i),(ii) we get $\mathcal{H}^1(\partial^* \Omega_{\rm bad} \setminus J'_u) \le c\theta\mathcal{H}^1(J_u')$ and $\mathcal{L}^2(\Omega_{\rm bad}) \le c\theta(\mathcal{H}^1(J_u'))^2$.

We define $u^\theta \in SBV(\Omega;\R^2) \cap L^\infty(\Omega;\R^2)$ by $u^\theta = u \chi_{\Omega \setminus \Omega_{\rm bad}} + t_0 \chi_{\Omega_{\rm bad}}$ for some $t_0 \in \R^2$ such that  $\mathcal{L}^2(\lbrace u=t_0 \rbrace)=0$, which is possible since $u$ is measurable. With this, $\mathcal{L}^2(\lbrace u^\theta \neq u \rbrace \triangle \Omega_{\rm bad})=0$. Observe that the previous calculation yields  \eqref {eq: kornpoinsharp1}. Let $(P_i)_{i=0}^{I}$ be the  Caccioppoli partition consisting of the sets $P_0 =  \Omega_{\rm bad}$ and $P_i = P^3_i$ for $1 \le i \le I $. Set $a_i = a^3_i$ for $1 \le i \le I$ and $a_0 = t_0$. Then \eqref{eq: kornpoinsharp2}(iii) follows from \eqref{eq:III}(iii) and \eqref{eq:III}(ii) yields \eqref{eq: kornpoinsharp2}(i). Finally, the choice of the index $I$ together with the fact that $\mathcal{H}^1(J_u') \ge \mathcal{H}^1(\partial \Omega)$ implies   \eqref{eq: kornpoinsharp2}(ii). 
 \end{proof}

\begin{proof}[Proof of Remark \ref{rem:square}]  Let $Q_\lambda = x+ (0,\lambda)^2$ be given and $u \in GSBD^2(Q_\lambda)$. After translation we may assume $x=0$. Define $\bar{u} \in GSBD^2(Q_1)$ by $\bar{u}(x) = u(\lambda x)$ and also note that $\nabla \bar{u}(x) = \lambda \nabla u(\lambda x)$ and $\mathcal{H}^1(J_{\bar{u}}) = \lambda^{-1} \mathcal{H}^1(J_u)$. Applying the above theorem for $\bar{u}$ on $Q_1$ we obtain $\bar{u}^\theta \in SBV(Q_1;\R^2) \cap L^\infty(Q_1;\R^2)$ such that 
\begin{align*}
\begin{split}
(i) & \ \ 
{\cal L}^2(\lbrace \bar{u} \neq \bar{u}^\theta \rbrace)  \le c\theta ({\cal H}^1(J_{\bar{u}})+ \mathcal{H}^1(\partial Q_1))^2, \\
(ii) & \ \    {\cal H}^1((\partial^* \lbrace \bar{u} \neq \bar{u}^\theta \rbrace  \cap Q_1) \setminus J_{\bar{u}})  \le c\theta ({\cal H}^1(J_{\bar{u}})+ \mathcal{H}^1(\partial Q_1)),
\end{split}
\end{align*}
a (finite) Caccioppoli partition $Q_1 = \bigcup^{I}_{i=0} \bar{P}_i$, and corresponding infinitesimal rigid motions $(\bar{a}_i)_{i=0}^I$ such that  
 $\bar{v} := \bar{u}^\theta - \sum\nolimits_{i=0}^{I} \bar{a}_i \chi_{\bar{P}_i} \in SBV(Q_1;\R^2)\cap L^\infty(Q_1;\R^2)$ is constant on $\bar{P}_0$ and  satisfies
\begin{align*}
\begin{split}
(i) & \ \  \sum\nolimits_{i=0}^{I}{\cal H}^1( (\partial^* \bar{P}_i \cap Q_1) \setminus J_{\bar{u}} ) \le c\theta ({\cal H}^1(J_{\bar{u}}) + \mathcal{H}^1(\partial Q_1)),\\
(ii) & \ \ \mathcal{L}^2(\bar{P}_i) \ge C_{Q_1}\theta^2 = C_{Q_1}\theta^2 \mathcal{L}^2(Q_1), \ \ 1 \le i \le I,\\
(iii)& \ \ \Vert \bar{v} \Vert_{L^\infty(Q_1; \R^2)} + \Vert \nabla \bar{v} \Vert_{L^1(Q_1;\R^{2\times 2})} \le C_{\theta,Q_1} \Vert  e(\bar{u}) \Vert_{L^2(Q_1; \R_{\mathrm{sym}}^{2\times 2})}.
\end{split}
\end{align*}
Set $P_i = \lambda \bar{P}_i$, $u^\theta(x) = \bar{u}^\theta(\lambda^{-1}x)$ and $v(x) = \bar{v}(\lambda^{-1}x) \in SBV(Q_\lambda;\R^2)$. The estimates for the modification in \eqref{eq: kornpoinsharp1}  follow since the estimate in (i) is two homogeneous and the estimate in (ii) is one homogeneous. For the same reason \eqref{eq: kornpoinsharp2}(i) and \eqref{eq:local1-new}(i) hold. We finally show \eqref{eq:local1-new}(ii).

By transformation formula and the fact that $\nabla \bar{u}(x) = \lambda \nabla u(\lambda x)$ we have  $\Vert e(u) \Vert^2_{L^2(Q_\lambda)} = \Vert e(\bar{u}) \Vert^2_{L^2(Q_1)}$.  Likewise, $\Vert \nabla v \Vert_{L^1(Q_\lambda)} = \lambda\Vert \nabla \bar{v} \Vert_{L^1(Q_1)}$ and finally we clearly have $\Vert v \Vert_\infty = \Vert \bar{v} \Vert_\infty$. Then (ii) follows as $\diam(Q_\lambda) = \sqrt{2}\lambda\ge~ \lambda$. 
\end{proof}

\subsection{A version with Dirichlet boundary conditions}

We now state a version of the piecewise Korn inequality with Dirichlet boundary conditions, which will be needed for the general existence result in  Section \ref{sec:comp}, but not for the jump transfer lemma in Section \ref{sec:jump}.  The reader more interested in the derivation of the latter may therefore wish to skip the  remainder of this section and to proceed directly with Section \ref{sec:jump}.

\begin{theorem}\label{th: korn-boundary}
Let $\Omega\subset\Omega'$ be bounded domains in $\R^2$ with Lipschitz boundary  such that \eqref{eq: omega'} holds. Let  $\theta>0$.  Then there is a  constant $ \bar{c}=\bar{c}(\Omega,\Omega')>0$ and some $C_{\theta,\Omega'}=C_{\theta,\Omega'}(\theta,\Omega')>0$   such that for  each $w \in H^1(\Omega';\R^2)$  and $u \in GSBD^2(\Omega')$  with $u = w$ on $\Omega' \setminus \overline{\Omega}$ there is a modification $u^\theta \in SBV(\Omega';\R^2)$ satisfying 
 \begin{align}\label{eq: boundary modi}
 \begin{split}
(i) & \ \   {\cal L}^2(\lbrace u \neq u^\theta\rbrace) \le \bar{c}\theta ({\cal H}^1(J_u) + 1)^2, \ \ \ {\cal H}^1(J_{u^\theta} \setminus J_u) \le \bar{c}\theta ({\cal H}^1(J_u) + 1)\\
(ii) & \ \ \Vert e(u^\theta) \Vert^2_{L^2(\Omega;\R^{2\times 2}_{\rm sym})} \le  \Vert e(u) \Vert^2_{L^2(\Omega;\R^{2\times 2}_{\rm sym})}+ \Vert \nabla w\Vert^2_{L^2(\lbrace u \neq u^\theta\rbrace; \R^{2\times 2})},\\
 \end{split}
 \end{align} 
a Caccioppoli partition $\Omega' = \bigcup^\infty_{j=1} P_j$ and corresponding infinitesimal rigid motions $(a_j)_j = (a_{A_j,b_j})_j$ such that  $v := u^\theta - \sum\nolimits^\infty_{j=1} a_j \chi_{P_j} \in SBV(\Omega';\R^2) \cap L^2(\Omega';\R^2)$,  
and 
\begin{align}\label{eq: small set main-boundary}
(i) & \ \  \sum\nolimits_{j=1}^\infty {\cal H}^1( (\partial^* P_j \cap \Omega') \setminus J_u ) \le \bar{c}\theta ({\cal H}^1(J_u) + 1),\notag\\
(ii) & \ \ v = w \ \text{ on } \Omega' \setminus \overline{\Omega},\\
(iii) &\ \ \Vert v \Vert_{L^2(\Omega';\R^2)} + \Vert \nabla v \Vert_{L^1(\Omega';\R^{2 \times 2})}  \le C_{\theta,\Omega'} \Vert  e(u) \Vert_{L^2(\Omega';\R_{\mathrm{sym}}^{2\times 2})} + C_{\theta,\Omega'}\Vert  w\Vert_{H^1(\Omega'; \R^2)}.\notag
\end{align}
\end{theorem}

As a preparation, we need the following lemma.

\begin{lemma}\label{lemma: maggi}
Let $A \subset \R^2$ open, bounded with Lipschitz boundary. Then there exists $\delta=\delta(A)$ such that for all indecomposable sets $E \subset A$ with finite perimeter satisfying ${\cal H}^1(\partial^* E \cap A) \le \delta(A)$ one has either
\begin{align*}
(i)  \ \ {\cal L}^2(E) > \tfrac{1}{2} {\cal L}^2(A) \ \ \ \text{ or } \ \ \ 
(ii)  \ \ {\rm diam}(E) \le C_A {\cal H}^1(\partial^* E \cap A)
\end{align*}
for some constant $C_A$ only depending on $A$.
\end{lemma}

\Proof Fix $\eps>0$. By \cite{mazja} (see also \cite{baldi}) there is a constant $K=K(A)$ and a Borel set $B_\eps \subset \R^2$ with $B_\eps \cap A=E$ such that
${\cal H}^1(\partial^* B_\eps) \le K {\cal H}^1(\partial^* E \cap A) + \eps$. It is not restrictive to assume that $B_\eps$ is indecomposable as otherwise we simply take the component containing $E$. By the isoperimetric inequality we derive
$$\min \lbrace {\cal L}^2(B_\eps), {\cal L}^2(\R^2 \setminus B_\eps) \rbrace \le c({\cal H}^1(\partial^* B_\eps))^2  < \tfrac{1}{2} {\cal L}^2(A),$$
where the last inequality holds provided that $\delta= \delta(A,c)$ is small enough. If ${\cal L}^2(\R^2 \setminus B_\eps)  < \tfrac{1}{2} {\cal L}^2(A)$, we find
$${\cal L}^2(E) = {\cal L}^2(B_\eps \cap A) = {\cal L}^2(A) - {\cal L}^2(A\setminus B_\eps) \ge {\cal L}^2(A) - {\cal L}^2(\R^2\setminus B_\eps) > \tfrac{1}{2}{\cal L}^2(A)$$
and (i) holds. Otherwise, we particularly obtain   ${\cal L}^2(\R^2 \setminus B_\eps) = + \infty$ and ${\cal L}^2(B_\eps) < + \infty$. Since $B_\eps$ has finite perimeter, by an approximation argument we may assume that $B_\eps$ is bounded. As $B_\eps$ is also indecomposable, Lemma \ref{lemma: diam} yields  
$${\rm diam}(E) \le {\rm diam}({B}_\eps) \le {\cal H}^1(\partial^* B_\eps) \le K {\cal H}^1(\partial^* E \cap A) + \eps.$$
The claim follows with $\eps \to 0$. \eop

\begin{proof}[Proof of Theorem \ref{th: korn-boundary}]   By Theorem  \ref{th: kornpoin-sharp} applied with $\Omega'$ in place of $\Omega$ we obtain a Caccioppoli partition $(P'_i)_{i=0}^{I}$, corresponding $(a'_i)_{i=0}^{I}$ as well as $\bar{u}^\theta \in SBV(\Omega';\R^2)$  and   $v':= \bar{u}^\theta - \sum_{i=0}^I a'_i \chi_{P_i'} \in SBV(\Omega';\R^2) \cap L^\infty(\Omega';\R^2)$ such that \eqref{eq: kornpoinsharp1}-\eqref{eq: kornpoinsharp2} hold. Define $u^\theta  =   u\chi_{\Omega \setminus P'_0} + w\chi_{P'_{0}}$. Then \eqref{eq: boundary modi} follows directly from \eqref{eq: kornpoinsharp1} and \eqref{eq: kornpoinsharp2}(ii).

Let ${\cal P}' = (P_j')_{j=1}^I$. Let ${\cal P}_1 \subset {\cal P}'$ be the components completely contained in $\Omega$ and let ${\cal P}_2 \subset {\cal P}'$ be the components $P'_j$ satisfying ${\cal L}^2(P'_j \cap (\Omega' \setminus \overline{\Omega})) \ge \theta$. Moreover, we set ${\cal P}_3 = {\cal P}' \setminus ({\cal P}_1 \cup {\cal P}_2)$. We now define the partition ${\cal P}  = (P_j)^\infty_{j=1}$ consisting of the components
$$ \lbrace P_0' \rbrace \cup  {\cal P}_1 \cup {\cal P}_2 \cup \lbrace P_j' \cap \Omega: P_j' \in {\cal P}_3 \rbrace \cup \lbrace P_j' \setminus \overline{\Omega}: P_j' \in {\cal P}_3 \rbrace.$$
(Strictly speaking, the number of components is even finite.) For $P_1 := P_0'$ we let $a_1 = 0$. For  $P_j  = P_k' \in {\cal P}_1$ we set $a_j = a_k'$ and for  $P_j =  P_k'\in {\cal P}_2$ we set $a_j = 0$. If $P_j \in {\cal P}$ with $P_j = P_k' \cap \Omega$ for some $P_k' \in {\cal P}_3$, we let $a_j = a_k'$. Finally,  if $P_j \in {\cal P}$ with $P_j = P_k' \setminus \overline{\Omega}$ for some $P_k' \in {\cal P}_3$, we set $a_j = 0$.

Now define $v =  u^\theta  - \sum_{j=1}^{\infty} a_j \chi_{P_j}$. By construction we get  $v = w$ on $\Omega' \setminus \overline{\Omega}$.  It remains to confirm \eqref{eq: small set main-boundary}(i),(iii). To see (iii), we first note that, since $ u^\theta=v=w$ on the open Lipschitz set $\Omega' \setminus \overline{\Omega}$,   by  \eqref{eq: kornpoinsharp2}(iii) with $v'$ in place of $v$ and \cite[Corollary 3.89]{AFP}, it suffices to show that the restriction of $v$ to $\Omega$ belongs to  $SBV(\Omega; \R^2) \cap L^2(\Omega; \R^2)$ and that
\begin{equation}\label{eq: v-v'}
\Vert v-v' \Vert_{L^2(\Omega;\R^2)} + \Vert \nabla v-\nabla v'\Vert_{L^1(\Omega; \R^{2\times 2})} \le C_{\theta,\Omega'} \Vert  e(u) \Vert_{L^2(\Omega';\R_{\mathrm{sym}}^{2\times 2})} + C_{\theta,\Omega'}\Vert  w\Vert_{ H^1(\Omega'; \R^2)}.
\end{equation}
By construction we have that $\lbrace v \neq v' \rbrace \cap \Omega \subset (P'_0 \cap \Omega) \cup \bigcup_{P_j \in {\cal P}_2} P_j$ (up to a set of negligible measure).  First, \eqref{eq: v-v'} with $P'_0 \cap \Omega$ in place of $\Omega$ follows directly from \eqref{eq: kornpoinsharp2}(iii) and the fact that $v=w$ on $P'_0$. Fix $P_j \in {\cal P}_2$. We first observe that $u = \bar{u}^\theta$ by \eqref{eq: kornpoinsharp2}(ii) and thus $(v-v')\chi_{P_j}=(u-v')\chi_{P_j}=a'_k\chi_{P_j}$ with $k$ such that $P_j = P'_k$. Since $u=w$ on $\Omega' \setminus \overline{\Omega}$, we then deduce
$$a'_{k}\chi_{(\Omega' \setminus \overline{\Omega}) \cap P_j}=(w-v')\chi_{(\Omega' \setminus \overline{\Omega}) \cap P_j}$$
and therefore
$$\Vert {a'_k}\Vert_{L^2((\Omega' \setminus \overline{\Omega}) \cap P_j;\R^2)} \le  \Vert w\Vert_{L^2(\Omega' \setminus \overline{\Omega};\,\R^2)} + \Vert v'\Vert_{L^2(\Omega' \setminus \overline{\Omega};\,\R^2)}.$$
Consequently, using ${\cal L}^2(P_j \cap (\Omega' \setminus \overline{\Omega})) \ge \theta$,  \eqref{eq: kornpoinsharp2}(iii) and Lemma \ref{lemma: rigid motion} for $\psi(s) =s^2$ we find 
$$| A_k'| + |b_k'| \le   C_{\theta,\Omega'} \Vert  e(u) \Vert_{L^2(\Omega';\R_{\mathrm{sym}}^{2\times 2})} + C_{\theta,\Omega'}\Vert  w\Vert_{L^2(\Omega'; \R^2)}.$$
Since $\# {\cal P}_2 \le \theta^{-1}{\cal L}^2(\Omega') = C(\Omega',\theta)$, this yields
$$\sum\nolimits_{P_j \in {\cal P}_2}\Vert a'_k \Vert_{L^2(P_j;\R^2)} +  \Vert  A'_k \Vert_{L^1(P_j;\R_{\mathrm{skew}}^{2\times2})} \le C_{\theta,\Omega'} (\Vert  e(u) \Vert_{L^2(\Omega';\R_{\mathrm{sym}}^{2\times 2})} + \Vert  w\Vert_{L^2(\Omega'; \R^2)})\,,$$
where for each $j$ the index $k$ is chosen such that $P_j = P'_k$. This implies $v \in SBV(\Omega; \R^2) \cap L^2(\Omega; \R^2)$, as well as \eqref{eq: v-v'}, and establishes \eqref{eq: small set main-boundary}(iii).

We now show \eqref{eq: small set main-boundary}(i).   To this end, we fix $\theta_0 = \theta_0(\Omega,  \Omega')>0$ to be specified below and we first observe that it suffices to treat the case where $ {\cal H}^1(J_u) + \mathcal{H}^1(\partial \Omega') \le \theta_0 \theta^{-1}$. In fact, otherwise \eqref{eq: small set main-boundary}(i) follows directly from  \eqref{eq: kornpoinsharp2}(i)  for $\bar{c}=\bar{c}(\Omega,\Omega')$ large enough.

Without restriction we suppose that each $P_j' \cap (\Omega' \setminus \overline{\Omega})$, $P_j'\in {\cal P}_3$, is indecomposable as otherwise we consider the indecomposable components. We show that each  $P_j' \cap (\Omega' \setminus \overline{\Omega})$ is contained in some ball of diameter $\bar{C} {\cal H}^1(\partial^* P'_j \cap (\Omega' \setminus \overline{\Omega}))$ for $\bar{C} = \bar{C}(\Omega,\Omega')$ large enough. To see this, we first observe that due to the fact that $J_u \subset \overline{\Omega}$ we have 
$$
{\cal H}^1(\partial^* P'_j \cap (\Omega' \setminus \overline{\Omega})) \le c \theta \mathcal{H}^1(J_u) + c\theta{\cal H}^1(\partial \Omega') \le c\theta_0
$$
by  \eqref{eq: kornpoinsharp2}(i).  Choose $\theta_0$ so small that $c\theta_0 \le \delta(\Omega' \setminus \overline{\Omega})$ with $\delta(\Omega' \setminus \overline{\Omega})$ as in Lemma \ref{lemma: maggi}. Then Lemma \ref{lemma: maggi} and the fact that ${\cal L}^2(P_j' \cap (\Omega' \setminus \overline{\Omega}) ) \le \theta$ imply for $\theta$ small 
\begin{align}\label{eq: maggi1}
{\rm  diam}(\partial^* P'_j \cap (\Omega' \setminus \overline{\Omega})) \le  \bar{C} {\cal H}^1(\partial^* P'_j \cap (\Omega' \setminus \overline{\Omega})) \le c\bar{C}\theta_0.
\end{align}

We cover $\Theta := \partial (\Omega' \setminus \Omega)$ with sets $U_1, \ldots, U_n$  such that $U_i \cap \Theta$ is the graph of a Lipschitz function  for $i=1,\ldots,n$ and the sets pairwise overlap such that each ball with radius $c\bar{C}\theta_0$ and center in $\Theta$ is contained in one $U_i$ provided that $\theta_0$ is chosen sufficiently small. Consequently, recalling \eqref{eq: maggi1}, each $P_j' \cap (\Omega' \setminus \overline{\Omega})$ is contained in some $U_i$. Since $U_i \cap \Theta$ is the graph of a Lipschitz function $f_i$ and $P_j' \cap (\Omega' \setminus \overline{\Omega}) \subset \subset U_i$, it follows that  
$${\cal H}^1(\partial \Omega \cap P_j') \le {\rm Lip}_{f_i} {\rm  diam}(\partial^* P'_j \cap (\Omega' \setminus \overline{\Omega})) \le  \hat{C}\bar{C} {\cal H}^1(\partial^* P'_j \cap (\Omega' \setminus \overline{\Omega})),$$
where $\hat{C} = \max_i {\rm Lip}_{f_i}$. For the last inequality we again used \eqref{eq: maggi1}.  Finally, noting that $\bigcup_{j=1}^\infty \partial^* P_j \setminus \bigcup_{j=0}^I \partial^* P_j' \subset \bigcup_{P_j' \in {\cal P}_3} (\partial \Omega \cap P_j')$ we find using  \eqref{eq: kornpoinsharp2}(i) 
$$\sum\nolimits_{j=1}^\infty {\cal H}^1( (\partial^* P_j \cap \Omega') \setminus J_u ) \le (1 + \hat{C} \bar{C}) \sum\nolimits_{j=0}^I  {\cal H}^1( (\partial^* P'_j \cap \Omega') \setminus J_u ) \le  \bar{c}\theta ({\cal H}^1(J_u) + 1) $$
for $\bar{c} = \bar{c}(\Omega,\Omega')$ large enough. 
\end{proof}

\section{Jump transfer lemma in GSBD}\label{sec:jump}

In this section we prove a jump transfer lemma which will be essential for the stability of the static equilibrium condition in the derivation of the existence result (Theorem \ref{th: main}).

\begin{theorem}\label{th: JTransf}
Let $\Omega\subset\Omega'$ be bounded domains in $\R^2$ with Lipschitz boundary  such that \eqref{eq: omega'} holds. Let $\ell \in \N$ and  let $(w^l_n)_n \subset H^1(\Omega';\R^2)$ be bounded sequences for $l=1,\ldots,\ell$. Let $(u^l_n)_n$ be sequences in $GSBD^2(\Omega')$ and $u^l \in GSBD^2(\Omega')$ such that
\begin{align}\label{eq: assu}
\begin{split}
(i) & \ \ \Vert e(u^l_n)\Vert_{L^2(\Omega';\R_{\mathrm{sym}}^{2\times 2})} + {\cal H}^1(J_{u_n})\le M \ \ \ \text{for all } n \in \N, \\ 
(ii) & \ \ u^l_n \to u^l  \text{ in measure in } \Omega', \ \ \ u^l_n= w^l_n \text{ on } \Omega' \setminus \overline{\Omega},
\end{split}
\end{align}
for $l=1,\ldots, \ell$.  Then it exists a (not relabeled) subsequence of $n \in \N$ with the following property: For each $\phi \in GSBD^2(\Omega')$ there is   a sequence $(\phi_n)_n \subset GSBD^2(\Omega')$ with $\phi_n = \phi$ on $\Omega' \setminus \overline{\Omega}$ such that for $n \to \infty$
\begin{align}\label{eq: conv}
\begin{split}
(i)& \ \ \phi_n \to \phi \text{ in measure in } \Omega,\\
(ii)& \ \  e(\phi_n) \to e(\phi) \text{ strongly in } L^2(\Omega;\R_{\mathrm{sym}}^{2\times 2}),\\
(iii)& \ \ {\cal H}^1\big((J_{\phi_n} \setminus \bigcup\nolimits^\ell_{l=1}J_{u^l_n}) \setminus (J_{\phi} \setminus \bigcup\nolimits^\ell_{l=1} J_{u^l})\big) \to 0.
\end{split} 
\end{align}

\end{theorem}

\subsection{Proof of the jump transfer lemma}
The general strategy is to follow the proof  of the $SBV$ jump transfer  (see \cite[Theorem 2.1]{Francfort-Larsen:2003}) with the essential difference that (a) in the definition of $\phi_n$ we transfer the jump  not by a reflection but by a suitable extension and (b) the control of the derivatives, which is needed for the application of the coarea formula, is recovered from \eqref{eq: assu}(i) by means of the piecewise Korn inequality in  Theorem \ref{th: kornpoin-sharp}.  The auxiliary results, allowing us to overcome such difficulties, are the following Lemmas \ref{lemma: nitsche} and \ref{lemma: loc}, as well as Theorem \ref{th: approx}. We postpone their proofs to the next subsection and first show that with these additional techniques Theorem \ref{th: JTransf} can be derived following the lines of \cite{Francfort-Larsen:2003}.

For problem (a) we need the   following extension lemma, based on an argument of \cite{Nit}.
\begin{lemma}\label{lemma: nitsche}
Let $R\subset \R^2$ be an open rectangle, let $R^-$ be the reflection of $R$ with respect to one of its sides, and let $\hat{R}$ be the open rectangle obtained by joining $R$, $R^-$ and their common side. Let $\phi \in GSBD^2(R)$. Then it exists an extension $\hat{\phi}\in GSBD^2(\hat{R})$ of $\phi$  satisfying
\begin{align}\label{nitsche}
\begin{split}
(i)& \ \ {\cal H}^1(J_{\hat{\phi}})\le c {\cal H}^1(J_\phi)\\
(ii)& \ \ \Vert e(\hat{\phi})\Vert_{L^2(\hat{R};\R_{\mathrm{sym}}^{2\times 2})}\le c \Vert e(\phi)\Vert_{L^2(R;\R_{\mathrm{sym}}^{2\times 2})}
\end{split} 
\end{align}
for some universal constant $c$ independent of $R$ and $\phi$. 
\end{lemma}

We concern ourselves with problem (b). A key point in the proof of the jump transfer lemma is to write the jump set of  a limiting function $u^l$ as a countable union of pairwise intersections of boundaries of super-level sets by the $BV$ coarea formula. Thus, as a first ingredient we state that the jump set of an $GSBD^2$ function  can be approximated suitably  by the jump set of an $SBV$ function.

\begin{lemma}\label{lemma: approx new}
Let $\Omega' \subset \R^2$ open, bounded with Lipschitz boundary and $\epsilon >0$. For each  $u \in GSBD^2(\Omega')$ there is $v \in SBV(\Omega';\R^2) { \cap GSBD^2(\Omega')}$ with $\mathcal{H}^1(J_{u-v}) \le \epsilon$ and
$\mathcal{H}^1(J_{u}\triangle J_{v}) \le \epsilon$. 
If in addition there exist an open subset $\Omega \subset \Omega'$ and $w \in H^1(\Omega'; \R^2)$ with $u=w$ on $\Omega' \setminus \overline{\Omega}$, the function $v$ can be also taken with $v=w$ on $\Omega' \setminus \overline{\Omega}$.
\end{lemma}

Recall that the main assumption in the $SBV$ jump transfer lemma (see \cite[Theorem 2.1]{Francfort-Larsen:2003}) was that the derivatives $|\nabla u_n^l|$, $n \in \N$, are equiintegrable. Although Theorem \ref{th: kornpoin-sharp} allows us to reduce the problem to the $SBV$ setting, we need further arguments since Theorem \ref{th: kornpoin-sharp} only provides an $L^1$-bound for the derivatives and the bound is not given in terms of the displacement field,   but  holds only after subtraction of a piecewise infinitesimal rigid motion. 

To overcome this difficulty, given a fine covering of the jump set of  $u^l$, we have to construct explicitly  modifications of the functions $u_n^l$  on the given covering which have almost the same jump set and whose gradients are small. Notice that this differs substantially from the proof strategy devised in \cite{Francfort-Larsen:2003} where, due to equiintegrability, one could ensure a priori that gradients do not concentrate on small neighborhoods of the jump set of  $u^l$.

In the following, for $u \in GSBD^2$ and $x \in J_{u}$  with unit normal $\nu(x)$ we denote by $Q_r(x)$ the  square with sidelength $2r$, center $x$ and two faces perpendicular to $\nu(x)$.

\begin{definition}\label{def: fine}
{\normalfont
Let $u \in GSBD^2(\Omega)$, $x \in J_u$ and $\eta>0$. We say $r$ is an $\eta$-\emph{fine radius} and $Q_r(x)$ an $\eta$-\emph{fine square} of $u$ at $x$ if
there are two sets $B^r_+,B^r_- \subset Q_r(x)$ such that
\begin{align*}
\mathcal{L}^2(B_\pm^r) \ge \frac{1}{2}(1 - \eta)\mathcal{L}^2(Q_r(x)), \ \  \ \ \ \Vert u -  u^\pm(x)  \Vert_{L^\infty(B_\pm^r;\R^2)} \le \frac12\eta.
\end{align*}
For given $x$, $u$, and $\eta$ we set  $r(u,x,\eta)$ as the  maximal radius such that   $r$ is an $\eta$-fine radius of $u$ at $x$ for all  $r < r(u,x,\eta)$. Observe that both notions are well defined for almost every jump point and that  $r(u,x,\eta)> 0$ for $\mathcal{H}^1$-a.e. $x \in J_u$.
}
\end{definition}

We  have the following approximation result.  
\begin{theorem}\label{th: approx}
Let $\Omega' \subset \R^2$ open, bounded with Lipschitz boundary. Let $M>0$, $0 <\theta,\delta<1$ with $\delta  \le \frac{1}{4}C_{Q_1}\theta^{8}$ with the constant $C_{Q_1}$ from Remark \ref{rem:square}.   Consider a sequence $(u_n)_n$ in $GSBD^2(\Omega')$  and $u \in GSBD^2(\Omega')$  with 
\begin{align}\label{eq: assu-lemma}
\begin{split}
(i) & \ \ 
\Vert e(u_n)\Vert^2_{L^2(\Omega';\R_{\mathrm{sym}}^{2\times 2})} + {\cal H}^1(J_{u_n}) \le M \ \ \text{for all } n \in \N, \\ 
(ii)& \ \ u_n \to u  \text{ in measure in } \Omega'.
\end{split}
\end{align}
Let  $Q_* =  \bigcup_{i=1}^m  Q_{r_i}(x_i)$ be a union of pairwise disjoint $\delta$-fine squares for $u$ at $x_i \in J_u$ with   $\sum\nolimits_{i=1}^m r_i \le   M$ and $r_i\le \delta^2$. Then there exist a sequence  $(v^{\delta,\theta}_n)_n \subset SBV(Q_*;\R^2)$, a universal constant $c>0$, and $C_\theta=C_\theta(\theta)>0$ independent of the sequence $(u_n)_n$ and $\delta$, such that
\begin{align}\label{eq:cubesequence}
\begin{split}
(i) & \ \ \limsup_{n \to \infty}\mathcal{H}^1(J_{v^{\delta,\theta}_n} \setminus J_{u_n}) \le cM\theta,\\
(ii) & \ \   \limsup_{n \to \infty}\Vert \nabla v^{\delta,\theta}_n \Vert_{L^1(Q_*;\R^{2\times 2})} \le   C_\theta M \delta,\\
(iii)  & \ \ \liminf_{n\to \infty}\Vert v^{\delta,\theta}_n - u \Vert_{L^1(Q_{r_i}(x_i) \setminus F_i;\R^2)} =0 \ \ \ \text{for} \  i=1,\ldots,m,
\end{split}
\end{align} 
\smallskip   
where $F_i$, $i=1,\ldots,m$, are Borel sets with 
\begin{align}\label{eq:cubesequence2}
\mathcal{L}^2(F_i) \le c\theta^2 (r_i^2 + \theta^2\liminf\nolimits_{n \to \infty} (\mathcal{H}^1(J_{u_n} \cap Q_{r_i}(x_i)))^2).
\end{align}
 
\end{theorem}
For the proof of the jump transfer lemma we will need the following extension to the case with boundary conditions.

\begin{corollary}\label{cor: approx}
Under the assumptions of Theorem \ref{th: approx}, suppose in addition that there is an open subset $\Omega$ of $\Omega'$ so that $u_n=w_n$ in $\Omega'\setminus \overline{\Omega}$, for a bounded sequence $(w_n)_n \subset H^1(\Omega';\R^2)$. Then the sequence $(v^{\delta,\theta}_n)_n$  can be taken such that $J_{v^{\delta,\theta}_n} \subset Q_*\cap \overline{\Omega}$, provided the constant $C_\theta$ is allowed to   additionally depend on $\sup_n \Vert w_n \Vert_{H^1(\Omega'; \R^2)}$. 

\end{corollary}

We now proceed with the proof of Theorem \ref{th: JTransf}.

\begin{proof}[Proof of Theorem \ref{th: JTransf}]
 We first consider the case $\ell =1$ and drop the superscript. Let $(u_n)_n$, $u$ and $\phi$ be given as in the hypothesis.  
 
 \smallskip

 \noindent \textit{Step 0.} We first show that it is not restrictive to assume that the limiting function $u$ additionally satisfies  $u\in SBV(\Omega';\R^2)$. Indeed, assume the theorem has been proved in this case. Then, in the general case of $u\in GSBD^2(\Omega')$,  for a fixed $\epsilon >0$ we choose $v_\epsilon \in SBV(\Omega';\R^2)  { \cap GSBD^2(\Omega')}$ with $v_\epsilon = u$ on $\Omega' \setminus \overline{\Omega}$  such that $\mathcal{H}^1(J_{u-v_{\epsilon}}) \le \epsilon$ and
$\mathcal{H}^1(J_{u}\triangle J_{v_{\epsilon}}) \le \epsilon$, thanks to Lemma \ref{lemma: approx new}.  We notice that the sequence  $v_{n, \epsilon}:= u_n+ (v_{\epsilon}-u)$   converges in measure to $v_\epsilon$ and satisfies the assumptions \eqref{eq: assu}.  Furthermore, 
\begin{equation}\label{eq: auxiliary}
\mathcal{H}^1(J_{v_{n,\epsilon}}\triangle J_{u_n}) \le \mathcal{H}^1(J_{v_{\epsilon}-u}) \le \epsilon.
\end{equation}
If we apply the theorem to the function $v_{\epsilon}$ and the the sequence  $(v_{n, \epsilon})_n$,  we find a sequence $\phi_{n,\epsilon}$ satisfying (i) and (ii) in \eqref{eq: conv} as well as
\[
\limsup_{n\to +\infty}{\cal H}^1\big((J_{\phi_{n,\epsilon}} \setminus J_{v_{n,\epsilon}}) \setminus (J_{\phi} \setminus J_{v_{\epsilon}})\big) =0\,.
\]
From   \eqref{eq: auxiliary} we then get
\[
\limsup_{n\to +\infty}{\cal H}^1\big((J_{\phi_{n, \epsilon}} \setminus J_{u_{n}}) \setminus (J_{\phi} \setminus J_{u})\big) \le 2\epsilon\,,
\]
whence the conclusion follows, by arbitrariness of $\epsilon$, through a diagonal argument.

\smallskip

 \noindent \textit{Step 1.} We now prove the theorem for $\ell=1$ and $u\in SBV(\Omega';\R^2)$, mainly following the proof of the $SBV$ jump transfer  (see \cite[Theorem 2.1]{Francfort-Larsen:2003}) employing additionally our auxiliary results.

Let $\theta>0$.  In the following all appearing generic constants $c$ are always independent of  $\theta$.  As a shortcut,  for $\lambda \in \R$ we introduce the scalar, auxiliary function $u^\lambda:=u^1+\lambda u^2$, where $u^1$ and $u^2$ denote the two components of the function $u$, respectively. We may fix $\lambda \in (0,1)$ such that
\begin{align}\label{eq:scalarfatto}
 \mathcal{H}^1(J_{u^\lambda} \triangle J_{u} ) = 0.
\end{align}
This follows from the fact that $A_\lambda:= \lbrace x \in J_{u}: [u^1(x)] + \lambda[u^2(x)] = 0 \rbrace$ satisfies ${\cal H}^1(A_\lambda) = 0$ except for a countable number of $\lambda$'s.
We further denote by $E_t$ the set of all Lebesgue-density 1 points for $\lbrace x \in  \Omega': u^\lambda(x) >t \rbrace$. Let $L = \lbrace t \in \R: {\cal L}^2(\lbrace x \in  \Omega': u^\lambda(x) = t\rbrace) = 0 \rbrace$. Then there is a countable, dense subset $D \subset L$ such that $J_{u^\lambda}$ (and thus, $J_u$) coincides up to a set of negligible ${\cal H}^1$-measure with
$$G :=  \bigcup_{t_1,t_2 \in D, t_1 < t_2} (\partial^* E_{t_1} \cap \partial^* E_{t_2} \cap \Omega').$$ 

For  $x \in G$ we can choose $t_1(x) < t_2(x)$ in $D$ such that $x \in \partial^* E_{t_1(x)} \cap \partial^* E_{t_2(x)}$ and $t_2(x) - t_1(x) \ge \frac{1}{2}|[u^\lambda(x)]|$. It can be shown that $\partial^* E_{t_1(x)}, \partial^* E_{t_2(x)}$ have a common outer unit normal $\nu(x)$. Let $N$ be the set of points, where $\partial \Omega$ is not differentiable.  We define 
\begin{align*}
G_j& = \Big\{ x \in G \setminus N:  |[u^\lambda(x)]| \ge \tfrac{1}{j}, \ \  \lim\nolimits_{r\to 0}\frac{{\cal H}^1( (J_{u} \setminus \partial^* E_{t_1(x)} ) \cap Q_r(x) )}{2r}  = 0\Big\},  
\end{align*}
where  $Q_r(x)$ is a square with sidelenth $2r$ and faces perpendicular to the normal $\nu(x)$. As in the proof of \cite[Theorem 2.1]{Francfort-Larsen:2003}, and recalling \eqref{eq:scalarfatto} we have that for fixed $\theta>0$ and $j=j(\theta)$ large enough 
\begin{align}\label{eq: jt3}
{\cal H}^1(J_{u} \setminus G_j)= {\cal H}^1(J_{u^\lambda} \setminus G_j) \le \theta.
\end{align}
We also fix the half squares 
\begin{align*}
Q_r^+(x) &:= \lbrace y \in Q_r(x): (y-x) \cdot \nu(x)>0 \rbrace,  \ \ \ \ Q_r^-(x) := Q_r(x) \setminus Q_r^+(x)
\end{align*}
and the one-dimensional faces  
\begin{align*}
H_r(x,s) = \lbrace y \in Q_r(x): (y-x) \cdot \nu(x)=s \rbrace, \ \ \ H_r(x) := H_r(x,0). 
\end{align*}
Let $\delta = \theta (2\sqrt{2}Mj C_{\theta})^{-1} \wedge \frac{1}{4}C_{Q_1}\theta^{8}$, with the constant $C_{\theta}$ from \eqref{eq:cubesequence}(ii) and $C_{Q_1}$ from Remark \ref{rem:square}.  Following    \cite[(2.3),(2.5)-(2.6)]{Francfort-Larsen:2003}  and covering $G_j$ using the Morse-Besicovitch Theorem (see e.g. \cite{fonseca}) we find a finite number of pairwise disjoint squares $Q_i := Q_{r_i}(x_i)$, $i=1,\ldots,m$, with $x_i \in J_u$, $r_i { <} r(u, x_i,\delta)\wedge \delta^2$ (cf. Definition \ref{def: fine}) such that 
\begin{align}\label{eq: jt12}
\begin{split}
(i)& \ \ {\cal L}^2\big(\bigcup\nolimits_{i=1}^m Q_{i}\big) < \theta, \ \ \  {\cal H}^1(G_j \setminus \bigcup\nolimits_{i=1}^m Q_i) < \theta, \\
(ii)& \ \ {\cal H}^1((J_\phi \setminus J_{u}) \cap  Q_i) \le \theta r_i,\\
(iii) & \ \ r_i\le {\cal H}^1(J_u \cap Q_i) \le 3r_i, \\
(iv) & \ \  {\cal H}^1\big( (J_{u} \setminus\partial^* E_{t_1(x_i)}) \cap Q_i \big) \le \theta r_i,\\
(v)& \ \  {\cal H}^1\big(\lbrace y \in \partial^* E_{t_1(x_i)} \cap Q_i: \dist(y,H_{r_i}(x_i)) \ge \tfrac{\theta}{2} r_i \rbrace \big) \le \theta r_i, \\
(vi) & \ \  {\cal L}^2\big(( E_{t_k(x_i)}\cap Q_i) \triangle Q_i^-)\big) \le \theta^2r_i^2, \ k=1,2,          \\
(vii) & \ \  Q_i \subset \Omega \text{ if } x_i \in \Omega, \ \ \ {\cal H}^1(\partial \Omega \cap Q_i) \le cr_i  \text{ if } x_i \in \partial \Omega,
\end{split}
\end{align}
where  $Q_i^-:=Q^-_{r_i}(x_i)$. { Recall that $r_i$ is a  $\delta$-fine radius of $u$ at $x_i$ in the sense of Definition \ref{def: fine} since $r_i < r(u, x_i,\delta)$.} Thus, each $Q_i$ is a $\delta$-fine square. By \eqref{eq: jt12}(iii) we can now apply Theorem \ref{th: approx} and Corollary \ref{cor: approx} to obtain a sequence $(v^{\delta, \theta}_n)_n \subset SBV(Q_*;\R^2)$ with $Q_* =  \bigcup_{i=1}^m Q_{i}$ satisfying \eqref{eq:cubesequence}, in particular we have 
\begin{align}\label{eq:small jumpXXX}
\mathcal{H}^1(J_{v^{\delta, \theta}_n} \setminus J_{u_n}) \le cM\theta, \ \ \ \ J_{v^{\delta, \theta}_n} \subset \overline{\Omega}.
\end{align}
For brevity we write $v_n$ instead of $v^{\delta, \theta}_n$ in the following. For the same $\lambda$ that we fixed in \eqref{eq:scalarfatto} we analogously define the scalar-valued auxiliary functions $v_n^\lambda$ and we denote by $E^n_t$ the set of all Lebesgue-density 1 points for $\lbrace x \in  \Omega': v_n^\lambda(x) >t \rbrace$. By construction, and applying  \eqref{eq:cubesequence}(ii) we obtain, recalling  $\delta \le  \theta (2\sqrt{2}Mj C_{\theta})^{-1}$ 
\[
\Vert \nabla v_n^\lambda \Vert_{L^1(Q_*;\R^{2})}\le \sqrt{2}\Vert \nabla v_n \Vert_{L^1(Q_*;\R^{2\times 2})} \le \sqrt{2}C_\theta M\delta  \le \tfrac{\theta}{2j}\,.
\]
In view of the coarea formula in $BV$ this implies  that there are $t_i \in [t_1(x_i),t_2(x_i)]$ with (see \cite[(2.7)]{Francfort-Larsen:2003})
\begin{align}\label{eq: coarea+}
\sum\nolimits_{i=1}^m {\cal H}^1\big( (\partial^* E^n_{t_i} \cap Q_i) \setminus J_{v_n^\lambda}\big) \le  \theta.
\end{align}
By construction it holds that $J_{v_n^\lambda}\subset J_{v_n}$: combining with \eqref{eq:small jumpXXX}, we deduce 
\begin{align}\label{eq: coarea}
\sum\nolimits_{i=1}^m {\cal H}^1\big( (\partial^* E^n_{t_i} \cap Q_i) \setminus J_{u_n}\big) \le  (1+cM)\theta.
\end{align}

We now denote with $\mathcal{I} \subset \lbrace 1,\ldots,m\rbrace$ the subset of \emph{good squares} such   that  
\begin{align}\label{eq:small jump}
\liminf\nolimits_{n \to \infty}\mathcal{H}^1(J_{u_n} \cap Q_i)  \le \theta^{-1} r_i 
\end{align}
if and only if $i \in  \mathcal{I}$.  For $t \in L$, \eqref{eq:cubesequence}(iii), \eqref{eq:cubesequence2}, and \eqref{eq:small jump} imply that 
$$\liminf\nolimits_{n \to \infty}  {\cal L}^2((E^n_t \triangle E_t) \cap Q_i) \le c\theta^2r_i^2$$
for $i \in \mathcal{I}$. Then taking \eqref{eq: jt12}(vi) into account and following  \cite[(2.8)-(2.9)]{Francfort-Larsen:2003} we find $N(\theta)$ such that for $n \ge N(\theta)$
$${\cal L}^2((E^n_{t_i} \cap Q_i ) \triangle Q_i^-)  + {\cal L}^2((E_{t_i} \cap Q_i ) \triangle Q_i^-) \le c\theta^2r_i^2. $$
Following \cite[(2.10)-(2.14)]{Francfort-Larsen:2003} and using \eqref{eq: jt12}(i),(iii)-(v) we get $s^+_i,s^-_i \in [\frac{\theta}{2}r_i, \theta r_i]$ for $i \in \mathcal{I}$ such that for $n \ge N(\theta)$
\begin{align}\label{eq: new}
\begin{split}
(i) & \ \ {\cal H}^1(H_i^- \setminus E_{t_i}^n) \le  c \theta r_i, \ \ \ \ {\cal H}^1(H_i^+ \cap E_{t_i}^n) \le c \theta r_i, \ \ i \in \mathcal{I},\\
(ii) & \ \ {\cal H}^1\big(G_j \setminus \big(\bigcup\nolimits_{i\in \mathcal{I}} R_i  \cup \bigcup\nolimits_{i \notin \mathcal{I}} Q_i  \big) \big) \le c\theta,
\end{split} 
\end{align}
where $H_i^+ = H_{r_i}(x_i,s_i^+)$, $H_i^- = H_{r_i}(x_i,s_i^-)$ and $R_i$ the open rectangle between $H_i^+$ and $H^-_i$. 

On the \textit{bad squares} $Q_i$ with $i \notin \mathcal{I}$ the above estimates are not available. On the other hand, there is not much jump of $u$ in those squares. Indeed, 
since $\mathcal{H}^1(J_{u_n} \cap Q_i) > \theta^{-1} r_i$ for all $i \notin \mathcal{I}$ and $n\in \N$ large enough, we derive, using \eqref{eq: jt12}(iii)
\begin{align}\label{eq:badsquares}
\sum\nolimits_{i \notin \mathcal{I}} {\cal H}^1(J_u \cap Q_i) \le \sum\nolimits_{i \notin \mathcal{I}} 3 r_i \le 3\theta \sum\nolimits_{i \notin \mathcal{I}} \mathcal{H}^1(J_{u_n} \cap Q_i) \le 3\theta M.
\end{align}

Therefore, our aim is now to transfer the jump set $J_\phi$ in $G_j \cap \bigcup_{i \in \mathcal{I}} Q_i$ to $\bigcup_{i \in \mathcal{I}} (\partial^* E_{t_i}^n \cap Q_i)$. Assume first   $x_i \notin \partial \Omega$. We set $\phi_- = \phi \chi_{Q_i^- \setminus R_i}$ extended to $R_i$ according to Lemma \ref{lemma: nitsche}. (This is possible when $\theta$ is small enough since $R_i$ is a small neighborhood of $H_{r_i}(x_i)$.) In a similar way we define $\phi_+$ on $(Q_i \setminus Q_i^-) \cup R_i$.

Now we let
$$\phi_n = \begin{cases} \phi_- & \text{on } Q_i^- \setminus R_i, \\ \phi_+ & \text{on } Q_i \setminus (Q_i^- \cup R_i), \\ \phi_- & \text{on } R_i \cap E^n_{t_i}, \\ \phi_+ & \text{on } R_i \setminus E^n_{t_i}. \end{cases} $$
If $x_i \in \partial \Omega$, we proceed similarly using \eqref{eq: jt12}(vii) and modifying $\phi$ to $\phi_n$ only in the part contained in $\Omega$.\footnote{See again the proof of \cite[Theorem 2.1]{Francfort-Larsen:2003} for details. Let us just mention that in this context it is crucial that the jump sets of $J_{u}, J_{v_n}$ are contained in $\overline{\Omega}$, cf. \eqref{eq:small jumpXXX}, as hereby the  function has to be indeed only modified in $\overline{\Omega}$.}    We repeat the modification for all $Q_i$, $i \in \mathcal{I}$, so $\phi_n$ is defined on $\bigcup_{i \in \mathcal{I}} Q_i$. Outside this union we let $\phi_n = \phi$. By the construction and \eqref{nitsche} we observe
\begin{align}\label{eq: jjt}
\begin{split}
(i) & \ \ \lbrace \phi_n \neq \phi \rbrace \subset \big( \bigcup\nolimits_{i \in \mathcal{I}} R_i  \big) \cap \overline{\Omega},\\
(ii)& \ \ {\cal H}^1\big(J_{\phi_n} \cap \bigcup\nolimits_{i\in \mathcal{I}}  (R_i \setminus \partial^* E^n_{t_i})\big) \le c{\cal H}^1\big(J_\phi \cap \bigcup\nolimits_{i\in \mathcal{I}} (Q_i \setminus R_i)\big), \\
(iii)& \ \ \Vert e(\phi_n)\Vert_{L^2(\bigcup_{i \in \mathcal{I}} Q_i;\R^{2\times 2}_{\mathrm{sym}})} \le c \Vert e(\phi)\Vert_{L^2(\bigcup_{i \in \mathcal{I}} Q_i;\R^{2\times 2}_\mathrm{sym})}.
\end{split}
\end{align}
Taking a sequence $\theta_k \to 0$ generates a sequence $\phi_n$ by choosing $\phi_n$ as above using $\theta_k$ for $n \in  [N(\theta_k), N(\theta_{k+1}))$. With  \eqref{eq: jt12}(i) and \eqref{eq: jjt} we immediately deduce  \eqref{eq: conv}(i),(ii).

Finally, to see   \eqref{eq: conv}(iii) we again follow the argumentation in  \cite{Francfort-Larsen:2003} and refer therein for details. By \eqref{eq: jt3}, \eqref{eq: new}(ii), \eqref{eq:badsquares},   and \eqref{eq: jjt}(i) we find 
$${\cal H}^1\Big(\big((J_{\phi_n} \setminus J_{{u}_n}) \setminus (J_{\phi} \setminus J_{u})\big) \setminus  \big( \bigcup\nolimits_{i \in \mathcal{I}} \overline{R_i}   \big) \Big)\le O(\theta).$$
 Consequently, to conclude it suffices to show   
\begin{align}\label{eq: 9}
{\cal H}^1\Big( (J_{\phi_n} \setminus J_{u_n})  \cap \bigcup\nolimits_{i\in \mathcal{I}} \overline{R_i} \Big)\le  O(\theta)\,.
\end{align}
To this end, we consider $(J_{\phi_n} \setminus J_{u_n})  \cap \overline{R_i}$ for a fixed  $i \in \mathcal{I}$ and assume $x_i \in \Omega$ (the case $x_i \in \partial \Omega$ is similar). We break $\overline{R_i}$ into the parts
\begin{align*}
\overline{R_i} = \bigcup\nolimits^4_{k=1} P^k_i & := (\overline{R_i} \cap \partial^* E^n_{t_i}) \cup (R_i \setminus \partial^* E^n_{t_i})  \cup ( (H_i^+ \cup H_i^-) \setminus \partial^* E^n_{t_i}) \\ & \ \ \ \ \cup   (\partial R_i \setminus ({H}_i^+ \cup {H}_i^- \cup \partial^* E^n_{t_i}) ).
\end{align*}
First, by  \eqref{eq: coarea} we have 
$$\sum\nolimits_{i \in \mathcal{I}}{\cal H}^1(P^1_i \setminus J_{u_n}) \le O(\theta).$$
Moreover, by \eqref{eq: jt3}, \eqref{eq: jt12}(ii),(iii),  \eqref{eq: new}(ii), and \eqref{eq: jjt}(ii)  we derive
\begin{align*}
\sum\nolimits_{i \in \mathcal{I}}{\cal H}^1(P_i^2 \cap J_{\phi_n})&  \le c {\cal H}^1(J_\phi \cap \bigcup\nolimits_{i\in \mathcal{I}} (Q_i \setminus R_i)) \\ &\le c {\cal H}^1\big((J_\phi \cap  J_{u})  \cap \bigcup\nolimits_{i\in \mathcal{I}} (Q_i \setminus R_i)\big) + c{\cal H}^1((J_\phi \setminus  J_{u}) \cap \bigcup\nolimits_i Q_i) \\ &\le O(\theta).
\end{align*}
By our construction the only possible jumps of $\phi_n$ in $H_i^+ \cup H_i^-$ are $H_i^+ \cap E^n_{t_i}$ and $H_i^- \setminus E^n_{t_i}$ so that
$$\sum\nolimits_{i \in \mathcal{I}} {\cal H}^1(J_{\phi_n} \cap P_i^3) = \sum\nolimits_{i\in \mathcal{I}} \big({\cal H}^1(H_i^+ \cap E^n_{t_i}) + {\cal H}^1(H_i^- \setminus E^n_{t_i}) \big) \le O(\theta),$$
where the last inequality follows from \eqref{eq: jt12}(iii) and \eqref{eq: new}(i). Finally, the estimate $\sum_{i \in \mathcal{I}} {\cal H}^1(P_i^4) \le O(\theta)$ is a consequence of  \eqref{eq: jt12}(iii)  and $|s^+_i|,|s_i^-| \le \theta r_i$.  Collecting the previous estimates we obtain \eqref{eq: 9}. This concludes the proof for $\ell = 1$.

\smallskip

 \noindent \textit{Step 2.} In the general case $\ell >1$ it suffices to observe that the same trick in \eqref{eq:scalarfatto} can be inductively applied also to a finite number of $GSBD$ functions. If there are sequences $(u_n^l)_n$ in $GSBD^2(\Omega')$ and $u^l \in GSBD^2(\Omega')$, one can find a single  sequence $(\bar{u}_n)_n  \subset GSBD^2(\Omega')$ converging to some $\bar{u}$ in measure  with
\begin{align*}
{\cal H}^1\big(J_{\bar{u}_n} \triangle \bigcup\nolimits^\ell_{l=1} J_{u^l_n}\big) = {\cal H}^1\big(J_{\bar{u}} \triangle \bigcup\nolimits^\ell_{l=1} J_{u^l}\big) = 0.
\end{align*}
With this, we reduce the problem to a single sequence $(\bar{u}_n)_n$ for which the hypotheses of the theorem are satisfied for a suitable bounded sequence $(\bar{w}_n)_n$ of boundary data. 
\end{proof}

\subsection{Proof of the auxiliary results}
We begin with the Proof of Lemma \ref{lemma: nitsche}.

\begin{proof}[Proof of Lemma \ref{lemma: nitsche}]
We can assume $R=(-l, l)\times (0, h)$ with $l,h>0$ and $R^-=(-l, l)\times (-h, 0)$. For a given parameter $0<\xi <1$ and a distribution $T$ on $(-l, l)\times (0, \xi h)$  the symbol $T^{\xi}$ denotes the distribution on $R^-$  obtained by composition of $T$ with the diffeomorphism $(x,y)\to (x,-\frac1\xi y)$.
We first assume $\phi:=(\phi_1, \phi_2)$ is a regular displacement in the sense of \eqref{def: regular}. Given $0<\lambda<\mu<1$ and $p>0$ we set for all $(x,y)\in R^-$ 
\begin{align}\label{eq:constru}
\begin{split}
\hat{\phi}_1(x,y)=p\phi_1(x,-\lambda y)+(1-p)\phi_1(x,-\mu y)\,\\
\hat{\phi}_2(x,y)=-\lambda p\phi_2(x,-\lambda y)+(1+\lambda p)\phi_2(x,-\mu y)\,.
\end{split}
\end{align}
Furthermore, $\phi$ and $\hat{\phi}$ have by construction the same trace on the common boundary $(-l,l)\times \{0\}$ so that no jump occurs there. With this, \eqref{nitsche}(i) follows.
In order to show (ii), we calculate the component $(E\hat{\phi})_{12}$ of the symmetrized distributional gradient of $\hat{\phi}$. A direct computation gives
$$
2(E\hat{\phi})_{12}=-\lambda p(\partial_1 \phi_2+\partial_2 \phi_1)^\lambda +(1+\lambda p)(\partial_1 \phi_2)^\mu-\mu(1-p)(\partial_2 \phi_1)^\mu\,.
$$
Choosing $p=\frac{1+\mu}{\mu-\lambda}$ we get
$$
2(E\hat{\phi})_{12}=-\lambda p(\partial_1 \phi_2+\partial_2 \phi_1)^\lambda +(1+\lambda p)(\partial_1 \phi_2+\partial_2 \phi_1)^\mu\,.
$$
Taking the absolutely continuous parts with respect to the Lebesgue measure we derive that the $L^2$ norm of $(e(\hat{\phi}))_{12}$ can be controlled with the $L^2$ norm of $(e(\phi))_{12}$ independently of $R$ and $\phi$, which was the only thing to be shown to get (ii).

Before coming to the general case, we notice that the function $\hat{\phi}$ has the following property: If $\psi:[0,+\infty)\to [0,+\infty)$ is an increasing continuous subadditive function satisfying \eqref{coerc} and
$
\int_R \psi(|\phi|)\,\mathrm{d}x \le 1,
$
then
\begin{align}\label{nitsche+}
\int_{\hat{R}} \psi(|\hat{\phi}|)\,\mathrm{d}x \le c
\end{align}
again for an absolute constant $c$ independent of $R$ and $\phi$. Indeed, it follows from the construction and the properties of $\psi$ that \eqref{nitsche+} holds for a constant $c$ only depending on $\lambda$, $\mu$, and $p$.

In the general case $\phi\in GSBD^2(R)$ we consider an approximating sequence of displacements with regular jump set $(\phi_k)_k$ in the sense of Theorem \ref{th: density}. (Again the reader willing to assume an $L^2$-bound may replace Theorem \ref{th: density} by Theorem \ref{th: density-l2}.) It follows by Remark \ref{meas-conv} that there exists a nonnegative concave (thus, continuous and subadditive) increasing function $\psi$ satisfying \eqref{coerc} and such that
$$
\int_R \psi(|\phi_k|)\,\mathrm{d}x \le 1
$$
for all $k \in \N$. To the functions $\phi_k$ we associate extensions $\hat{\phi}_k \in GSBD^2(\hat{R})$ satisfying \eqref{nitsche} and \eqref{nitsche+}. In particular, there is a constant $C$ indepdendent of $k$ such that
$$
\int_{\hat{R}} \psi(|\hat{\phi}_k|)+|e(\hat{\phi}_k)|^2\,\mathrm{d}x +  {\cal H}^1(J_{\hat{\phi}_k})\le C,
$$
so that \eqref{eq:  convergence sense} implies the existence of $\hat{\phi} \in GSBD^2(\hat{R})$ such that $\hat{\phi}_k\to \hat{\phi}$ in measure in  $\hat{R}$ and
\begin{equation}\label{semicont}
{\cal H}^1(J_{\hat \phi})\le \liminf_{k\to +\infty}{\cal H}^1(J_{\hat{\phi}_k})\, \quad \Vert e(\hat \phi)\Vert_{L^2(\hat{R};\R_{\mathrm{sym}}^{2\times 2})}\le \liminf_{k\to +\infty} \Vert e(\hat{\phi}_k)\Vert_{L^2(\hat{R};\R_{\mathrm{sym}}^{2\times 2})}\,.
\end{equation}
Passing to the limit and using \eqref{semicont},  Theorem \ref{th: density}, and the corresponding inequalities for $\phi_k$ we get \eqref{nitsche}.\footnote{Notice that by the explicit construction of $\hat{\phi}_k$ in \eqref{eq:constru}  and the convergence in measure of $\phi_k$ one can also show that  $\phi$ and $\hat{\phi}$  have the same trace on the common boundary $(-l,l)\times \{0\}$.}
\end{proof}

We go further by proving Lemma \ref{lemma: approx new}.
\begin{proof}[Proof of Lemma \ref{lemma: approx new}]
We apply Theorem \ref{th: korn} to $u$ and find a Caccioppoli partition $\Omega' = \bigcup^\infty_{j=1} P_j$ and infinitesimal rigid motions $(a_j)_{j=1}^\infty$ such that  $u - \sum_{j\ge 1} a_j \chi_{P_j}$ lies in $SBV(\Omega';\R^2) \cap L^\infty(\Omega';\R^2)$.  Using  $\sum\nolimits_{j=1}^{+\infty}{\cal H}^1\left(\partial^* P_j\right)<+\infty$
and  Theorem \ref{th: local structure}, we choose $j_0$ as the smallest index $j$ such that
\begin{align}\label{eq: j0}
{\cal H}^1\big(\bigcup\nolimits_{j \ge j_0} \partial^* P_j\big) + {\cal H}^1\big(J_u \cap \bigcup\nolimits_{j \ge j_0} (P_j)^1\big) \le \epsilon.
\end{align}
Then we define $\Omega_{\rm good} = \bigcup\nolimits_{j=1}^{j_0 -1}P_j$ and see that the function $v := u\chi_{\Omega_{\rm good}}$ lies in  $SBV(\Omega';\R^2) { \cap GSBD^2(\Omega')} \cap L^\infty(\Omega';\R^2)$. Since by construction we have
\[
J_{u-v}\subseteq \bigcup\nolimits_{j \ge j _0} \partial^* P_j \, \cup \,  \left(J_u \cap \bigcup\nolimits_{j \ge j_0} (P_j)^1 \right),
\]
we get $\mathcal{H}^1(J_{u-v}) \le \epsilon$. As $J_{u}\triangle J_{v} \subseteq J_{u-v}$, the first part of the statement follows.

For the second part, observe that, since $v=u\chi_{\Omega_{\rm good}}$, it holds $v = u$ on the set  $\{u = 0\}$. Therefore, if $w= 0$, the proof is concluded. In the general case we set $\hat{u} = u - w$ and apply the above procedure to $\hat{u}$ to construct a function $\hat{v}$ with $\mathcal{H}^1(J_{\hat u-\hat v}) \le \epsilon$ and $\hat v=0$ in $\Omega'\setminus \overline{\Omega}$, since $\hat u=0$ there. We then conclude setting $v=\hat v +w$.
\end{proof}

For the proof of Theorem \ref{th: approx} we need the following preliminary lemma, which is a consequence of the piecewise Korn inequality in Theorem \ref{th: kornpoin-sharp}.

\begin{lemma}\label{lemma: loc}
Let  $\theta, \delta>0$ with $\delta \theta^{-2} \le \frac{1}{4}C_{Q_1}$ with the constant $C_{Q_1}$ from Remark \ref{rem:square}. Consider a square $Q \subset \R^2$  and $u \in GSBD^2(Q)$, and assume that there are two sets $B_1, B_2 \subset Q$ and $t_1,  t_2 \in \R^2$ with 
\begin{align}\label{eq:locassum}
\mathcal{L}^2(B_m) \ge (\tfrac{1}{2}-\delta) \mathcal{L}^2(Q), \  \ \ \ \Vert u - t_m \Vert_{L^\infty(B_m; \R^2)} \le \delta, \ \ \  \  m=1,2. 
\end{align} 
 Then there is a universal constant $c>0$, some $C_\theta=C_\theta(\theta)>0$, both independent of $Q$ and $u$, and a modification $u^\theta \in SBV(Q;\R^2) \cap L^\infty(Q;\R^2)$ such that $u^\theta$ is constant on $\lbrace u \neq u^\theta \rbrace$ and
\begin{align}\label{eq: localization}
\begin{split}
(i) & \ \  {\cal L}^2(\lbrace u \neq u^\theta \rbrace)  \le c\theta ({\cal H}^1(J_u\cap Q)+ \diam(Q))^2, \\
(ii) & \ \  {\cal H}^1(\partial^*\lbrace u \neq u^\theta\rbrace  \setminus J_u)\le c\theta ({\cal H}^1(J_u\cap Q)+ \diam(Q)),\\
(iii) & \ \  \Vert \nabla u^\theta \Vert_{L^1(Q;\R^{2\times 2})} \le  C_\theta \, \diam(Q)\big( \Vert  e(u) \Vert_{L^2(Q; \R_{\mathrm{sym}}^{2\times 2})} +\delta\big),\\
(iv) & \ \  \Vert u^\theta \Vert_{L^\infty(Q;\R^{2})}  \le  C_\theta \big( \Vert  e(u) \Vert_{L^2(Q; \R_{\mathrm{sym}}^{2\times 2})} +\delta\big) + c(t_1+t_2).
\end{split}
\end{align}
\end{lemma}

\begin{rem}\label{rem:scaling}
The essential point is that  \eqref{eq: localization}(iii), differently from \eqref{eq: kornpoinsharp2}(iii), holds now for the modification $u^\theta$,   coinciding with $u$ outside a small set. Moreover, the estimate for $\nabla u^\theta$ scales with the diameter of the square which is fundamental for the proof of \eqref{eq:cubesequence}(ii).  
\end{rem}

\begin{proof}[Proof of Lemma \ref{lemma: loc}]
   We apply Theorem \ref{th: kornpoin-sharp} and obtain $u^\theta \in SBV(Q;\R^2) \cap L^\infty(Q;\R^2)$ as well as $v  =  u^\theta - \sum\nolimits_{j=0}^{I} a_j \chi_{P_j}\in SBV(Q;\R^2)$ for a partition $(P_j)_{j=0}^I$ and infinitesimal rigid motions $(a_j)_{j=0}^I$  such that  \eqref{eq: kornpoinsharp1}, \eqref{eq: kornpoinsharp2}(i) hold. Recall that  $P_0=\lbrace u \neq u^\theta \rbrace$ (see \eqref{eq: kornpoinsharp2}(ii)) and that $u^\theta$ can be defined constantly on $P_0$.   Now \eqref{eq: localization}(i),(ii) follow from \eqref{eq: kornpoinsharp1}. From Remark \ref{rem:square} we get 
\begin{align}\label{eq:local1}
\begin{split}
(i) & \ \ \mathcal{L}^2(P_j) \ge C_{Q_1}\mathcal{L}^2(Q)\theta^2 \ \ \text{ for } \ \ 1 \le j \le I, \\
(ii) &  \ \   \Vert v \Vert_{L^\infty(Q; \R^2)} + (\diam(Q))^{-1} \Vert \nabla v \Vert_{L^1(Q;\R^{2\times 2})} \le C_{\theta, Q_1} \Vert  e(u) \Vert_{L^2(Q; \R_{\mathrm{sym}}^{2\times 2})}.
\end{split}
\end{align}
By \eqref{eq:locassum}, \eqref{eq:local1}(i) and the assumption that $\delta \theta^{-2} \le \frac{1}{4}C_{Q_1}$, we find 
\begin{align*}
\mathcal{L}^2( (B_1 \cup B_2) \cap P_j) &\ge \mathcal{L}^2(P_j) - \mathcal{L}^2(Q \setminus (B_1 \cup B_2))\\
&\ge (C_{Q_1}\theta^2 - 2\delta)\mathcal{L}^2(Q) \ge \frac{1}{2} C_{Q_1}\theta^2 \mathcal{L}^2(Q)
\end{align*}
for each $P_j$, $1 \le j \le I$. Fix now $1 \le j \le I$. By the above argument it holds that $\max_{m=1,2} \mathcal{L}^2(B_{m} \cap P_j) \ge   \frac{1}{4} C_{Q_1}\theta^2 \mathcal{L}^2(Q)$. Assuming without loss of generality that the  maximum is achieved by $B_1$, we then have by the previous and the isodiametric inequality, Lemma \ref{lemma: rigid motion}, \eqref{eq:locassum}, and \eqref{eq:local1}(ii), that  
\begin{align*}
\theta \diam(Q) |A_j| &\le c \Vert a_j  -t_1\Vert_{L^\infty(P_j \cap B_1;\R^2)} \le  c \Vert v\Vert_{L^\infty(P_j \cap B_1;\R^2)} +  c \Vert u  -t_1\Vert_{L^\infty(P_j \cap B_1:\R^2)} \\
& \le  cC_{\theta, Q_1} \Vert  e(u) \Vert_{L^2(Q; \R_{\mathrm{sym}}^{2\times 2})} + c\delta
\end{align*} 
for $c = c(C_{Q_1})$ universal, where we used that $u=u^\theta$ on $P_j$.  Since $\mathcal{L}^2(P_j) \le (\diam(Q))^2$ and $\#I \le c\theta^{-2}$ by \eqref{eq:local1}(i),  we calculate by  \eqref{eq:local1}(ii)
\begin{align*}
\Vert \nabla u^\theta \Vert_{L^1(Q;\R^{2\times 2})} &= \Vert \nabla u \Vert_{L^1(Q \setminus P_0;\R^{2\times 2})} \le \Vert \nabla v \Vert_{L^1(Q;\R^{2\times 2})} + \sum\nolimits_{j=1}^I \mathcal{L}^2(P_j)|A_j| \\&\le C_\theta\diam(Q) \big( \Vert  e(u) \Vert_{L^2(Q; \R_{\mathrm{sym}}^{2\times 2})} +\delta\big). 
\end{align*}
This gives \eqref{eq: localization}(iii) and finally \eqref{eq: localization}(iv) can be seen along similar lines using that, for a fixed $1\le j \le I$, $\min_{m=1,2} \Vert a_j  -t_m\Vert_{L^\infty(Q;\R^2)} \le cC_{\theta, Q_1} \Vert  e(u) \Vert_{L^2(Q; \R_{\mathrm{sym}}^{2\times 2})} + c\delta$. 
\end{proof}

We now proceed with the proof of Theorem \ref{th: approx}.

\begin{proof}[Proof of Theorem \ref{th: approx}]
Let $Q_* := \bigcup_{i=1}^m Q_i$  be given as in the statement. By Definition \ref{def: fine}, for each $i$ we  find two subsets $B^+_{r_i}(x_i)$ and $B^-_{r_i}(x_i)$ of $Q_i$ such that
\begin{align}\label{eq: almost constant}
\begin{split}
{\cal L}^2(B^+_{r_i}(x_i)) \ge \frac{1}{2}(1-\delta\big){\cal L}^2(Q_i),  \ \  \ \ \ {\cal L}^2(B^-_{r_i}(x_i)) \ge \frac{1}{2}(1-\delta){\cal L}^2(Q_i),
\end{split}
\end{align}
where $\Vert u - u^+(x_i)\Vert_{L^\infty(B^+_{r_i}(x_i);\R^2)} \le \frac{1}{2}\delta$ and $\Vert u - u^-(x_i)\Vert_{L^\infty(B^-_{r_i}(x_i);\R^2)} \le \frac{1}{2}\delta$, respectively. For each $i$ and $n$ we set $B^\pm_{i,n} = \lbrace y \in Q_i: |u_n(y) - u^\pm(x)| \le \delta\rbrace$ and observe that due to \eqref{eq: almost constant} and the fact that $u_n \to u$ in measure,  we obtain for $n$ large enough
\[{\cal L}^2(B^+_{i,n}) \ge \Big(\frac{1}{2}-\delta\Big){\cal L}^2(Q_i),  \ \  \ \ \ {\cal L}^2(B^-_{i,n}) \ge \Big(\frac{1}{2}-\delta\Big){\cal L}^2(Q_i). 
 \]

Since $\delta \theta^{-8} \le \frac{1}{4}C_{Q_1}$, we can now apply Lemma \ref{lemma: loc} on the sequence $(u_n)_n$ and on each $Q_i$ with $\theta^4$ in place of $\theta$, with $B_1$ and $B_2$ being given by $B^+_{i,n}$, and $B^-_{i,n}$, respectively, $t_1= u^+(x_i)$, and $t_2=u^-(x_i)$. Therefore, we obtain a sequence of functions $v_n^{\delta, \theta, i} \in  SBV(Q_i;\R^2) \cap L^\infty(Q_i;\R^2)$ such that \eqref{eq: localization} holds (with $\theta^4$ in place of $\theta$). Recall that the constants in \eqref{eq: localization} are independent of $i$ and $n$. The functions $v^{\delta, \theta}_n$ are then defined as being given by $v_n^{\delta, \theta, i}$ on each of the disjoint squares $Q_i$.

By \eqref{eq: assu-lemma}(i) and \eqref{eq: localization}(ii) for each $i$ the perimeter of the sets $\lbrace v_n^{\delta, \theta, i} \neq u_n \rbrace$ is uniformly bounded in $n$ and by a compactness theorem for sets of finite perimeter together  with \eqref{eq: localization}(i) we thus obtain a set $F_i \subset Q_i$  such that $\chi_{\lbrace v_n^{\delta, \theta, i} \neq u_n \rbrace} \to \chi_{F_i}$ in measure as $n \to \infty$, after passing to a suitable  (not relabeled) subsequence. Therefore, we obtain by \eqref{eq: localization}(i) (again with $\theta^4$ in place of $\theta$) 
\[
\mathcal{L}^2(F_i) \le \liminf_{n \to \infty} \mathcal{L}^2(\lbrace v_n^{\delta, \theta, i} \neq u_n \rbrace)  \le c \liminf_{n \to \infty}\theta^4 ({\cal H}^1(J_{u_n} \cap Q_i)+ \diam(Q_i))^2,
\]
which implies \eqref{eq:cubesequence2}.  Notice that by construction, the sequence $(v_n^{\delta, \theta, i})_n$ converges to $u$ in measure on $Q_i \setminus F_i$.  By \eqref{eq: localization}(iv), $v_n^{\delta, \theta, i}$ is bounded uniformly in $L^\infty$ so that we deduce
\begin{align*}
\liminf_{n\to \infty}\Vert v_n^{\delta, \theta, i} - u \Vert_{L^1(Q_i \setminus F_i;\R^2)} = 0.
\end{align*}
Recall that, by Lemma \ref{lemma: loc}, $v_n^{\delta, \theta, i}$ are constant on the sets $\lbrace v_n^{\delta, \theta, i} \neq u_n \rbrace$, which therefore contain no jump of them. With this,  \eqref{eq: localization}(ii) together with $\sum_{i=1}^m r_i \le M$ yields \eqref{eq:cubesequence}(i). Finally, to confirm \eqref{eq:cubesequence}(ii), we use \eqref{eq: localization}(iii) to compute by H\"older's inequality and the fact that $r_i \le \delta^2$, $\sum_{i=1}^m r_i \le M$
\begin{align*}
\Vert \nabla v^{\delta, \theta}_n \Vert_{L^1(Q_*;\R^{2\times 2})}&= \sum_{i=1}^m \Vert \nabla v^{\delta, \theta, i}_n \Vert_{L^1(Q_i;\R^{2\times 2})}\le  C_\theta\sum_{i=1}^m r_i \left(\Vert   e(u_n) \Vert_{L^2(Q_i; \R_{\mathrm{sym}}^{2\times 2})} +\delta \right)\\
& \le C_\theta (\sum\nolimits_{i=1}^m r^2_i)^{\frac{1}{2}} \Vert  e(u_n) \Vert_{L^2(\Omega'; \R_{\mathrm{sym}}^{2\times 2})} + C_\theta\delta M \le C_\theta\delta M.
\end{align*}
\end{proof}

We conclude with the proof of Corollary \ref{cor: approx}.

\begin{proof}[Proof of Corollary \ref{cor: approx}]
Consider the functions $v^{\delta, \theta}_n$ constructed above. By \eqref{eq: localization}(ii) and the assumption $\sum\nolimits_{i=1}^m r_i \le   M$ it holds
\begin{equation}\label{eq: bd-conditions}
{\cal H}^1(\partial^*\lbrace v^{\delta, \theta}_n \neq u_n\rbrace  \setminus J_{u_n})\le c\theta \sum\nolimits_{i=1}^m ({\cal H}^1(J_{u_n} \cap Q_i)+ \diam(Q_i))\le 2Mc\theta.
\end{equation}
We now set
\[
\hat{v}^{\delta, \theta}_n:=(w_n-v^{\delta, \theta}_n)\chi_{\lbrace v^{\delta, \theta}_n \neq u_n\rbrace}+ v^{\delta, \theta}_n\,.          
\]
These functions satisfy $\hat{v}^{\delta, \theta}_n =  w_n$ on $Q_* \setminus \overline{\Omega}$ and thus $J_{v^{\delta, \theta}_n} \subset Q_*\cap \overline{\Omega}$. Since the sets $F_i$ were constructed as limits of $\lbrace v^{\delta, \theta}_n \neq u_n\rbrace \cap Q_i$, the new sequence still satisfies \eqref{eq:cubesequence}(iii). From $J_{\hat{v}^{\delta, \theta}_n}\setminus J_{v^{\delta, \theta}_n} \subset \partial^*\lbrace v^{\delta, \theta}_n \neq u_n\rbrace$ and \eqref{eq: bd-conditions} we get \eqref{eq:cubesequence}(i).

Finally, the assumptions $\sum\nolimits_{i=1}^m r_i \le   M$ and $r_i\le \delta^2$ imply that ${\cal L}^2(Q_*)\le 4\delta^2 M$, so that by H\"older's inequality  
$
\|\nabla w_n\|_{L^1(Q_*; \mathbb{R}^{2\times 2})}\le 2\delta\sqrt{M} \|\nabla w_n\|_{L^2(Q_*; \mathbb{R}^{2\times 2})}
$.
With this and the trivial inequality 
\[
\|\nabla \hat{v}^{\delta, \theta}_n\|_{L^1(Q_*; \mathbb{R}^{2\times 2})}\le \|\nabla v^{\delta, \theta}_n\|_{L^1(Q_*; \mathbb{R}^{2\times 2})}+\|\nabla w_n\|_{L^1(Q_*; \mathbb{R}^{2\times 2})}
\]
we get \eqref{eq:cubesequence}(ii) for a constant also depending on $\sup_n\Vert w_n \Vert_{H^1(\Omega';\R^2)}$.  
\end{proof}

\section{A general compactness and existence result}\label{sec:comp}

Notice that for the compactness theorem in $GSBD$  (see Theorem \ref{th: GSBD comp}) it is necessary that the integral for some integrand $\psi$ with $\lim_{s \to \infty} \psi(s) = \infty$ is uniformly bounded. However, in many application, e.g. in our model presented below, such an a priori bound is not available. Partially following ideas in \cite{Friedrich:15-2} we now show that by means of Theorem \ref{th: korn-boundary} it is possible to establish a compactness and existence result for suitably modified functions.

We first prove the following general compactness result.

\begin{theorem}\label{th: comp}
Let $\Omega \subset \Omega' \subset \R^2$ open, bounded with Lipschitz boundary such that \eqref{eq: omega'} holds. Let $M>0$, $w \in H^1(\Omega',\R^2)$ and $\Gamma$ be a rectifiable set with ${\cal H}^1(\Gamma) \le M$. Define
\begin{align}\label{eq: energy}
E(u) = \int_{\Omega'} Q(e(u)) \,\mathrm{d}x + {\cal H}^1(J_u \setminus \Gamma)
\end{align}
for $u \in GSBD^2(\Omega')$, where $Q$ is a positive definite quadratic form on $\R^{2 \times 2}_{\rm sym}$.\\
\smallskip
\noindent Then there is an increasing concave function $\psi:[0,\infty) \to [0,\infty)$ satisfying \eqref{coerc} only depending on $\Omega, \Omega', M$ such that for every sequence $(u_k)_k \subset GSBD^2(\Omega')$ with $\sup_{k\ge 1} E(u_k) \le M$ and $u_k = w$ on $\Omega' \setminus \overline{\Omega}$ we find a (not relabeled) subsequence and modifications $(y_k)_k \subset GSBD^2(\Omega')$ with $y_k = w$ on $\Omega' \setminus \overline{\Omega}$ and 
\begin{align}\label{eq: compi1}
  E(y_k) \le E(u_k) + \tfrac{1}{k}, \ \ \  \ \  \sup\nolimits_{k\ge 1}\int_{\Omega'} \psi(|y_k|)\,\mathrm{d}x \le 1.
\end{align}
Moreover, there is a function $y \in GSBD^2(\Omega')$ with $y = w$ on $\Omega' \setminus \overline{\Omega}$ such that $\int_{\Omega'} \psi(|y|)\,\mathrm{d}x \le 1$ and for $k \to \infty$
\begin{align}\label{eq: compi2}
\begin{split}
(i)&  \ \ y_k \to y \ \ \text{ in measure on } \Omega', \\
(ii)& \ \ e(y_k) \rightharpoonup e(y) \ \ \text{ weakly in } L^2(\Omega', \R^{2\times 2}_{\rm sym}),\\
(iii)& \ \ {\cal H}^1(J_y \setminus \Gamma) \le \liminf\nolimits_{k \to \infty} {\cal H}^1(J_{y_k} \setminus \Gamma).
\end{split}\end{align}

\end{theorem}

Note that properties \eqref{eq: compi2}(ii),(iii) also hold with $u_k$ in place of $y_k$. Moreover, observe that in general a passage to modifications is indispensable since the behavior on components completely detached from the rest of the body cannot be controlled.

\Proof Let be given a sequence $(u_k)_k$ with $E(u_k) \le M$ and $u_k = w$ on $\Omega' \setminus \overline{\Omega}$. This implies $\Vert e(u_k) \Vert^2_{L^2(\Omega';\R^{2 \times 2}_{\rm sym})} + {\cal H}^1(J_{u_k}) \le cM$ for all $k \in \N$. Let $\theta_l =  2^{-2l}$  for all $l\in \N$.  By Theorem \ref{th: korn-boundary} we find  functions $(v^l_k)_k \subset SBV(\Omega';\R^2) \cap L^2(\Omega';\R^2)$ of the form 
\begin{align}\label{eq: comp4}
v^l_k = u^l_k - \sum\nolimits_{j=1}^\infty a^{k,l}_j \chi_{P_j^{k,l}},
\end{align}
 where $u^l_k$ are  modifications,  $(P_j^{k,l})_j$ are partitions of $\Omega'$ and $(a_j^{k,l})_j$ infinitesimal rigid motions. In particular,  for all $l\in \N$, $k \in \N$ we have $v^l_k = w$ on $\Omega' \setminus \overline{\Omega}$ and by  \eqref{eq: boundary modi} the modifications satisfy 
\begin{align}\label{eq: comp1-new}
\begin{split}
(i) & \ \ \mathcal{L}^2(E_k^l) \le \bar{c}\theta_l, \ \ \ \ {\cal H}^1(J_{u^l_k} \setminus J_{u_k}) \le \bar{c}\theta_l, \\ 
(ii) & \ \   \Vert e(u^l_k) \Vert^2_{L^2(\Omega;\R^{2\times 2}_{\rm sym})} \le  \Vert e(u_k) \Vert^2_{L^2(\Omega;\R^{2\times 2}_{\rm sym})}+ \eps_l
\end{split}
\end{align}  
 for some $\bar{c} = \bar{c}(M,\Omega,\Omega')>0$, where $E_k^l:= \lbrace u \neq u_k^l \rbrace$ and $(\eps_l)_l$ is a null sequence only depending on $w$. Moreover, by \eqref{eq: small set main-boundary} we get
 \begin{align}\label{eq: comp1}
(i) \ \, \Vert v^l_k \Vert_{L^2(\Omega';\R^2)}  \le \hat{C}_l, \ \ (ii) \,  \  \Vert e(v^l_k) \Vert^2_{L^2(\Omega',\R^{2 \times 2}_{\rm sym})} \le cM, \, \ (iii) \ \ {\cal H}^1(J_{v_k^l} \setminus J_{u_k}) \le  \bar{c} \theta_l
\end{align}
for $\hat{C}_l = \hat{C}_l(\theta_l,  M,\Omega, \Omega')$. Here we used, possibly passing to a larger $M$, that $\eps_l\le M$ for all $l \in \N$.  Without restriction we  assume that $\hat{C}_l$ is increasing in $l$.

Using a diagonal argument we get a (not relabeled) subsequence of $(k)_{k \in \N}$ such that by Theorem \ref{th: GSBD comp} for every $l \in \N$ we find a function $v^l \in GSBD^2(\Omega')$  with $v^l_k \to v^l$ in $L^1(\Omega';\R^2)$ for $k\to \infty$ and 
\begin{align*}
e( u^l_k) = e(v^l_k) \rightharpoonup e(v^l) \text{ weakly in } L^2(\Omega';\R_{\mathrm{sym}}^{2\times 2}), \  {\cal H}^1(J_{v^l} \setminus \Gamma) \le \liminf_{k \to \infty}{\cal H}^1(J_{v_k^l} \setminus \Gamma).
\end{align*}
In particular, by \eqref{eq: comp1} we have
\begin{align}\label{eq: comp2}
\Vert v^l \Vert_{L^1(\Omega';\R^2)}  \le \hat{C}_l, \ \ \  \  \Vert e (v^l) \Vert^2_{L^2(\Omega';\R^{2 \times 2}_{\rm sym})} + {\cal H}^1(J_{v^l}) \le  cM + \bar{c}.
\end{align}
Likewise, we can establish a compactness result for the Caccioppoli partitions. By construction (see \eqref{eq: comp4}) and \eqref{eq: comp1}(iii) we have 
\begin{align}\label{eq: comp3}
\sum\nolimits_j {\cal H}^1(\partial^* P^{k,l}_j \cap \Omega') \le 2 {\cal H}^1(J_{u_k} \cup J_{v_k^l}) \le 2cM + 2\bar{c}
\end{align}
for all $k,l \in \N$. Thus, by Theorem \ref{th: comp cacciop} we find for all $l \in \N$ an (ordered) partition $(P_j^l)_j$ with  $\sum_j {\cal H}^1(\partial^* P^{l}_j \cap \Omega') \le 2cM + 2\bar{c}$ such that for a suitable subsequence one has $P^{k,l}_j \to P_j^l$ in measure for all $j \in \N$ as $k \to \infty$ and $\sum_j \mathcal L^2\left(P_j^{k,l} \triangle P_j^l\right)\to 0$ for $k \to \infty$.
As $\sum_j {\cal H}^1(\partial^* P^{l}_j \cap \Omega') \le 2cM + 2\bar{c}$ for all $l \in \N$, we can repeat the arguments and obtain a partition $(P_j)_j$ such that $\sum_j \mathcal L^2\left(P_j^{l} \triangle P_j\right)\to 0$ for $l \to \infty$ after extracting a suitable subsequence. Consequently, using a diagonal argument we can choose a (not relabeled) subsequence of $(l)_{l \in \N}$ and afterwards of $(k)_{k \in \N}$ such that
\begin{align}\label{eq: comp12}
\sum\nolimits_j \mathcal L^2\left(P_j^{l} \triangle P_j\right)\le 2^{-l}, \ \ \ \ \ \ \sum\nolimits_j \mathcal L^2\left(P_j^{k,l} \triangle P_j^l\right)\le 2^{-l} \ \ \text{ for all } k \ge l.
\end{align}
We now want to pass to the limit $l \to \infty$ for the sequence $(v^l)_l$.  However, we see that  the  compactness result in $GSBD$ cannot be  applied directly  as the $L^1$ bound depends on $\theta_l$ (cf. \eqref{eq: comp2}).   We  show that by choosing the infinitesimal rigid motions on the elements of the partitions appropriately (see \eqref{eq: comp4}) we can construct the sequence $(v^l)_l$ such that  
\begin{align}\label{eq: comp5}
\mathcal L^2\left(\bigcup\nolimits_{ m \ge n}\lbrace |v^n - v^{m}| > 1\rbrace\right)\le \hat{c} 2^{-n} \ \  \ \text{for all } n \in \N
\end{align}
for a constant $\hat{c} = \hat{c}(M, \Omega, \Omega')>0$, whence Lemma  \ref{rig-lemma: concave function2} is applicable.

We fix $k \in \N$ and describe an iterative procedure to redefine $a^{k,l}_j = a_{A^{k,l}_j,b_j^{k,l}}$ for all $l,j \in \N$. Let $\hat{v}^1_k = v^1_k $ as defined in \eqref{eq: comp4} and assume $\hat{v}^l_k$ as well as $(\hat{a}^{k,l}_j)_j$ have been chosen (which may differ from $(a^{k,l}_j)_j$)  such that \eqref{eq: comp1}(i) still holds possibly passing to a larger constant $\hat{C}_l$.  Fix some $P^{k,l+1}_{j}$, $j \in \N$.  If ${\cal L}^2(P^{k,l}_{j} \cap P^{k,l+1}_j) > 3\bar{c}\theta_{l}$, we define $\hat{a}^{k,l+1}_{j} = \hat{a}^{k,l}_{j}$ on $P^{k,l+1}_{j}$. Otherwise, we  set $\hat{a}^{k,l+1}_{j} = a^{k,l+1}_{j}$. In the first case we then obtain by the triangle inequality and the fact that $u_k^l = u$ on $\Omega' \setminus E_k^l$
\begin{align*}
\Vert \hat{a}^{k,l+1}_{j} -  a^{k,l+1}_{j}&  \Vert_{L^2( (P^{k,l}_{j} \cap P^{k,l+1}_j) \setminus (E_k^l \cup E_k^{l+1});\R^2)} \\
& \le \Vert u- \hat{a}^{k,l}_{j}\Vert_{L^2(\Omega' \setminus E_k^l;\R^2)} + \Vert u- {a}^{k,l+1}_{j}\Vert_{L^2(\Omega' \setminus E_k^{l+1};\R^2)} \\&\le \Vert \hat{v}_k^l\Vert_{L^2(\Omega';\R^2)} + \Vert v_k^{l+1}\Vert_{L^2(\Omega';\R^2)}   \le \hat{C}_l + \hat{C}_{l+1} \le 2\hat{C}_{l+1}.
\end{align*}
In the penultimate step we have used that \eqref{eq: comp1}(i) holds for $\hat{v}_k^{l}$ and $v_k^{l+1}$. By \eqref{eq: comp1-new}(i) we get $\mathcal{L}^2((P^{k,l}_{j} \cap P^{k,l+1}_j) \setminus (E_k^l \cup E_k^{l+1})) \ge \bar{c}\theta_l$. Consequently, by Lemma \ref{lemma: rigid motion} for $\psi(s) = s^2$ we find  $|\hat{A}^{k,l+1}_{j} -  A^{k,l+1}_{j}| + |\hat{b}^{k,l+1}_{j} -  b^{k,l+1}_{j}| \le C_*^{l+1}$  for a constant $C^{l+1}_*$ only depending on $\Omega$, $\Omega'$, $\hat{C}_{l+1}$, $\theta_{l}$ and $M$. We define $\hat{v}_k^{l+1}$ as in \eqref{eq: comp4} replacing $a_j^{k,l+1}$ by $\hat{a}_j^{k,l+1}$ and summing over all components we derive 
$$\Vert \hat{v}_k^{l+1} \Vert_{L^2(\Omega';\R^2)}  \le \Vert v_k^{l+1} \Vert_{L^2(\Omega';\R^2)} + C_{\Omega'}C^{l+1}_* \le \hat{C}_{l+1} + C_{\Omega'}C^{l+1}_*$$
for a constant $C_{\Omega'}$ depending only on $\Omega'$. I.e., \eqref{eq: comp1}(i) is also satisfied for $\hat{v}_k^{l+1}$ after possibly passing to a larger constant $\hat{C}_{l+1} = \hat{C}_{l+1}(\theta_{l+1},M,\Omega,\Omega')$.

For simplicity the modified functions and the infinitesimal rigid motions will still be denoted by $v_k^l$ and $a_j^{k,l}$ in the following. We now show that \eqref{eq: comp5} holds. To this end, we define $A^n_{k,l} = \bigcap_{n \le m \le l} \lbrace v^m_k =  v^n_k \rbrace$ for all $n \in \N$ and  $n \le l \le k$. If we show
\begin{align}\label{eq: comp8}
\mathcal L^2\left(\Omega' \setminus A^n_{k,l}\right) \le \hat{c}2^{-n},
\end{align}
then \eqref{eq: comp5} follows. Indeed, for given $l\ge n$ we can choose $K=K(l)\ge l$ so large that $\mathcal L^2\left(\lbrace |v^m_K - v^m| > \frac{1}{2} \rbrace\right) \le 2^{-m}$ for all $n \le m \le l$ since $v^m_k \to v^m$ in $L^1(\Omega';\R^2)$ for $k \to \infty$. This implies 
\begin{align*}
&\mathcal L^2\left(\bigcup\nolimits_{n \le m \le l} \lbrace |v^m -v^n| > 1 \rbrace\right)\\
&\le \mathcal L^2\left(\Omega' \setminus A^n_{K,l}\right) + \sum\nolimits_{n \le m \le l} \mathcal L^2\left(\lbrace |v^m_K - v^m| >  \tfrac{1}{2} \rbrace\right) \le \hat{c}2^{-n}.
\end{align*}
Passing to the limit $l \to \infty$ we then derive $\mathcal L^2\left(\bigcup\nolimits_{ m \ge n} \lbrace |v^{m} -v^n| > 1 \rbrace\right) \le \hat{c}2^{-n}$, as desired.

We now confirm \eqref{eq: comp8}.  To this end, fix $k \ge l$ and first observe that by \eqref{eq: comp4} and \eqref{eq: comp1-new}(i)
\begin{align}\label{lll}
\mathcal{L}^2\big( \bigcap\nolimits_{n \le m \le l} \lbrace T^n_k = T^m_k \rbrace \setminus A^n_{k,l}  \big) \le \sum\nolimits_{n \le m \le l} \mathcal{L}^2(E_k^m) \le 2\bar{c}\theta_n \le \bar{c}2^{-n},
\end{align}
where  $T^n_k = \sum_j a_j^{k,n}\chi_{P_j^{k,n}}$. We consider $\lbrace T^m_k = T^{m+1}_k \rbrace$ for $n \le m \le l-1$ and from \eqref{eq: comp12} we deduce $\sum_j \mathcal L^2\left(P_j^{k, m+1} \triangle P_j^{k,m}\right) \le 3 \cdot 2^{-m}$. Define  $J_1 \subset \N$ such that $\mathcal L^2\left(P_j^{k, m+1}\right) \le 6\bar{c}\theta_{m}$ for $j \in J_1$. Then let $J_2 \subset \N \setminus J_1$ such that  $\mathcal L^2\left(P^{k, m+1}_j \cap P^{k, m}_j\right) > \frac{1}{2} \mathcal L^2\left(P_j^{k, m+1}\right)$ for all  $j \in J_2$. Finally, we observe that $\mathcal L^2\left(P_j^{k, m+1}\right)\le 2\mathcal L^2\left(P_j^{k, m+1} \setminus P_j^{k, m}\right)$ for $j \in J_3:= \N \setminus (J_1 \cup J_2)$. Using the isoperimetric inequality and \eqref{eq: comp3} we derive 
\begin{align*}
\sum\nolimits_{j \in J_1} \mathcal L^2\left(P_j^{k,m+1}\right) &\le \sqrt{6\bar{c}\theta_{m}} \sum\nolimits_{j \in J_1} \mathcal L^2\left(P_j^{k,m+1}\right)^{\frac{1}{2}}\\
&\le c2^{-m} \sum\nolimits_{j \in J_1} {\cal H}^1(\partial^* P_j^{k,m+1}) \le c(M+\bar{c}) 2^{-m}.
\end{align*}
 Due to the above construction of the infinitesimal rigid motions we obtain $\lbrace T^m_k = T^{m+1}_k \rbrace \supset \bigcup_{j \in  J_2} (P^{k,m+1}_j \cap P^{k,m}_j)$  and therefore 
\begin{align*}
&\mathcal L^2\left(\Omega' \setminus \lbrace T^m_k = T^{m+1}_k \rbrace\right)  \le \sum_{j \in J_2} \mathcal L^2\left(P_j^{k, m+1} \setminus P_j^{k, m}\right) + \sum_{j \in J_1 \cup J_3} \mathcal L^2\left(P_j^{k, m+1}\right) \\
&\le \sum _{j \in J_2} \mathcal L^2\left(P_j^{k, m+1} \setminus P_j^{k, m}\right) + \sum_{j \in J_3} 2\mathcal L^2\left(P_j^{k, m+1} \setminus P_j^{k, m}\right) + c(M+\bar{c})2^{-m}    \le c2^{-m}
\end{align*}
for $c$ only depending on $M, \Omega, \Omega'$. Summing over $n \le m \le l-1$ and recalling \eqref{lll}, we establish \eqref{eq: comp8}  and consequently \eqref{eq: comp5}.

In view of \eqref{eq: comp2} and \eqref{eq: comp5}  we can apply Lemma \ref{rig-lemma: concave function2} on the sequences $s_l = \hat{C}_l$ and $t_l = \hat{c} 2^{-l}$ to obtain an increasing, concave function $\tilde{\psi}$ with \eqref{coerc} such that $\sup_{l \ge 1} \int_{\Omega'} \tilde{\psi}(|v^l|)\,\mathrm{d}x\le 1$.  Define $\psi(s) = \frac{1}{2} \min \lbrace \tilde{\psi}(s) ,s \rbrace$ and observe that $\psi$ has the desired properties. In particular, the choice of $\psi$ only depends on $\Omega,\Omega'$ and $M$. Recalling $v_k^l \to v^l$ in $L^1(\Omega';\R^2)$, \eqref{eq: comp1-new}  and \eqref{eq: comp1}(iii) we can now select a subsequence of  $(u_k)_k$ and a diagonal sequence $(y_k) \subset (v^l_k)_{k,l}$ such that $\Vert y_k - v^l \Vert_{L^1(\Omega';\R^2)} \le 1$ for some $v^l$ and  $ E(y_k) \le E(u_k) + \tfrac{1}{k}$. Then we get  that \eqref{eq: compi1} holds. 

The existence of a function $y \in GSBD^2(\Omega')$ with $y = w$ on $\Omega' \setminus \overline{\Omega}$ and $\int_{\Omega'}\psi(|y|)\,\mathrm{d}x \le 1$ as well as the convergence \eqref{eq: compi2} now  directly follow from Theorem \ref{th: GSBD comp}.  \eop

As a consequence we now obtain the following existence result.

\begin{theorem}\label{th: existence}
Let $\Omega \subset \Omega' \subset \R^2$ open, bounded with Lipschitz boundary such that \eqref{eq: omega'} holds. Let $w \in H^1(\Omega',\R^2)$ with $\Vert w\Vert_{H^1(\Omega';\R^2)} \le M$ and $E$ as given in \eqref{eq: energy}. Then the following holds:

\noindent (i) There is a minimizer of $E(u)$ among all functions $u \in GSBD^2(\Omega')$ with $u=w$ on $\Omega' \setminus \overline{\Omega}$. 

\noindent (ii) There is   an increasing concave function $\psi:[0,\infty) \to [0,\infty)$ with \eqref{coerc} only depending on $\Omega, \Omega', M$ such that $\int_{\Omega'} \psi(|u|)\,\mathrm{d}x \le 1$ for at least a minimizer $u$ of the minimization problem in (i). 
\end{theorem}

\Proof  Let  ${\cal A}:= \lbrace u \in GSBD(\Omega'): u = w \text{ on } \Omega' \setminus \overline{\Omega} \rbrace$ and $(u_k)_k \subset {\cal A}$ with $E(u_k) \to \inf_{u \in {\cal A}} E(u)$. We employ Theorem \ref{th: comp} and let $(y_k)_k$  be a (sub-)sequence of modifications converging to $u \in {\cal A}$ in the sense of \eqref{eq: compi2}. Then we find by \eqref{eq: compi1},\eqref{eq: compi2}
$$E(u) \le \liminf\nolimits_{k\to\infty} E(y_k) = \liminf\nolimits_{k\to\infty} E(u_k) = \inf\nolimits_{u \in {\cal A}} E(u).$$
Consequently, $u$ is a minimizer for the problem (i).  Moreover, by Theorem \ref{th: comp} we find a function $\psi$ with the desired properties such that $\int_{\Omega'}\psi(|u|)\,\mathrm{d}x \le 1$. \eop

\begin{rem}
 By inspection of the proof, the above compactness and existence result also holds for more general energies  in $GSBD^2$ of the form
\[
\int_{\Omega'} f(x,e(u)(x)) \,\mathrm{d}x + \int_{J_u \setminus \Gamma} g(x, \nu)\mathrm{d}{\cal H}^1
\]
which are lower semicontinuous with respect to the convergence in measure. Here, it is crucial that the surface density $g$, while possibly depending on the material point and the orientation of the jump, is insensitive to the jump height. Likewise, the existence result stated in Section \ref{sec: general} may  be generalized in this direction. 

We also mention that, in the same spirit, a derivation of an existence result in the realm of finite elasticity (see \cite{DFT}) without a-priori bounds on the deformations  or applied body forces is possible. We  defer  a  more  thorough  analysis  of  these  issues to  a  subsequent work. 
\end{rem}

We later will use property (ii) to derive compactness in $GSBD^2$ of the minimizers of our incremental problems. Concerning the stability of minimizers with respect to converging sequences of boundary data  we have the following corollary being a consequence of the jump transfer lemma. As before $Q$ is a strictly positive quadratic form on $\R^{2\times 2}_{\rm sym}$.

\begin{corollary}\label{cor: stability}
Let $\Omega \subset \Omega' \subset \R^2$ open, bounded with Lipschitz boundary such that \eqref{eq: omega'} holds. Let $\Gamma \subset \R^2$ be a measurable set with ${\cal H}^1(\Gamma) < \infty$, let $(u_n)_n,u \in GSBD^2(\Omega')$ and $u_n = w_n$ in  $\Omega' \setminus \overline{\Omega}$ for $(w_n)_n \subset H^1(\Omega';\R^2)$ such that $u_n \to u$ in measure, $e(u_n) \rightharpoonup e(u)$ weakly in $L^2(\Omega';\R^{2\times 2}_{\mathrm{sym}})$. If $u_n$ minimize 
$$\int_\Omega Q(e(v))\,\mathrm{d}x + {\cal H}^1(J_v \setminus (J_{u_n} \cup \Gamma))$$
among all functions with the same Dirichlet data, then $u$ minimizes
$$\int_\Omega Q(e(v))\,\mathrm{d}x + {\cal H}^1(J_v \setminus (J_{u} \cup \Gamma))$$
among all functions $v$ such that $v=u$ on $\Omega' \setminus \overline{\Omega}$. If furthermore $(w_n)_n$ is a constant sequence, we have $e(u_n) \to e(u)$ strongly in $L^2(\Omega';\R^{2\times 2}_{\mathrm{sym}})$. 
\end{corollary} 

The proof is omitted  as it is completely  analogous to Corollary 2.10 in \cite{Francfort-Larsen:2003} provided one substitutes the Dirichlet energy with the linearized elastic energy and the gradient by the symmetrized gradient.

\section{Proof of the existence result}\label{se: exist result}
Equipped with the theoretical results in the previous section, we can obtain the announced existence result Theorem \ref{th: main} by passing to the limit in the usual scheme of time-incremental minimization. The discussion in this section will closely follow the analogous one in \cite[Section 3]{Francfort-Larsen:2003}  and therefore not all the proofs will be detailed. For the reader's convenience we will only focus on some points where our $GSBD^2$ setting involves some modifications of the arguments developed there.
Through all this section we will write ${\cal H}^1(\Gamma)$ in place of ${\cal H}^1(\Gamma\cap \Omega')$, since all the cracks we consider in the proof will have by construction no intersection with $\partial \Omega \setminus \partial_D\Omega$.

We fix a time interval $[0,T]$ and consider a countable dense subset $I_\infty$ thereof. We can assume that $0$ and $T$ belong to $I_\infty$. For each $n\in \N$ we choose a subset $I_n:=\{0=t^n_0<t^n_1<\dots<t^n_n=T\}$ such that  $(I_n)_n$ form an increasing sequence of nested sets whose union is $I_\infty$. Setting $\Delta_n:=\displaystyle\sup_{1\le k \le n}(t^n_{k}-t^n_{k-1})$, we have that $\Delta_n \to 0$ when $n \to +\infty$.
As discussed in Section \ref{sec: general}, we consider a boundary datum $g\in W^{1,1}([0,T]; H^1(\R^2;\R^2))$   and the corresponding left-continuous piecewise constant interpolation
$$
g^n(t):=g(t^n_k) \hbox{ for all }t\in [t^n_k, t^n_{k+1})
$$
which satisfies $g(t)=g^n(t)$ for all $t \in I_\infty$, when $n$ is large enough. Moreover, $g^n(t)\to g(t)$ strongly in  $H^1$ for all $t \in [0,T]$.
We set $u^n(0)=u(0)$, the given initial datum, while for all $k=1,\dots,n$ we recursively define $u^n_k$ as a minimizer  of the problem
\begin{align}\label{eq: minprob}
\int_\Omega Q(e(v)) \,\mathrm{d}x + {\cal H}^1\left(J_v \setminus \bigcup_{0\le j \le k-1}J_{u^n_j}\right)
\end{align}
among the functions $v \in GSBD^2(\Omega')$ satisfying $v=g(t^n_k)$ in $\Omega'\setminus \overline{\Omega}$. The existence of such a minimizer follows from Theorem \ref{th: existence}. We then construct left-continuous piecewise constant interpolation
$$
u^n(t):=u^n_k \hbox{ for all }t\in [t^n_k, t^n_{k+1})\,.
$$
The following a-priori estimates on the interpolations can be then derived combining similar arguments as those developed in \cite{DM-Toa} and \cite{Francfort-Larsen:2003} with the additional property (ii) of Theorem \ref{th: existence}.

\begin{lemma}\label{lemma: apriori bound}
There exists an increasing concave function $\psi:[0,\infty) \to [0,\infty)$, satisfying \eqref{coerc}, which only depends on $\Omega, \Omega'$ and $\sup_{t\in [0,T]}\|g(t)\|_{H^1}$, such that the interpolations $u^n(t)$ satisfy
\begin{equation}\label{eq: a-priori bounds}
\int_{\Omega'}\psi(|u^n(t)|)\,\mathrm{d}x+ \Vert e(u^n(t))\Vert_{L^2(\Omega';\R^{2\times 2}_{\mathrm{sym}})} + {\cal H}^1\left(\bigcup_{\tau \in I_\infty\,,\,\tau \le t}J_{u^n(\tau)}\right) \le M
\end{equation}
for a constant $M$ independent of $t \in [0,T]$. Furthermore, setting $\sigma^n(t):=\C e(u^n(t))$ with $\C$ as in \eqref{eq: elasticity-tensor}, it exists a modulus of continuity $\omega$ such that the following energy inequality holds at every $t\in [0,T]$:
\begin{align}\label{eq: en-ineq}
 \nonumber \int_{\Omega}&Q(e(u^n(t))\,\mathrm{d}x+ {\cal H}^1\left(\bigcup_{\tau \in I_\infty\,,\,\tau \le t}J_{u^n(\tau)}\right)  \\ &\le
\int_{\Omega}Q(e(u(0)))\,\mathrm{d}x+ {\cal H}^1\left(J_{u(0)}\right) +\int_0^t \langle \sigma^n(s), e(\dot g(s))\rangle\,\mathrm{d}s + \omega(\Delta_n)\,.
\end{align}
\end{lemma}

\begin{proof}
The bound on $\Vert e(u^n(t))\Vert_{L^2(\Omega',\R^{2 \times 2}_{\rm sym})}$ is simply obtained by comparing the minimizer $u^n(t)$ with the admissible competitor $g^n(t)$, while the existence of a $\psi$ as in \eqref{eq: a-priori bounds} follows from (ii) in Theorem \ref{th: existence} and the assumptions on $g$. Fix now $t\in [0,T]$, and for fixed $n$, let $k$ be such that $t \in [t^n_k, t^n_{k+1})$. By construction, since $I_n \subset I_{\infty}$, one has
$$
\displaystyle\bigcup_{\tau \in I_\infty, \tau \le t}J_{u^n(\tau)}=\bigcup_{j=0}^k J_{u^n_j}.
$$
 Testing for every $1\le j \le k$ the minimality of $u^n(t^n_{j})$ with the admissible competitor $u^n(t^n_{j-1}) + g(t^n_j) - g(t^n_{j-1})$, summing up all steps until step $k$ and using the above equality, we obtain \eqref{eq: en-ineq} (for the details, use the same arguments leading to \cite[(3.4)]{Francfort-Larsen:2003}, upon replacing the Dirichlet energy with the linearized elastic energy.) Once \eqref{eq: en-ineq} is proved, the  uniform a-priori  bound on $ {\cal H}^1\left(\bigcup_{\tau \in I_\infty, \tau \le t}J_{u^n(\tau)}\right)$ simply follows by the Cauchy-Schwarz inequality and the already proven bound on $\sigma^n(t)$. 
\end{proof}

The following lower semicontinuity result will be needed in order to pass to the limit in the previous bounds. We do not report the proof, which is {\it verbatim} the same as in \cite[Lemma 3.1]{Francfort-Larsen:2003}, provided one uses the $GSBD$ compactness and lower semicontinuity theorem in place of the one in $SBV$.

\begin{lemma}\label{lemma: semicont-crack}
Let $A \subset \R^2$ be open, bounded. For all $\ell \in \N$, let $(v^n_\ell)_n$ be a sequence of functions in $GSBD^2(A)$ satisfying the assumptions of Theorem \ref{th: GSBD comp}, and let $v_\ell\in GSBD^2(A)$ be such that $v^n_\ell \to  v_\ell $ in measure  when $n\to +\infty$. Then
$$
{\cal H}^1\left(\bigcup_{\ell=0}^{+\infty}J_{v_\ell}\right)\le \liminf_{n\to +\infty}{\cal H}^1\left(\bigcup_{\ell=0}^{+\infty}J_{v^n_\ell}\right)\,.
$$
\end{lemma}

Using the bounds in \eqref{eq: a-priori bounds}, we will initially define $u(t)$ only for $t\in I_\infty$. This will already allow us to define a crack set $\Gamma(t)$ for all $t\in [0,T]$ with $J_{u(t)} \subset \Gamma(t)$ for $t \in I_\infty$. The function $u(t)$ will be later extended to all $t$ in a way that the inclusion $J_{u(t)}\subset \Gamma(t)$ still holds.

\begin{theorem}\label{th: dense subset}
There exists a (not relabeled) subsequence $(u^n(t))_n$ independently of $t\in I_\infty$ and a function $u\colon I_\infty\to GSBD^2(\Omega')$ such that $u^n(t) \to u(t)$ in measure for all $t\in I_\infty$ and, setting 
\begin{equation}\label{eq: crackset}
\Gamma(t):=\bigcup_{\tau \in I_\infty\,,\, \tau \le t}J_{u(\tau)} \hbox{ for all }t\in [0,T]\,,
\end{equation}
the following properties are satisfied:
\begin{align}\label{eq: limitproperties1}
\begin{split}
(i)&  \ \ u(t)=g(t) \ \ \text{ in } \Omega'\setminus \overline{\Omega} \hbox{ for all } t \in I_\infty, \\
(ii)& \ \ e(u^n(t)) \to e(u(t)) \ \ \text{ strongly in } L^2(\Omega', \R^{2\times 2}_{\rm sym})\hbox{ for all } t \in I_\infty, \\
(iii)& \ \ {\cal H}^1(\Gamma(t)) \le \liminf_{n \to \infty} {\cal H}^1\left(\bigcup_{\tau \in I_\infty\,,\,\tau \le t}J_{u^n(\tau)}\right)\hbox{ for all }t\in [0,T]\,.
\end{split}
\end{align}
Furthermore, for all $t\in I_\infty$, $u(t)$ minimizes
\begin{align}\label{eq: minthroughjtl}
\int_\Omega Q(e(v))\,\mathrm{d}x + {\cal H}^1(J_v \setminus \Gamma(t))
\end{align}
among all functions $v$ such that $v=g(t)$ on $\Omega' \setminus \overline{\Omega}$.
\end{theorem}

\begin{proof}
By \eqref{eq: a-priori bounds}  the sequence $(u^n(t))_n$  satisfies the assumptions of Theorem \ref{th: GSBD comp}  for every $t\in I_\infty$. With this, up to extracting a diagonal sequence, there exists $u\colon I_\infty\to GSBD^2(\Omega')$ such that $u^n(t) \to u(t)$ in measure and $e(u^n(t)) \wto e(u(t))$ weakly in $L^2(\Omega', \R^{2\times 2}_{\rm sym})$ for all $t \in I_\infty$. Since $u^n(t)=g^n(t)$ in $\Omega'\setminus \overline{\Omega}$ and $g^n(t)=g(t)$ for $n$ large enough, \eqref{eq: limitproperties1}(i) follows. At the expense of a numbering of $I_\infty$, \eqref{eq: limitproperties1}(iii) follows from Lemma \ref{lemma: semicont-crack}.

From the definition of $u^n(t)$ and $g^n(t)$ (cf. \eqref{eq: minprob}), for all $t\in [0,T]$ we have that $u^n(t)$ is minimizing
\begin{equation}\label{eq: minimality}
\int_\Omega Q(e(v))\,\mathrm{d}x + {\cal H}^1\left(J_v \setminus \bigcup_{\tau \in I_n\,,\, \tau \le t}J_{u^n(\tau)}\right)
\end{equation}
among the functions $v \in GSBD^2(\Omega')$ satisfying $v=g^n(t)$ in $\Omega'\setminus \overline{\Omega}$. {\it A fortiori}, we deduce that $u^n(t)$ is a minimizer with respect to its own jump set, that is with $J_{u^n(t)}$ in place of  $\bigcup_{\substack{\tau \in I_n\\ \tau \le t}}J_{u^n(\tau)}$ in the above problem.   If additionally $t\in I_\infty$,  we can choose $n$ so large that $t\in I_n \cap I_\infty$, and thus $g^n(t)=g(t)$. With this, Corollary \ref{cor: stability} gives \eqref{eq: limitproperties1}(ii).

We now fix $\delta>0$ and $t\in I_\infty$. Since ${\cal H}^1(\Gamma(t))$ is finite, we can find $\ell \in \N$ so that $t \in I_\ell$ and the subset $\Gamma_\ell(t)$ of $\Gamma(t)$ defined by 
$$
\Gamma_\ell(t)=\bigcup_{\tau \in I_\ell\,,\, \tau \le t}J_{u(\tau)}
$$
satisfies ${\cal H}^1(\Gamma(t)\setminus\Gamma_\ell(t))<\delta$. For all $n\ge \ell$, we similarly define $\Gamma^n_\ell(t)$ with $u^n(\tau)$ in place of $u(\tau)$. Notice that  $J_{u(t)}\subset \Gamma_\ell(t)$ and $J_{u^n(t)}\subset \Gamma^n_\ell(t)$ since $t\in I_\ell$.  With this and using \eqref{eq: minimality} we have that $u^n(t)$ is minimizing $\int_\Omega Q(e(v)) \,\mathrm{d}x + {\cal H}^1\left(J_v \setminus \Gamma^n_\ell(t)\right) $ among the functions $v \in GSBD^2(\Omega')$ which satisfy $v=g(t)$ in $\Omega'\setminus \overline{\Omega}$.

We observe that by Lemma \ref{lemma: apriori bound} the sequences $(u^n(\tau))_n$ with $\tau \in I_\ell$, $\tau \le t$, and the corresponding limiting functions $u(\tau)$ defined above satisfy \eqref{eq: assu}. Consequently,  for any $v$ with  $v=g(t)$ in $\Omega'\setminus \overline{\Omega}$ we can apply Theorem \ref{th: JTransf}  to $\phi=v-u(t)$ and   to the finite unions of jump sets $\Gamma^n_\ell(t)$ and $\Gamma_\ell(t)$. Therefore, we get the existence of a sequence $(\phi_n)_n$ such that $\phi_n= v- u(t)=0$ in $\Omega'\setminus \overline{\Omega}$ satisfying, by \eqref{eq: conv} and \eqref{eq: limitproperties1}(ii),
\begin{align}\label{eq: phi_n}
\Vert e(u^n(t)+\phi_n)-e(v)\Vert_{L^2(\Omega',\R^{2 \times 2}_{\rm sym})} \to 0,\, \quad
\limsup_{n\to+\infty}{\cal H}^1(J_{\phi_n}\setminus\Gamma^n_\ell(t))\le{\cal H}^1(J_{\phi}\setminus\Gamma_\ell(t))\,
\end{align}
as $n \to + \infty$. Furthermore, since $t\in I_\infty$, when $n$ is so big that $g^n(t)=g(t)$ in $\Omega'\setminus \overline{\Omega}$ we have that $u^n(t)+\phi_n$ $=g(t)$ in $\Omega'\setminus \overline{\Omega}$. The minimality of $u^n(t)$,  \eqref{eq: limitproperties1}(ii), and \eqref{eq: phi_n} then imply that
\begin{align*}
&\int_\Omega Q(e(u(t))) \,\mathrm{d}x=\lim_{n+\infty}\int_\Omega Q(e(u^n(t))) \,\mathrm{d}x \\
&\le \limsup_{n\to+\infty}\int_\Omega Q(e(u^n(t)+\phi_n)) \,\mathrm{d}x + {\cal H}^1\left(J_{u^n(t)+\phi_n} \setminus \Gamma^n_\ell(t)\right)\\
&\le \int_\Omega Q(e(v)) \,\mathrm{d}x + {\cal H}^1\left(J_v \setminus \Gamma_\ell(t)\right)\le\int_\Omega Q(e(v)) \,\mathrm{d}x + {\cal H}^1\left(J_v \setminus \Gamma(t)\right)+\delta\,,
\end{align*}
where we in the third step we used that  $J_{u(t)}\subset \Gamma_\ell(t)$ and $J_{u^n(t)}\subset \Gamma^n_\ell(t)$. This concludes the proof of \eqref{eq: minthroughjtl} since $\delta$ is arbitrary.
\end{proof}

\begin{rem}
Let $t\notin I_\infty$ and let $w\in GSBD^2(\Omega')$ be such that $(u^n(t))_n$ has a subsequence, possibly depending on $t$, which converges to $w$ in the sense of \eqref{eq: convergence sense}. Fix $\delta>0$ and $\Gamma_\ell(t)$ and $\Gamma^n_\ell(t)$ as in the previous proof, without the request $t\in I_\ell$. We can apply Theorem \ref{th: JTransf} for the  finite number of sequences  $(u^n(t))_n$ and $(u^n(\tau))_n$ with $\tau \in I_\ell$, $\tau \le t$, and thus for any  $v$ with $v=g(t)$ in $\Omega'\setminus \overline{\Omega}$,  we can apply  \eqref{eq: conv}  to $\phi=v$ to obtain a corresponding sequence $(\phi_n)_n$.  It follows now from \eqref{eq: minimality} that (with $v_n := \phi_n+g^n(t)-g(t)$)
$$
\int_\Omega Q(e(u^n(t)))\,\mathrm{d}x \le \int_\Omega Q(e(v_n))\,\mathrm{d}x+ {\cal H}^1 (J_{v_n} \setminus \left(\Gamma^n_\ell(t)\cup J_{u^n(t)}\right))\,.
$$
By \eqref{eq:  convergence sense}(ii),  the strong convergence of $g^n(t)$ to $g(t)$ in $H^1$, \eqref{eq: conv} and the arbitrariness of $\delta$ we deduce the minimality property
\begin{align}\label{eq: minimality2}
\int_\Omega Q(e(w))\,\mathrm{d}x \le \int_\Omega Q(e(v))\,\mathrm{d}x+ {\cal H}^1 (J_{v} \setminus \left(\Gamma(t)\cup J_{w}\right))\,.
\end{align}
For $v=w$  one also gets  $\displaystyle\lim_{n\to +\infty}\int_\Omega Q(e(u^n(t)))\,\mathrm{d}x = \int_\Omega Q(e(w))\,\mathrm{d}x$, which implies
\begin{equation}\label{eq: strongL2}
\Vert e(u^n(t))-e(w)\Vert_{L^2(\Omega',\R^{2 \times 2}_{\rm sym})} \to 0
\end{equation}
by the strict convexity of $Q$.
\end{rem}

In the next theorem we extend $u$ from $I_\infty$ to  a function defined on all of $[0,T]$. We prove that this extension satisfies the inclusion $J_{u(t)}\subset \Gamma(t)$ for all $t\in [0,T]$ (notice that, at this stage of the proof, the crack set $\Gamma(t)$ is already defined on the whole interval $[0,T]$), the global minimality condition, as well as the ``$\le$''-inequality in the energy balance of Theorem \ref{th: main}. The proof follows very closely in the footsteps of \cite[Lemma 3.8]{Francfort-Larsen:2003}:  A sketch is reported for the reader's convenience.

\begin{theorem}\label{th: extension}
There exists a function $u\colon[0,T] \to GSBD^2(\Omega')$ with $u(t)=g(t)$ in $\Omega'\setminus \overline{\Omega}$ and an $\mathcal H^1$-rectifiable crack $\Gamma(t) \subset \overline{\Omega}$, nondecreasing in $t$, such that $J_{u(t)}\subset \Gamma(t)$ up to an $\mathcal{H}^1$-negligible set for all $t\in [0,T]$ and:
\begin{itemize}
 \item (global stability) for all $t\in [0,T]$, $u(t)$ minimizes
$$\int_\Omega Q(e(v)) \,\mathrm{d}x + {\cal H}^1\left(J_v \setminus \Gamma(t)\right) $$
among the functions $v \in GSBD^2(\Omega')$ which satisfy $v=g(t)$ in $\Omega'\setminus \overline{\Omega}$. 

\item (energy inequality) defining the stress $\sigma(t)$ and the total energy $\mathcal E(t)$ as in Theorem \ref{th: main}, it holds
$$
\mathcal E(t)\le \mathcal E(0)+\int_0^t \langle \sigma(s), e(\dot g(s)) \rangle\,\mathrm{d}s\,.
$$
\end{itemize}
\end{theorem}

\begin{proof}
 We consider $u\colon I_\infty \to GSBD^2(\Omega')$ as in Theorem \ref{th: dense subset}. Accordingly, we define $\Gamma(t)$ as in \eqref{eq: crackset} for all $t\in [0,T]$. Thus, we simply have to define $u$ when $t \notin I_\infty$.  We fix $t \notin I_\infty$ and an increasing sequence $(t_k)_k \subset I_\infty$ converging to $t$. Notice that for the interpolants $u^n(t)$ the inequality \eqref{eq: a-priori bounds} holds with a constant $M$ and a function $\psi$ which are not depending on $k$. Since, for all $k$, $u^n(t_k) \to u(t_k)$ in measure when $n \to +\infty$  and thus also $u^n(t_k) \to u(t_k)$ a.e. for a not relabeled subsequence, by Fatou's lemma and \eqref{eq: convergence sense}, also the sequence $(u(t_k))_k$  satisfies \eqref{eq: a-priori bounds}. Then, it exists a limit point $u(t)\in GSBD^2(\Omega')$ with $u(t_k) \to  u(t) $  in measure and $e(u(t_k))\wto e(u(t))$ weakly in $L^2$ as $k \to \infty$. It is obvious that, $u(t)=g(t)$ in $\Omega'\setminus \overline{\Omega}$ while an application 
of \eqref{eq: 
minthroughjtl} together with the arguments leading to \cite[(3.24)]{Francfort-Larsen:2003}, again simply using $GSBD$ in place of $SBV$ compactness, shows that the inclusion $J_{u(t)}\subset \Gamma(t)$ holds up to an $\mathcal{H}^1$-negligible set.

We now prove the global stability property. Notice that for all $k$ one has by definition $\Gamma(t_k) \subset \Gamma(t)$ and, since the sequence of cracks $\Gamma(t_k)$ is nondecreasing, it holds that ${\cal H}^1(\Gamma(t)\setminus\Gamma(t_k))\to 0$ when $k\to +\infty$. For each $v\in GSBD^2(\Omega')$ with $v=g(t)$ in $\Omega'\setminus \overline{\Omega}$, the sequence $v_k=v+g(t_k)-g(t)$ has the same jump set as $v$ and clearly satisfies $e(v_k)\to e(v)$ in $L^2(\Omega',\R^{2\times 2}_{\rm sym})$. By Theorem \ref{th: dense subset} we have
$$
\int_\Omega Q(e(u(t_k)))\,\mathrm{d}x \le \int_\Omega Q(e(v_k))\,\mathrm{d}x+ {\cal H}^1 (J_{v} \setminus  \Gamma(t_k))\,.
$$
Taking the limit we get the global stability because of the inclusion $J_{u(t)}\subset \Gamma(t)$. We also get, for $v=u(t)$ in the above argument, that 
$$\lim_{n\to +\infty}\int_\Omega Q(e(u^n(t)))\,\mathrm{d}x = \int_\Omega Q(e(u(t)))\,\mathrm{d}x,$$
which implies the strong convergence of $e(u(t_k))$ to $e(u(t))$. Furthermore, due to the strict convexity of $Q$, the function $e(u(t))$ is uniquely determined by the global stability and the condition $J_{u(t)}\subset \Gamma(t)$. Thus, $e(u(t))$ is uniquely determined once $I_\infty$ is fixed. This implies the strong convergence of $e(u(t_k))$ to $e(u(t))$ on the whole sequence $(t_k)_k$ and not only along a subsequence, and that the mapping $t\to e(u(t))$ is strongly left continuous in $L^2$ at any $t\in [0,T]\setminus I_\infty$. 

The energy inequality immediately follows from \eqref{eq: en-ineq} and \eqref{eq: limitproperties1}(iii), once the following claim is proved:
$$
e(u^n(t))\to e(u(t))\hbox{ strongly in }L^2(\Omega',\R^{2\times 2}_{\rm sym})
$$
for a.e. $t\in [0,T]$. In fact,  one can then pass to the limit in \eqref{eq: en-ineq} also in the term associated to the work of the external loads. Because of \eqref{eq: limitproperties1}, it suffices to show the claim for $t\in [0,T]\setminus I_\infty$. Notice that because of \eqref{eq: a-priori bounds} the $L^2$-norm of the sequence $e(u^n(t))$ is bounded.  Furthermore, again  \eqref{eq: a-priori bounds}, together with Theorem \ref{th: GSBD comp} imply that any  weak accumulation point of  $e(u^n(t))$ must be of the form $e(w)$, where $w$ is a $GSBD^2$ function such that a subsequence, possibly depending on $t$, of $u^n(t)$ converges to $w$ in the sense of \eqref{eq: convergence sense}. Therefore, to prove the claim it suffices to show that $e(w)=e(u(t))$ for a.e.\ $t$, so that the limit is independent of the chosen subsequence and the strong convergence holds because of \eqref{eq: strongL2}. 

Let us a consider a weak accumulation point $w$. Notice that $u(t)$ is an admissible competitor for the problem \eqref{eq: minimality2}, which as shown above additionally satisfies  $J_{u(t)}\subset \Gamma(t)$.  Therefore, if we prove
\begin{equation}\label{eq: Q-ineq}
\int_\Omega Q(e(u(t)))\,\mathrm{d}x \le \int_\Omega Q(e(w))\,\mathrm{d}x
\end{equation}
for a.e.\ $t$, we will get $e(w)=e(u(t))$ as requested, otherwise, using the strict convexity of $Q$, we would contradict \eqref{eq: minimality2} with $v=\frac12\left(w+u(t)\right)$.  
Now, using \eqref{eq: strongL2} and the left continuity of $t\to e(u(t))$ at $t\notin I_\infty$, the inequality \eqref{eq: Q-ineq} follows from the minimality of $u^n(t)$ arguing exactly as in the proof of part (d) in \cite[Lemma 4.3]{Gia-Pons}, again upon substituting the Dirichlet with the linear elastic energy. We omit the details.
\end{proof}

We are finally in a position to give the proof of Theorem \ref{th: main}.

\begin{proof}[Proof of Theorem \ref{th: main}]
Defining $u(t)$ as in Theorems \ref{th: dense subset} and \ref{th: extension} and $\Gamma(t)$ as in \eqref{eq: crackset}, the only thing left to show is the ``$\ge$''-inequality in the energy balance. This follows from global stability by a well-known argument (see \cite[Lemma 4.6]{Gia-Pons})  that we sketch for the reader's convenience.
We first notice that the map $t\mapsto {\cal H}^1\left(\Gamma(t)\right)$ is bounded monotone increasing, so that it is continuous at each $t\in [0,T]\setminus \mathcal N$, where $\mathcal N$ has $0$-Lebesgue measure. At each $t\in [0,T]\setminus (I_\infty \cup \mathcal N)$ we already now that $e(u(\cdot))$ is left continuous with respect to to the $L^2$-norm. We can show that it is indeed continuous, arguing as follows. Fixing a decreasing sequence $t_k \to t$, any weak-$L^2$ accumulation point $e(w)$ of $e(u(t_k))$ satisfies, because of \eqref{eq: convergence sense}, the inclusion $\Gamma(t)\subset \Gamma(t_k)$, and the continuity of ${\cal H}^1\left(\Gamma(\cdot)\right)$ at time $t$, that
$$
{\cal H}^1\left(J_w\setminus \Gamma(t)\right)\le \liminf_{k\to +\infty}{\cal H}^1\left(J_{u(t_k)}\setminus \Gamma(t_k)\right)=0\,.
$$
Consequently, $J_w\subset \Gamma(t)$ up to an $\mathcal{H}^1$-negligible set. With this, testing the global stability of $u(t_k)$ with $v+g(t_k)-g(t)$ for any $v$ with $v=g(t)$ in $\Omega'\setminus \overline{\Omega}$ and arguing as in the proof of Theorem \ref{th: extension}, we obtain $e(w)=e(u(t))$ as well as the claimed strong convergence.  

Fix now $t\in [0,T]$. Setting  for every $k \in \N$ and every $i=0,\dots,k$, $s^i_k=\frac{i}{k}t$ and $u_k(s)=u(s^{i+1}_k)$ whenever $t\in (s^i_k, s^{i+1}_k]$, we have that $\Vert e(u_k(s))\Vert_{L^2(\Omega',\R^{2 \times 2}_{\rm sym})}$ is uniformly bounded because of the energy inequality (see Theorem \ref{th: extension}), and
\begin{equation}\label{eq: strong2}
\Vert e(u_k(s))-e(u(s))\Vert_{L^2(\Omega',\R^{2 \times 2}_{\rm sym})} \to 0\hbox{ for all }s\in [0,t]\setminus (I_\infty \cup \mathcal N)\,,
\end{equation}
that is a.e. in $[0,t]$. Testing the global stability of $u(s^i_k)$ with $u(s^{i+1}_k)-g(s^{i+1}_k)+g(s^i_k)$, summing up on $i$ and exploiting the absolute continuity of $t\mapsto g(t)$, one obtains
$$
\mathcal E(t) \ge \mathcal E(0)+\int_0^t \langle \sigma_k(s), e(\dot g(s))\rangle\,\mathrm{d}s + \eta_k\,,
$$
where $\sigma_k(s):=\C e(u_k(s))$ and $\eta_k$ is an infinitesimal remainder. The thesis now follows by dominated convergence and \eqref{eq: strong2}  when taking the limit  $k\to +\infty$. 
\end{proof}

\noindent \textbf{Acknowledgements} This work has been funded by the Vienna Science and Technology Fund (WWTF)
through Project MA14-009. Manuel Friedrich's research has been supported by the Alexander von Humboldt Stiftung. Francesco Solombrino's work is part of the project ``Quasistatic and Dynamic Evolution
Problems in Plasticity and Fracture" (P.I. Prof. G. Dal Maso), financed by ERC under Grant No. 290888. The authors wish to thank Gilles Francfort for helpful suggestions on the first version of this paper.


 \typeout{References}


\begin{thebibliography}{10}
\bibitem{Ambrosio:89}
{\sc L.~Ambrosio}. 
\newblock {\em A compactness theorem for a new class of functions of bounded variation}.
\newblock Boll.\ Un.\ Mat.\ Ital.\ B(7)
\newblock {\bf 3} (1989), 857--881. 


\bibitem{Ambrosio:90}
{\sc L.~Ambrosio}. 
\newblock {\em Existence theory for a new class of variational problems}.
\newblock Arch.\ Ration.\ Mech.\ Anal.\
\newblock {\bf 111} (1990), 291--322. 

\bibitem{ACD}
{\sc L.~Ambrosio, A.~Coscia, G.~Dal Maso}.
\newblock {\em Fine properties of functions with bounded deformation}.
\newblock Arch.\ Ration.\ Mech.\ Anal.\ 
\newblock{\bf 139} (1997), 201--238. 


\bibitem{AFP}
{\sc L.~Ambrosio, N.~Fusco, D.~Pallara}.
\newblock {\em Functions of bounded variation and free discontinuity problems}. 
\newblock Oxford University Press, Oxford 2000. 






\bibitem{baldi}
{\sc A.~Baldi, F.~Montefalcone}. 
\newblock {\em A note on the extension of BV functions in metric measure
spaces}.
\newblock J.\ Math.\ Anal.\ Appl.\
\newblock {\bf 340} (2008), 197--208. 


\bibitem{BCD}
{\sc G.~Bellettini, A.~Coscia, G.~Dal Maso}. 
\newblock {\em  Compactness and lower semicontinuity properties in SBD($\Omega$)}.
\newblock Mathematische\ Zeitschrift\ 
\newblock {\bf 228} (1998), 337--351. 

 


\bibitem{Buc-Var:2000}
{\sc D.~Bucur, N.~Varchon}.
\newblock {\em Boundary variation for a Neumann problem}.
\newblock Ann. Scuola Norm. Sup. Cl. Sci. 
\newblock {\bf 29}(4) (2000), 807--821.

\bibitem{mazja}
{\sc Yu.~D.~Burago, V.~G.~Maz'ja}. 
\newblock {\em Potential Theory and Function Theory for irregular regions.  Translated from Russian}.
\newblock Seminars\ in\ Mathematics, V.\ A.\ Steklov\ Math.\ Ins., Leningrad,
\newblock vol.\ 3, 1969. 


\bibitem{Chambolle:2003}
{\sc A.~Chambolle}.
\newblock {\em A density result in two-dimensional linearized elasticity, and
  applications}.
\newblock Arch.\ Ration.\ Mech.\ Anal.\ 
\newblock {\bf 167} (2003), 211--233.

\bibitem{Chambolle:2004}
{\sc A.~Chambolle}. 
\newblock {\em An approximation result for special functions with bounded deformation}.
\newblock J.\ Math.\ Pures\ Appl. 
\newblock {\bf 83} (2004), 929--954. 


\bibitem{Chamb-Dov:1997}
{\sc A.~Chambolle, F.~Doveri}.
\newblock {\em Continuity of Neumann linear elliptic problems on varying two-dimensional bounded open sets}.
\newblock Comm.\ Partial\ Differential\ Equations\
\newblock {\bf 22} (1997), 811--840.

\bibitem{Conti-Focardi-Iurlano:15}
{\sc S.~Conti, M.~Focardi, F.~Iurlano}.
\newblock {\em Integral representation for functionals defined on $SBD^p$ in dimension two} 
\newblock Arch.\ Rat.\ Mech.\ Anal.\  
\newblock {\bf 223} (2017), 1337--1374.




\bibitem{DM}
{\sc G.~Dal Maso}. 
\newblock {\em Generalised functions of bounded deformation}.
\newblock J.\ Eur.\ Math.\ Soc.\ (JEMS)\
\newblock {\bf 15} (2013), 1943--1997. 

\bibitem{DFT}
{\sc G.~Dal Maso, G.~A.~Francfort, R.~Toader}. 
\newblock {\em Quasistatic crack growth in nonlinear elasticity}.
\newblock Arch.\ Ration.\ Mech.\ Anal.\ 
\newblock {\bf 176} (2005), 165--225. 


\bibitem{DM-Toa} 
{\sc G.~Dal Maso, R.~Toader}.
\newblock A model for the quasi-static growth of brittle fractures: existence and approximation results.
\newblock Arch.\ Ration.\ Mech.\ Anal.\  
\newblock {\bf 162} (2002), 101--135.


\bibitem{dg91}
{\sc E.~{D}e {G}iorgi}.
\newblock {Free-discontinuity problems in calculus of variations}.
\newblock In R.~Dautray, editor, {\em Frontiers in pure and applied
  mathematics, a collection of papers dedicated to J.-L. Lions on the occasion
  of his $60^{th}$ birthday}, pages 55--62. North Holland, 1991.
  
\bibitem{DeGiorgi-Carriero-Leaci:1989}
{\sc E.~De Giorgi, M.~Carriero, A.~Leaci}. 
\newblock {\em Existence theorem for a minimum problem with free discontinuity set}.
\newblock Arch.\ Ration.\ Mech.\ Anal.\ 
\newblock {\bf 108} (1989), 195--218. 


\bibitem{fonseca}
{\sc I.~Fonseca, G.~Leoni}.
\newblock {\em Modern methods in the calculus of variations: $L^p$ spaces}. 
\newblock Springer, New York 2007. 



\bibitem{Francfort-Larsen:2003}
{\sc G.~A.~Francfort, C.~J.~Larsen}. 
\newblock {\em Existence and convergence for quasi-static evolution in brittle fracture}.
\newblock Comm.\ Pure Appl.\ Math.\ 
\newblock {\bf 56} (2003), 1465--1500. 


\bibitem{Friedrich:15-2}
{\sc M.~Friedrich}.
\newblock {\em A derivation of linearized Griffith energies from nonlinear models}. 
\newblock Arch.\ Rat.\ Mech.\ Anal.\  
\newblock {\bf 225} (2017), 425--467.

\bibitem{Friedrich:15-3}
{\sc M.~Friedrich}.
\newblock {\em A Korn-type inequality in  SBD for functions with small jump sets}. 
\newblock  Math.\ Models Methods Appl.\ Sci.\
 \newblock {\bf 27} (2017), 2461--2484.

\bibitem{Friedrich:15-4}
{\sc M.~Friedrich}.
\newblock {\em A piecewise Korn inequality in SBD and applications to embedding and density results}. 
\newblock {preprint {\it ArXiv: 1604.08416}},  (2016). SIAM J.\ Math.\ Anal.\ , to appear.

 

\bibitem{frma98}
{\sc G.~A.~Francfort, J.-J.~Marigo}.
\newblock {\em Revisiting brittle fracture as an energy minimization problem.}
\newblock J.\ Mech.\ Phys.\ Solids\
\newblock {\bf 46} (1998), 1319--1342. 

\bibitem{Gia-Pons}
{\sc A.~Giacomini, M.~Ponsiglione}.
\newblock {\em Discontinuous finite element approximation of
quasi-static growth of brittle fractures}.
\newblock {Numer.\ Funct.\ Anal.\ Optim.} 
\newblock {\bf 24} (2003), 813--850. 

\bibitem{Golab:28}
{\sc S.~Go{\l}ab}.
\newblock {\em Sur quelques points de la th\'eorie de la longueur}.
\newblock Annales\ de\ la\ Soci\'et\'e\ Polonaise\ de\ Math\'ematiques\ 
\newblock {\bf 7} (1928), 227--241. 

\bibitem{Griffith:20}
{\sc A.~Griffith}. 
\newblock {\em The phenomena of rupture and flow in solids}.
\newblock Phil.\ Trans.\ Roy.\ Soc.\ London\
\newblock {\bf 221-A} (1920), 163--198. 

\bibitem{Iur}
{\sc F.~Iurlano}. 
\newblock {\em A density result for GSBD and its application to the approximation of brittle fracture energies}.
\newblock Calc.\ Var.\ PDE 
\newblock {\bf 51} (2014), 315--342. 

\bibitem{maggi}
{\sc F.~Maggi}.
\newblock {\em Sets of Finite Perimeter and Geometric Variational Problems: An Introduction to Geometric Measure Theory}. 
\newblock Cambridge Studies in Advanced Mathematics No. 135. Cambridge University Press, Cambridge 2012. 

\bibitem{Mum-Sha:1989}
{\sc D.~Mumford, J.~Shah}.
\newblock {\em Optimal approximation by piecewise smooth functions and associated
  variational problems}.
\newblock Commun.\ Pure Appl.\ Math.\ 
\newblock {\bf 42} (1989), 577--684.


\bibitem{Murat}
{\sc F.~Murat}.
\newblock {\em The Neumann sieve}. 
\newblock Nonlinear variational problems (Isola d'Elba, 1983).
\newblock Research Notes in Mathematics, 127. Pitman, Boston 1985. 




\bibitem{Nit}
{\sc J.~A.~Nitsche}. 
\newblock {\em On {K}orn's second inequality}.
\newblock RAIRO Anal. Num\'er
\newblock {\bf 15} (1981), 237--248. 




\bibitem{Tem}
\textsc{R.~Temam}.
\textit{Mathematical problems in plasticity}.
Gauthier-Villars, Paris 1985.



\end{thebibliography}
\end{document}